\documentclass[final,leqno,onefignum,onetabnum]{siamltex1213}

\usepackage{showkeys}
\usepackage{amssymb,amsmath}
\usepackage{algorithm}
\usepackage{algorithmic}

\newtheorem{remark}{\bf  Remark}[section]
\newtheorem{assumption}{\bf  Assumption}

\newcommand{\bbE}{{\mathbb{E}}}
\newcommand{\bbR}{{\mathbb{R}}}
\newcommand{\bbN}{{\mathbb{N}}}
\newcommand{\bbJ}{{\mathbb{J}}}
\newcommand{\bbP}{{\mathbb{P}}}

\newcommand{\cB}{\mathcal{B}}

\newcommand{\cF}{\mathcal{F}}

\newcommand{\cN}{\mathcal{N}}

\newcommand{\cR}{\mathcal{R}}

\newcommand{\cQ}{\mathcal{Q}}
\newcommand{\cS}{\mathcal{S}}

\newcommand{\argmax}{\operatornamewithlimits{argmax}}
\newcommand{\argmin}{\operatornamewithlimits{argmin}}

\newcommand{\bsnull}{{\boldsymbol 0}}

\newcommand{\bsgamma}{{\boldsymbol{\gamma}}}

\newcommand{\bstau}{\boldsymbol{\tau}}

\newcommand{\bsnu}{{\boldsymbol{\nu}}}
\newcommand{\bsmu}{{\boldsymbol{\mu}}}

\newcommand{\bsl}{\boldsymbol{l}}

\newcommand{\bse}{{\boldsymbol{e}}}

\newcommand{\bsk}{{\boldsymbol{k}}}

\newcommand{\bsy}{{\boldsymbol{y}}}

\newcommand{\bszero}{{\boldsymbol{0}}}

\newcommand{\beq}{\begin{equation}}
\newcommand{\eeq}{\end{equation}}
\newcommand{\ba}{\begin{array}}
\newcommand{\ea}{\end{array}}

\title{Sparse Quadrature for High-Dimensional Integration with Gaussian Measure \thanks{This work is supported by DARPA's EQUiPS program under contract number W911NF-15-2-0121.}} 

\author{ Peng Chen \thanks{Institute for Computational Engineering \& Sciences, The University of Texas at Austin, Stop C0200, Austin, TX 78712 (\email{peng@ices.utexas.edu}).} }
\begin{document}
\maketitle
\slugger{sisc}{xxxx}{xx}{x}{x--x}

\begin{abstract}
In this work we analyze the dimension-independent convergence property of an abstract sparse quadrature scheme for numerical integration of functions of high-dimensional parameters with Gaussian measure.  
Under certain assumptions of the exactness and the boundedness of univariate quadrature rules as well as 
the regularity of the parametric functions with respect to the  parameters, we obtain the convergence rate $O(N^{-s})$, where $N$ is the number of indices, and $s$ is independent of the number of the parameter dimensions. Moreover, we propose both an a-priori and an a-posteriori schemes for the construction of a practical sparse quadrature rule and perform numerical experiments to demonstrate their dimension-independent convergence rates.

\end{abstract}



\pagestyle{myheadings}
\thispagestyle{plain}
\markboth{Sparse Quadrature for High-dimensional Integration with Gaussian Measure}{P. Chen}

\section{Introduction}
In the mathematical modelling of a physical system, uncertainties may arise from various sources of the system input, such as material properties, initial/boundary conditions, and computational geometries. These uncertainies lead to the discrepancy between experimental/observational data and the output of mathematical models in many computational science and engineering fields. How to propagate the uncertainties through the mathematical models and how to calibrate them with given data are known as uncertainty quantification (UQ) problems \cite{ghanem2003stochastic, le2010introduction, xiu2010numerical, smith2013uncertainty}. One of the central tasks of UQ is to compute the integral of some quantity of interest related to the solution with respect to the probability law of the uncertain input. When the uncertain input are approximated by many or a countably infinite number of random variables or parameters, e.g., by Karhunen--Lo\`eve expansion \cite{Schwab2006Karhunen}, one faces high/infinite-dimensional integration problems. 
Since the integral with respect to the parameters can not be computed analytically in general, numerical integration based on certain quadrature rules has to be employed. However, it is of great challenge to perform high-dimensional numerical integration as the computational complexity grows exponentially fast with respect to the number of the parameter dimensions for most deterministic quadratures, which is widely known as ``curse of dimensionality". On the other hand, probabilistic quadrature rules, in particular the Monte Carlo \cite{caflisch1998monte}, are best known to break the curse of dimensionality. However, the convergence of these quadrature rules are often very slow, e.g., the convergence rate of Monte Carlo quadrature is $O(M^{-1/2})$ with $M$ samples, even for functions smoothly depending on low-dimensional parameters. 

Recent years have seen a great development of a sparse quadrature -- numerical integration based on sparse grids  \cite{gerstner1998numerical, gerstner2003dimension, xiu2005high, bungartz2004sparse, nobile2008sparse, babuvska2010stochastic, Schillings2013, beck2014quasi, chen2015new} -- to efficiently deal with high-dimensional integration problems. The curse of dimensionality is shown to be alleviated and/or broken by adaptive allocation of the quadrature points in different dimensions by ample numerical evidence \cite{gerstner1998numerical, gerstner2003dimension, griebel2010dimension, Schillings2013, nobile2014convergence, chen2015sparse, chen2016sparse}, which is also observed for interpolation problems by the same or similar dimension-adaptive algorithms \cite{nobile2008anisotropic, ma2009adaptive, chkifa2014high, chkifa2015breaking}. The dimension-independent convergence rate of the sparse quadrature for infinite-dimensional integration with respect to uniformly distributed parameters was proved in \cite{Schillings2013, Schillings2014}, which is based on 
the dimension-independent convergence of Legendre/Tayor polynomial chaos approximation of stochastic problems in \cite{cohen2010convergence, cohen2011analytic, chkifa2015breaking}. Different approximation methods of the stochastic problems with (lognormal) Gaussian random parameters have been studied in \cite{li2007probabilistic, lin2009efficient, gittelson2010stochastic, schillings2011efficient, charrier2012strong, chen2013weightedrbm, ernst2014stochastic, graham2015quasi, kuo2015multilevel, nobile2016adaptive}. More recently, a dimension-independent convergence rate of the polynomial chaos (based on Hermite polynormials) approximation for an elliptic problem with lognormal coefficients is obtained in \cite{hoang2014n}, whose convergence rate is improved in \cite{bachmayr2017sparse}. A convergence result based on \cite{bachmayr2017sparse} is obtained in \cite{ernst2016convergence} for a sparse collocation method.

In this work, we show the dimension-independent convergence rate of an abstract sparse quadrature scheme for infinite-dimensional integration problems with i.i.d. standard Gaussian distributed parameters. The result holds under certain assumptions of the exactness and the boundedness of univariate quadrature rules, and certain regularity assumptions of the parametric functions with respect to the parameters. In particular, only weighted finitely many derivatives are required to exist as in \cite{bachmayr2017sparse}, compared to an analytic regularity requirement for the result with uniform distribution in \cite{Schillings2013}. 
Two examples are provided to illustrate the regularity assumptions, including an infinite-dimensional nonlinear parametric function, and an elliptic PDE with nonlinear parametric lognormal coefficients. The key of the proof relies on three results: 1). the exactness and the boundedness of the sparse quadrature in arbitrary number of dimensions; 2). the bound of the sparse quadrature error by a weighted sum of the Hermite coefficients; 3). the summability of a weighted sequence of the coefficients arising from the regularity assumptions of the parametric function. Based on the proof, we propose a-priori construction of the sparse quadrature, whose error is guaranteed to converge with a dimension-independent convergence rate with respect to the number of indices. We also present a goal-oriented a-posteriori construction of the sparse quadrature, which turns out to be more accurate for the test examples. Both the a-priori and the a-posteriori construction schemes are
built on several univariate quadrature rules, including the non-nested Gauss--Hermite quadrature rule \cite{gil2007numerical}, the nested transformed Gauss--Kronrod--Patterson (or Gauss--Patterson) quadrature rule \cite{gerstner1998numerical}, and the nested Genz--Keister quadrature rule \cite{genz1996fully}. We will investigate and compare the convergence properties of the construction schemes with different quadrature rules in high dimensions. Numerical experiments on the sparse quadrature for a nonlinear parametric function and an elliptic parametric PDE 
are performed to demonstrate the dimension-independent convergence rate, and to compare the a-priori and the a-posteriori construction schemes with different quadrature rules.

The rest of the paper is organized as follows. In Section \ref{sec:ASQ} we present the sparse quadrature.  Several univariate quadrature rules are introduced in hierarchical representation in Section \ref{subsec:UnivQuad}, followed by the presentation of tensorization of these rules in Section \ref{subsec:TensQuad} and of the sparse quadrature in Section \ref{subsec:ASQ}. Section \ref{sec:ConvAnal} is devoted to a convergence analysis of the sparse quadrature, with a dimension-independent convergence rate obtained in the main theorem in Section \ref{subsec:DimenIndeCon} and two examples shown to satisfy the regularity assumptions in Section \ref{subsec:Examples}. In Section \ref{sec:construction} we introduce an a-priori scheme (in Section \ref{sec:aprioricons}) and an a-posteriori scheme (in Section \ref{sec:aposteriori}) for the construction of the sparse quadrature. We present two sets of numerical experiments in Section \ref{sec:Numerics}, one is on the sparse quadrature for numerical integration of a infinite-dimensional parametric function in Section \ref{subsec:function} and the other for numerical integration of two quantities of interest related to the solution of an elliptic parametric PDE in Section \ref{subsec:PDE}. 
In the last Section \ref{sec:Conclusion} we conclude with some further research perspectives.

\section{Sparse quadrature with Gaussian measure}
In this section, we present a sparse quadrature for numerical  integration of a function of high/infinite-dimensional parameters with Gaussian measure. At first, we formulate a hierarchical representation of a univariate quadrature with three different quadrature rules. Then a tensor-product quadrature is constructed by tensorization of the univariate quadrature. The sparse quadrature is then defined by a sum of the tensorized univariate quadrature in an admissible index set. 

\label{sec:ASQ}
\subsection{Univariate quadrature}\label{subsec:UnivQuad} 
Let $f: \bbR \to \cS$ be a univariate function of a random variable with standard Gaussian (or normal) distribution $N(0,1)$, which takes values in some Banach space $\cS$. Let $I$ denote an \emph{integral operator} defined as
\beq
I(f) = \int_{\bbR} f(y) d\gamma(y),
\eeq
where $\gamma(y)$ is a Gaussian measure with the probability density function $\rho(y)$ given by
\beq
\rho(y) = \frac{1}{\sqrt{2\pi}} e^{-y^2/2}.
\eeq
We introduce a sequence of \emph{quadrature operators} $\{\cQ_l\}_{l\geq 0}$ indexed by \emph{level} $l \in \bbN$, defined as
\beq\label{eq:UnivQuad}
\cQ_l(f) = \sum_{k = 0}^{m_l-1} w_k^l f(y_k^l), \quad l \geq 0,
\eeq 
where $y_k^l \in \bbR$ and $w_k^l \in \bbR$, $k = 0, \dots, m_l-1$,  represent quadrature points and weights; $m_l$ is the number of the quadrature points at level $l$, which satisfies 
$
m_0 = 1 \text{ and } m_l < m_{l+1}.
$
We consider two classical choices of $m_l$ \cite{genz1996fully, klimke2006uncertainty, babuvska2010stochastic} -- adding one point or doubling the number of points from level $l$ to $l+1$, i.e., 
$
m_{l+1} = l + 1 \text{ or } m_l = 2^{l+1}-1.
$
Let $\{\triangle_l\}_{l \geq 0}$ denote a set of \emph{difference quadrature operators}, which are defined as 
\beq\label{eq:DiffOper}
\triangle_l = \cQ_l - \cQ_{l-1}, \quad l \geq 0\;,
\eeq
where we set $\cQ_{-1} = 0$ by convention, i.e., $\cQ_{-1} (f) = 0$. Then we obtain a hierarchical representation of $\cQ_l$ through a telescopic sum of $\triangle_i$, $i = 0, \dots, l$, i.e., 
\beq
\cQ_l = \sum_{i = 0}^l \triangle_i\;.
\eeq

As for the quadrature points and weights in \eqref{eq:UnivQuad} as well as the specific number of points in each level, 
we consider the following ones.

\begin{enumerate}
\item \textbf{Gauss--Hermite (GH) quadrature.} A Gauss quadrature is used for the approximation of the integral with the density $\rho$ as the weight function \cite{gil2007numerical}, 
where $y_0^0 = 0$ and $w_0^0 = 1$ for $l = 0$, and for $l \geq 1$, $y_k^l$, $k = 0, \dots, m_l-1$, are the roots of the orthonormal (with respect to $\rho$) Hermite polynomial $H_n$ for $n = m_l$, where
\beq\label{eq:Hermite}
H_{n}(y) = \frac{(-1)^{n}}{\sqrt{n!}} \frac{\rho^{(n)}(y)}{\rho(y)}, \quad n \geq 0 \;,  
\eeq
and the weights $w_k^l$, $k = 0, 1, \dots, m_l-1$, are given by 
\beq\label{eq:HermiteWeight}
w_k^l = \frac{ 1}{m_l^2(H_{m_l-1}(y_k^l))^2}\;.
\eeq
Note that this quadrature rule is provided for the weight function $\rho(y)$ instead of $e^{-y^2}$ in the classical formula \cite[\S 5.3]{gil2007numerical}. It is exact with $m_l$ points for polynomials of degree up to $2m_l - 1$, the maximum possible exactness. However, the quadrature points are not nested in the sense that $\{y^l_k\}$ are not included in $\{y^{l'}_k\}$ for $l' > l$ (except for $l = 0$ and $m_{l'}$ odd which share the point $y = 0$), so that we need to evaluate the function at all the quadrature points at each level $l$. 
As for the number of points $m_l$ at each level $l$, we consider $m_l = l+1$ (denoted as GH1) and $m_l  = 2^{l+1}-1$ (GH2).

\item \textbf{Transformed Gauss--Kronrod--Patteron (tGKP) quadrature.} In \cite{kronrod1965nodes}, Kronrod 
 presented a method to add $m+1$ points to a $m$-point Gauss--Legendre quadrature rule for integration with constant weight and showed its optimality in integrating polynomials with such nested construction. Patterson \cite{patterson1968optimum} extended this construction iteratively and obtained a nested quadrature rule with $m_l = 2^{l+1}-1$ points at level $l$ (denoted as GKP). 
Then for integration with more general weight, e.g., normal weight $\rho$ in our problem, we can make a change of variables, e.g., by the following map 
 \beq
x = F_\rho(y)\;,
\eeq
where $F_\rho$ is the cumulative distribution function given by $F_\rho(y) = \int_{-\infty}^y \rho(y)dy$,  so that $dx = \rho(y)dy$ and the integration with weight $\rho$ can be transformed as 
\beq
\int_\bbR f(y)\rho(y)dy = \int_{0}^1 f(F_\rho^{-1}(x)) dx \approx \sum_{k = 0}^{m_l-1} f(F_\rho^{-1}(x_k^l)) w_k^l\;. 
\eeq
where $F_\rho^{-1}$ is the inverse of $F_\rho$, $x^l_k$ and $w_k^l$ are the GKP points and weights at level $l$. This transformed GKP (tGKP) has been used, e.g., in \cite{gerstner1998numerical}.

\item \textbf{Genz--Keister (GK) quadrature}: In \cite{genz1996fully}, Genz and Keister extended the GKP construction for uniform distribution to that for normal distribution. However, the construction does not follow that of GKP since the quadrature points obtained by Kronrod's method in level $l = 2$ are not real valued, thus they can not be used as quadrature points. Instead, Genz and Keister showed that, among several extensions, $1, 2, 6, 10, 16$ points can be added, resulting in $m_l = 1, 3, 9, 19, 35$ points at level $l = 0, 1, 2, 3, 4$. Further extension to higher levels is limited by the construction error due to ill-conditioned matrix equations, see details in \cite{genz1996fully}. 
\end{enumerate}

\subsection{Tensor-product quadrature}
\label{subsec:TensQuad}
For a given function $f: Y \to \cS$, where $Y = \bbR^J$, $J \in \bbN$ for finite dimensions or $J = \infty$ for infinite dimensions, we consider the product measure space $(Y, \cB(Y), \bsgamma)$ as in \cite{bachmayr2017sparse} where $\cB(Y)$ is the $\Sigma$-algebra generated by the Borel cylinders and $\bsgamma$ is a tensorized Gaussian probability measure. The task is to compute the integral  
\beq\label{eq:MultiIntegral}
I(f) = \int_{Y} f(\bsy) d\bsgamma(\bsy)\;.
\eeq 
In order to approximate \eqref{eq:MultiIntegral}, we define a tensor-product quadrature as follows.
By $\cF$ we denote a multi-index set of indices $\bsnu = (\nu_1, \dots, \nu_J) $, which is defined as  
\beq
\cF = \{\bsnu \in \bbN^J: |\bsnu|_1 < \infty\},
\eeq
where  $|\bsnu|_1 =  \nu_1 + \cdots + \nu_J$. Note that each $\bsnu \in \cF$ is finitely supported and we denote its finite support set as 
\beq
\bbJ_\bsnu = \{j \in \bbN: \nu_j \neq 0\}.
\eeq
Given $\bsnu \in \cF$, we define a multivariate quadrature operator $\cQ_\bsnu$ as tensorization of the univariate quadrature operators on the tensor-product grids $G_\bsnu =\{\bsy^\bsnu_\bsk: k_{j} = 0, \dots, m_{\nu_{j}} - 1, j \in \bbJ_\bsnu\}$, i.e., 
\beq\label{eq:TensorQuad}
\cQ_\bsnu (f) = \bigotimes_{j \in \bbJ_\bsnu} \cQ_{\nu_j}(f) \equiv \sum_{k_{j_1}=0}^{m_{\nu_{j_1}}-1} \cdots \sum_{k_{j_d}=0}^{m_{\nu_{j_d}}-1} w^{\nu_{j_1}}_{k_{j_1}} \cdots w^{\nu_{j_d}}_{k_{j_d}} f\left(y^{\nu_{j_1}}_{k_{j_1}}, \dots, y^{\nu_{j_d}}_{k_{j_d}}\right)\;,
\eeq
where we suppose $\bbJ_\bsnu$ is explicitly given as $\bbJ_\bsnu = \{j_1, \dots, j_d\}$ for some $d\in \bbN$, and we set $y_j = 0$ for all $j \not \in \bbJ_\bsnu$ and omit their appearance in the arguments of $f$ by slight abuse of notation.  \emph{A full tensor-product quadrature} for approximation of \eqref{eq:MultiIntegral} is defined as $\cQ_{\bsnu}(f)$ for $\bsnu = \bsl$, i.e., $\nu_j = l$ for each $j = 1, \dots, J$ at given $l \in \bbN$. However, the total computational cost of $(m_l)^J$ function evaluations grows exponentially with respect to the dimension $J$, known as \emph{curse of dimensionality}, rendering this quadrature rule computationally prohibitive for large $J$, especially when evaluation of $f$ is expensive.

\subsection{Sparse quadrature}
\label{subsec:ASQ}
In order to alleviate the curse of dimensionality, we turn to a \emph{sparse quadrature}, which breaks the restriction of taking $\nu_j = l$ in each dimension and allows free choice of $\bsnu\in \cF$. For each $\bsnu \in \cF$ with support $\bbJ_{\bsnu}$ in $d$ dimensions, we define a multivariate difference quadrature operator as 
\beq\label{eq:TensorDiff}
\triangle_\bsnu (f) = \bigotimes_{j \in \bbJ_\bsnu} \triangle_{\nu_j} (f) \equiv \bigotimes_{j \in \bbJ_\bsnu} (\cQ_{\nu_j} - \cQ_{\nu_j-1}) (f)\;,
\eeq
which can be computed through  \eqref{eq:TensorQuad} with $2^d$ terms. If the quadrature points are nested, this computation only involves  $\prod_{j \in \bbJ_\bsnu} m_{\nu_j}$ times of evaluation of the function $f$. Otherwise, the number becomes $\prod_{j \in \bbJ_\bsnu}(m_{\nu_{j}}+m_{\nu_{j}-1})$. Both cost becomes feasible for small $d$.
By $\Lambda$ we denote an \emph{admissible} index set \cite{gerstner2003dimension}, also called \emph{downward closed} or \emph{monotonic} index set \cite{chkifa2014high, Schillings2013}, which is defined such that
\beq
\text{for any } \bsnu \in \cF, \text{ if } \bsnu \in \Lambda, \text{ then } \bsmu \in \Lambda \text{ for all } \bsmu \preceq\bsnu \; (i.e., \mu_j \leq \nu_j, \forall j \geq 1)\;.
\eeq
Then we can define a \emph{sparse quadrature operator} on the grids $G_{\Lambda} = \cup_{\bsnu \in \Lambda} G_\bsnu$ as
\beq\label{eq:QuadLambdaIntegral}
\cQ_{\Lambda} (f) = \sum_{\bsnu \in \Lambda} \triangle_\bsnu(f)\;.
\eeq
Note that both the full tensor-product quadrature and the Smolyak quadrature \cite{smolyak1963quadrature, gerstner1998numerical} can be represented as the sparse quadrature with $\Lambda := \{\bsnu \in \cF, |\bsnu|_\infty \leq l\}$ for the former,  where $|\bsnu|_\infty := \max_{j\geq 1} \nu_j$, and $\Lambda := \{\bsnu \in \cF, |\bsnu|_1 \leq l\}$ for the latter. 
A more general sparse quadrature is an anisotropic sparse quadrature in \cite{gerstner2003dimension, nobile2008anisotropic}, where the maximum level of the index $\nu_j$ is allowed to vary for different $j$. The index set $\Lambda$ and the corresponding quadrature points $G_{\Lambda}$ for the full tensor-product quadrature, the isotropic Smolyak sparse quadrature, and the anisotropic sparse quadrature are shown for GK with $l = 4$ in Fig. \ref{fig:sparsegrid} in two dimensions, from which we can observe large reduction of the points successively.

\begin{figure}[!htb]
\begin{center}
\includegraphics[scale=0.25]{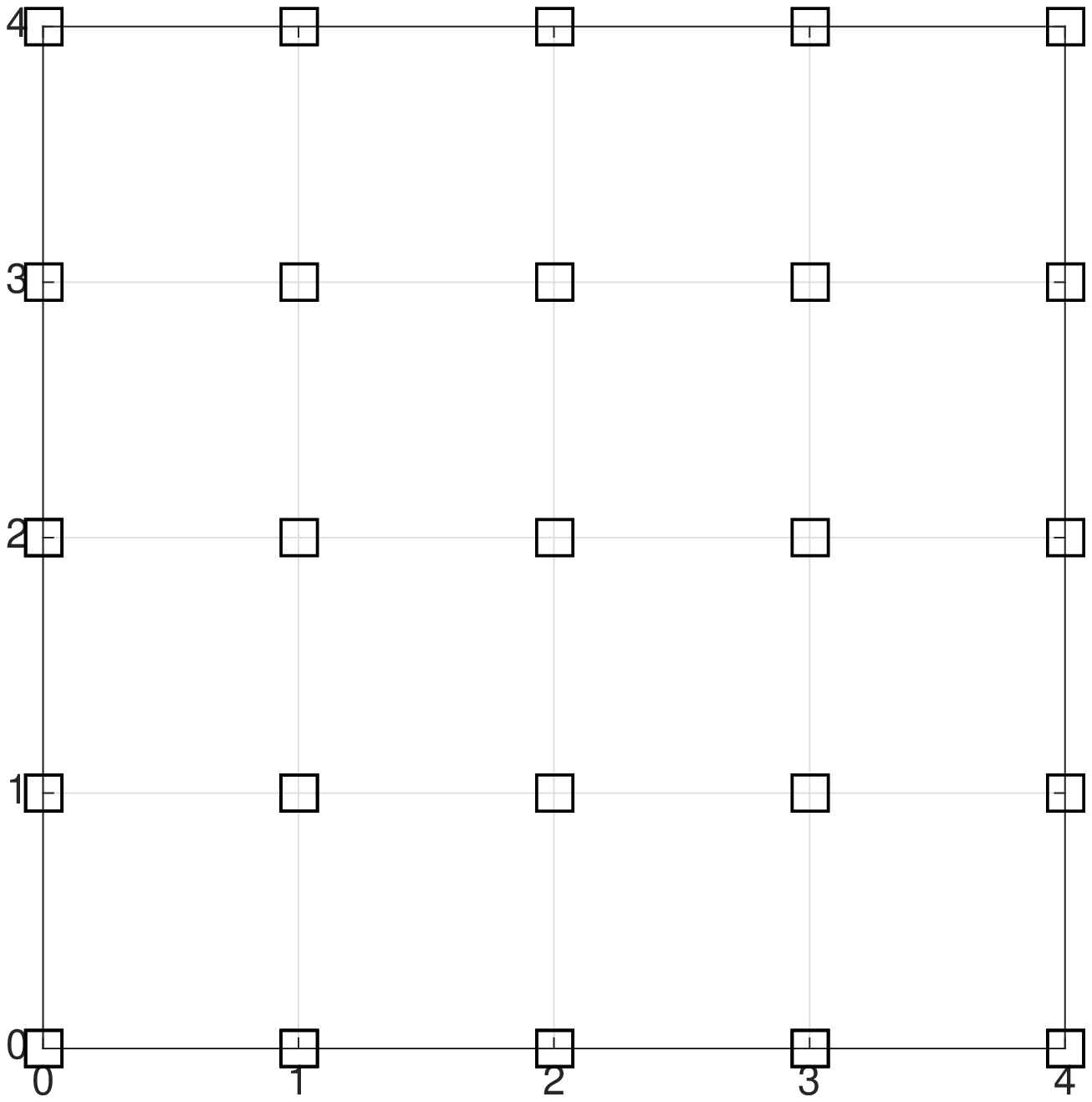}\hspace*{1cm}
\includegraphics[scale=0.25]{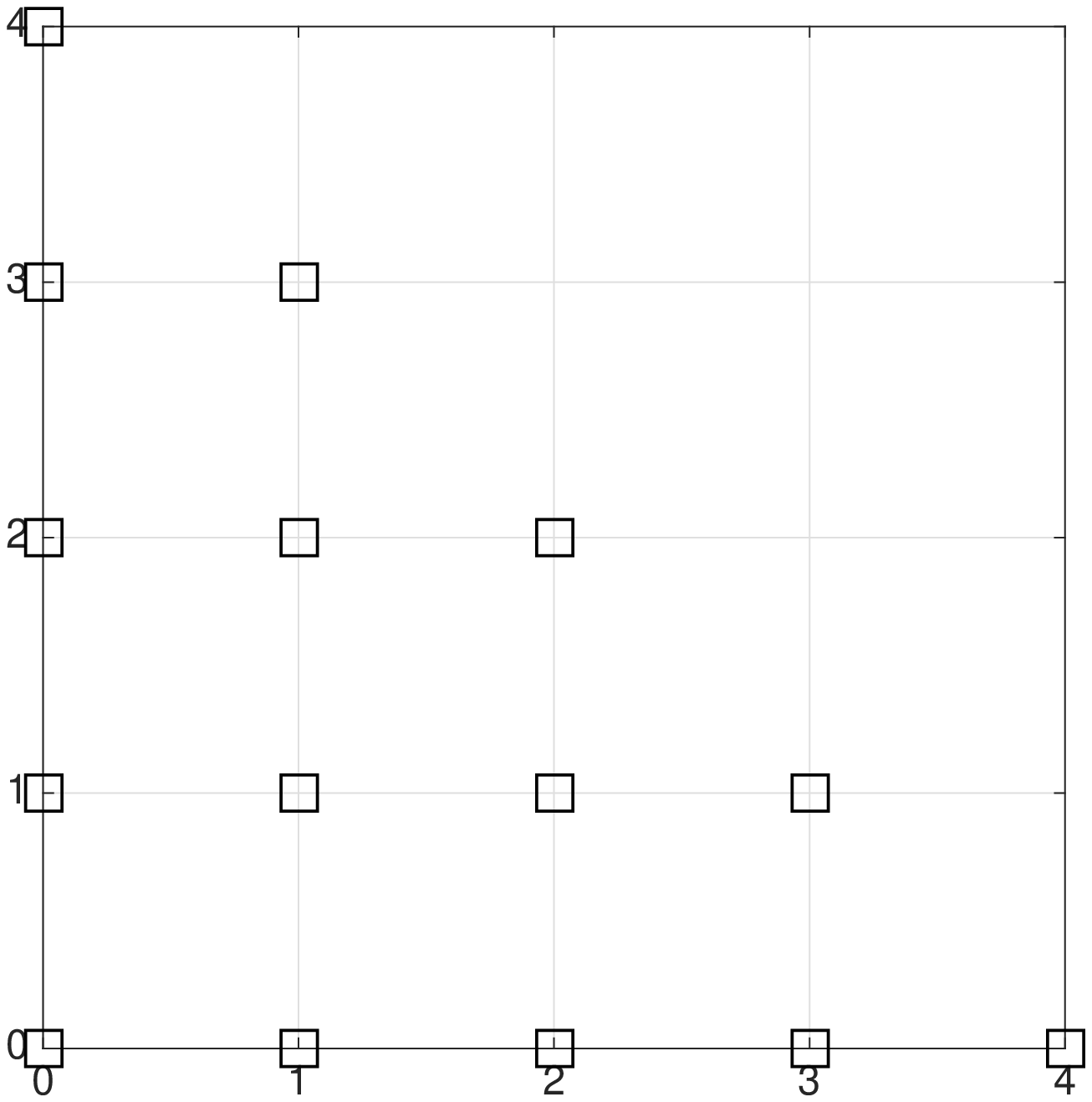}\hspace*{1cm}
\includegraphics[scale=0.25]{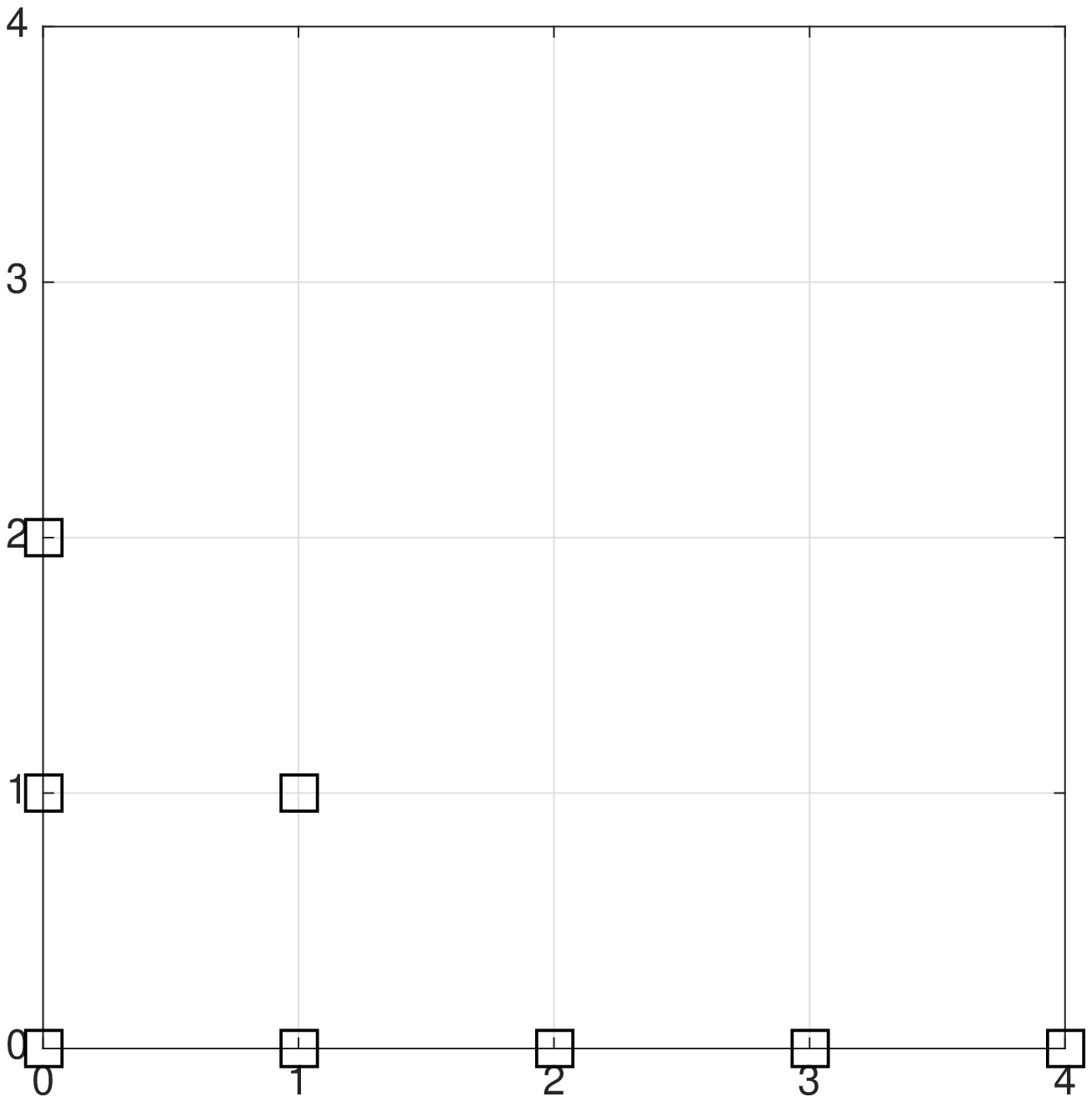}\\[8pt]

\includegraphics[scale=0.25]{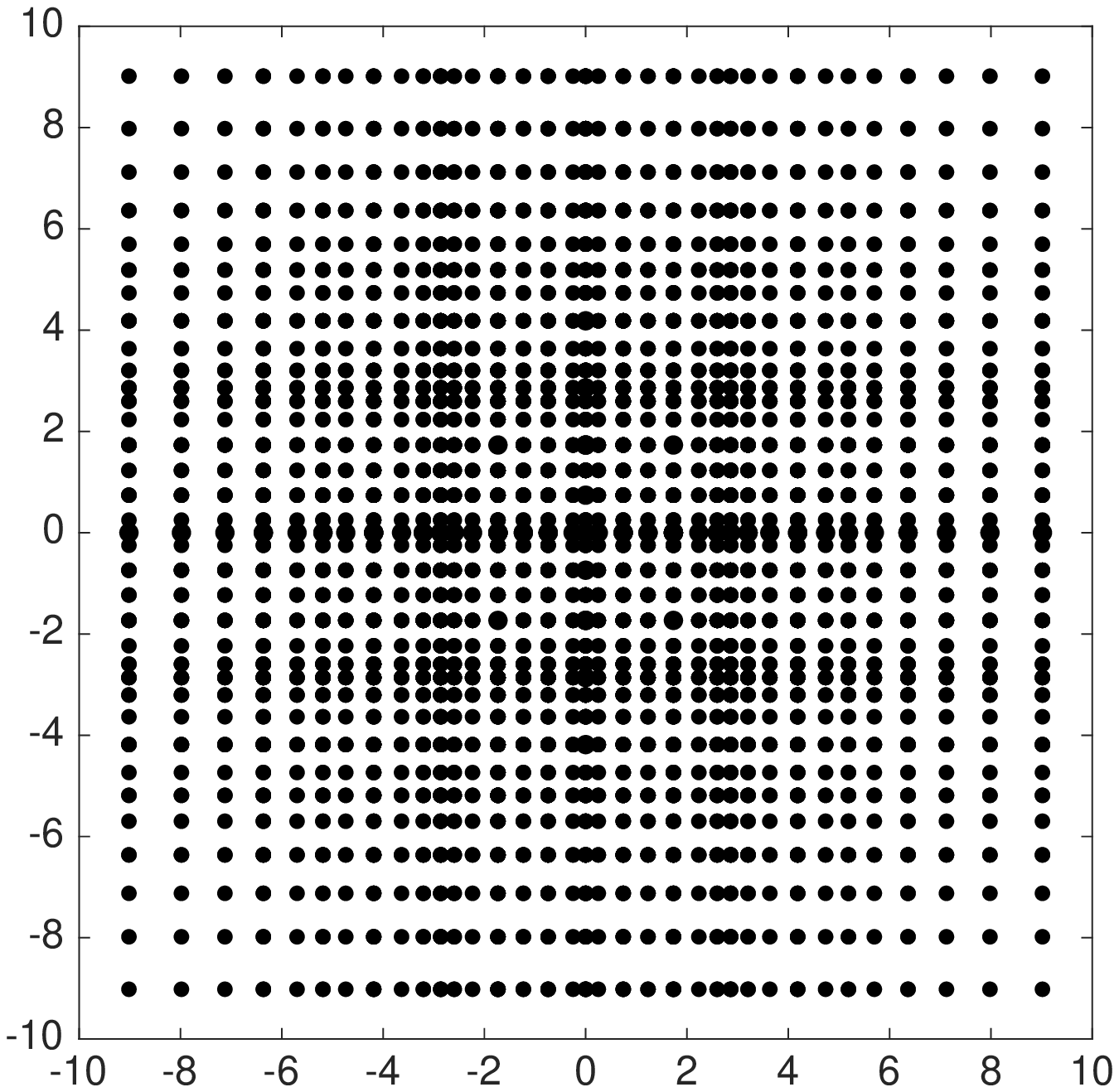}\hspace*{.9cm}
\includegraphics[scale=0.25]{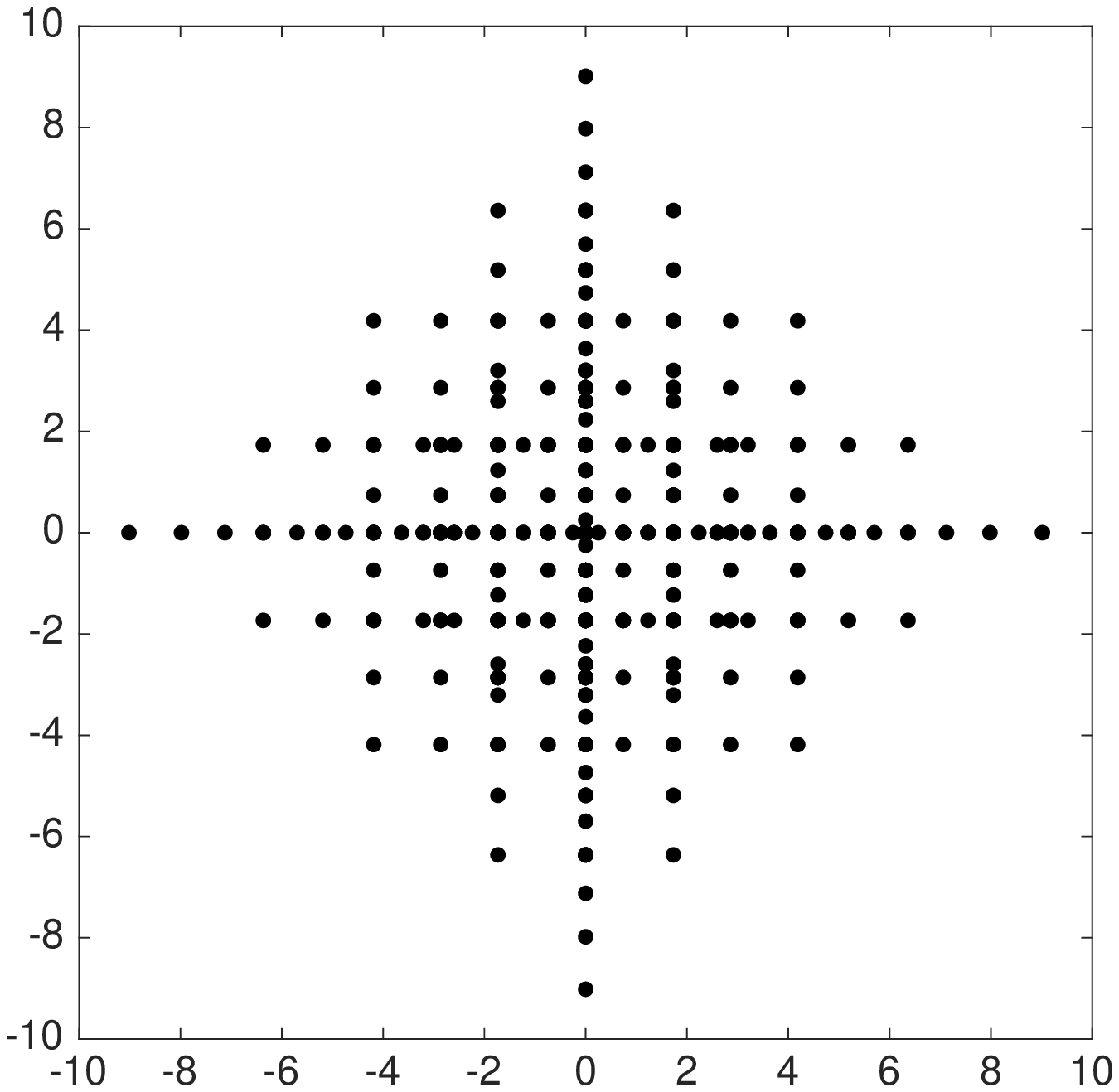}\hspace*{.9cm}
\includegraphics[scale=0.25]{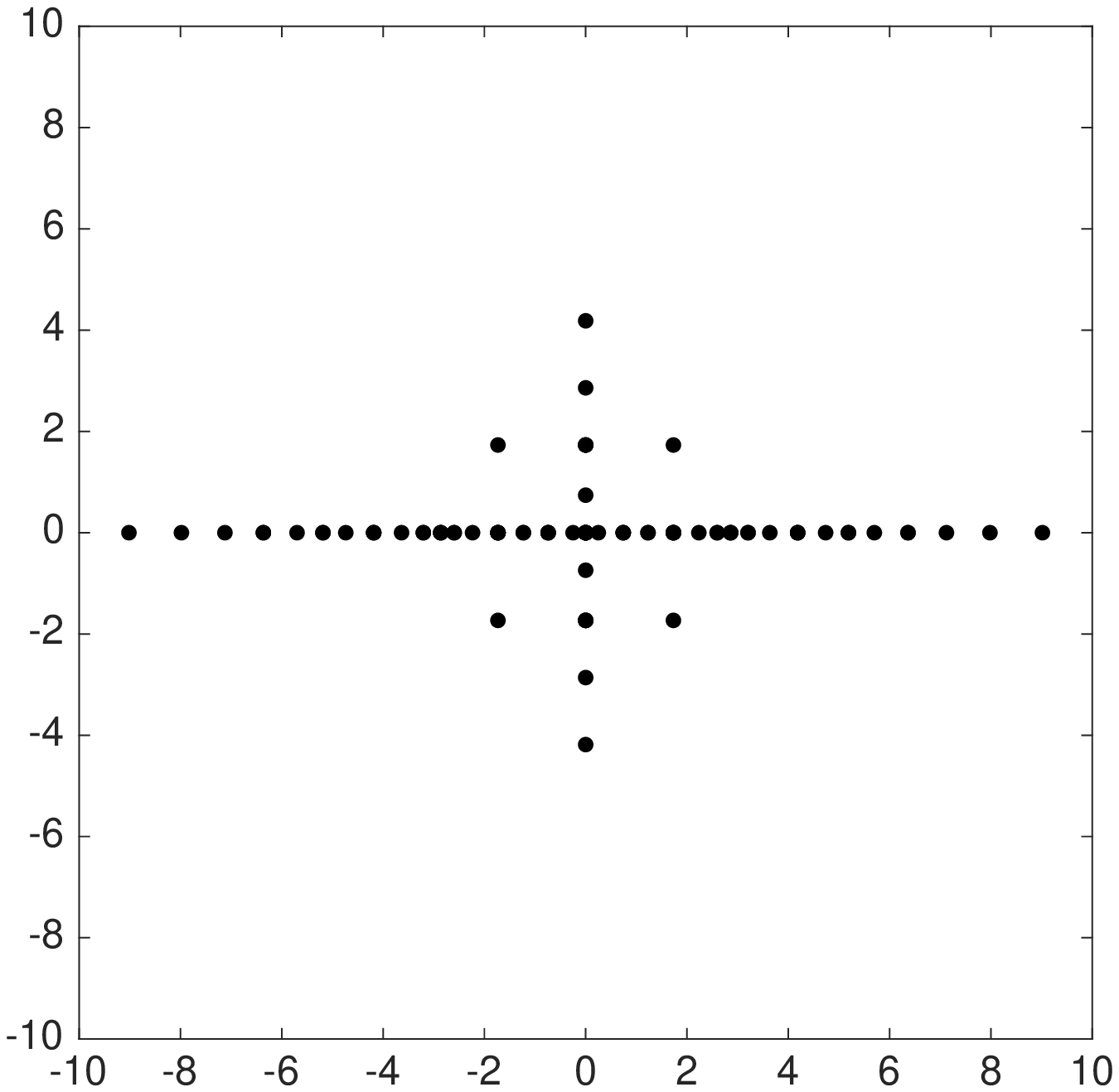}
\end{center}
\caption{The admissible index sets (top) and the corresponding GK quadrature points (bottom). Left: tensor-product grids; middle: isotropic Smolyak sparse grids; right: anisotropic sparse grids.}\label{fig:sparsegrid}
\end{figure}

\section{Convergence analysis}
\label{sec:ConvAnal}
Let $N$ be the cardinality of an admissible index set $\Lambda$, which we denote as $\Lambda_N$ to reflect its cardinality. In this section we provide sufficient conditions for the existence of a sparse quadrature $\cQ_{\Lambda_N}$ whose quadrature error $||I(f) - \cQ_{\Lambda_N}(f)||_\cS$ does not depend on the dimension $J$, thus breaking the curse of dimensionality. Moreover, we analyze the convergence rate of this error with respect to $N$ under certain assumptions on the regularity of the function $f$ with respect to $\bsy$. We provide two specific examples for which such assumptions are illustrated.

\subsection{Convergence analysis}
\label{subsec:DimenIndeCon}
 
In general, we consider the function $f$ to have finite second moment,  i.e., 
\beq
||f||_{L^2_\bsgamma(Y, \cS)} = \left(\int_{Y}||f(\bsy)||^2_\cS d\bsgamma(\bsy)\right)^{1/2} < \infty\;.
\eeq 
In this situation, $f$ admits a polynomial expansion on the Hermite series \cite{bachmayr2017sparse}, i.e.
\beq\label{eq:HermiteExpansion}
f(\bsy) = \sum_{\bsnu \in \cF} f_{\bsnu} H_{\bsnu}(\bsy)\;,
\eeq
where the multivariate Hermite polynomials $H_\bsnu(\bsy)$ and the coefficient $f_{\bsnu}$ read
\beq
H_\bsnu(\bsy) = \prod_{j \geq 1} H_{\nu_j}(y_j), \text{ and } f_\bsnu = \int_{Y} f(\bsy)H_{\bsnu}(\bsy)d\bsgamma(\bsy) \;.
\eeq
Here and in what follows we consider $J=\infty$ ($J \in \bbN$ is a special case where $y_j = 0$ for $j > J$). The univariate Hermite polynomials $\{H_n\}_{n\geq 0}$, as given in \eqref{eq:Hermite}, are orthonormal. Due to this orthonormality, we have the Parseval's identity
\beq
||f||_{L^2_\bsgamma(Y,\cS)}^2 = \sum_{\bsnu \in \cF} ||f_\bsnu||_\cS^2\;,
\eeq
i.e., $\{||f_\bsnu||_\cS\}_{\bsnu \in \cF} \in \ell^2(\cF)$,  a sufficient and necessary condition for $f \in L^2_\bsgamma(Y,\cS)$.

\begin{assumption}\label{ass:Quadrature}
We make the following assumptions on the properties of the univariate quadrature operators $\{\cQ_l\}_{l\geq 0}$:
\begin{enumerate}
\item[A.1] The quadrature at level $l$ is exact for all the functions $f \in \bbP_l \otimes \cS$, where $\bbP_l = \text{span}\{y^i: i = 0, \dots, l\}$, i.e.
\beq
I(f) = \cQ_l(f) \quad \forall f \in \bbP_l \otimes \cS\;.
\eeq
In particular, $I(H_n) = \cQ_l(H_n)$ for Hermite polynomials $H_n$, $n = 0, \dots, l$.

\item[A.2] The quadrature $\cQ_l(H_n)$ for $H_n$ with $n > l$ is bounded by $2$, i.e.
\beq 
|\cQ_l(H_n)| < 2, \quad \forall l \geq 0\;.
\eeq
\end{enumerate}
\end{assumption}
Both the Gauss--Hermite (GH) quadrature and the Genz--Keister (GK) quadrature satisfy assumption A.1 for $m_l \geq l+1$, see \cite{gil2007numerical} and \cite{genz1996fully}, while it does not hold for the transformed Gauss--Kronrod--Patterson (tGKP) quadrature. As for assumption $A.2$, we can verify it for the GH quadrature in the following lemma.

\begin{lemma}\label{prop:HermiteBound}
For the Gauss--Hermite quadrature with $m_l = l + 1$ quadrature points at any level $l \geq 0$, see Sec. \ref{subsec:UnivQuad}, we have the bound
\begin{equation}
|\cQ_l(H_n)| < 2, \quad \forall n \geq 0\;,
\end{equation}
for the orthonormal Hermite polynomials $H_n$, $n \geq 0$, defined in \eqref{eq:Hermite}.
\end{lemma}

The proof is based on the Cram\'er inequality, e.g., in \cite{abramowitz1966handbook}, that is made aware from \cite[Lemma 14]{ernst2016convergence}, and the Markoff's theorem, e.g., in \cite{szeg1939orthogonal}.
\begin{proof}
For the (physicists') orthogonal Hermite polynomials $\tilde{H}_n$, $n = 0, 1, \dots $, defined as \cite[Chap. 22, p. 776]{abramowitz1966handbook}
\beq
\tilde{H}_n(x) = (-1)^n e^{x^2} \frac{d^n}{dx^n} e^{-x^2}, \text{ with } \int_{-\infty}^\infty \tilde{H}_n(x) \tilde{H}_m(x) e^{-x^2} dx = \sqrt{\pi} 2^n n! \delta_{nm}\;,
\eeq
we have the Cram\'er inequality \cite[Chap. 22, p. 787]{abramowitz1966handbook}
\beq
|\tilde{H}_n(x)| < c 2^{n/2} \sqrt{n!} e^{x^2/2}, \text{ with } c  \approx 1.086435\;.
\eeq
Consequently, with proper rescaling for the (probabilists') orthonormal Hermite polynormials defined in \eqref{eq:Hermite}, i.e., $H_n(x) = (2^{n/2}\sqrt{n!})^{-1} \tilde{H}_n(x/\sqrt{2})$, we have 
\beq\label{eq:Hnfn}
|H_n(x)| < c e^{x^2/4}\;. 
\eeq
For the smooth function $f(x) = e^{x^2/4}$, by Markoff's theorem \cite[Chap. 16, p. 378]{szeg1939orthogonal} (note there $n = m_l =  l + 1$ for our $l$ here) there exists $\xi \in \bbR$ s.t.
\beq\label{eq:fxQlf}
\int_{\bbR} f(x) \rho(x)dx = \cQ_l(f) + \frac{f^{(2l+2)}(\xi)}{(2l+2)!}k_{l+1}^{-2}\;,
\eeq
where $k_{l+1}$ is the highest coefficient of the Hermite polynomial $H_{l+1}(x)$. As any even order derivative of $f$ is non-negative (see \cite[Lemma 4]{nevai1980mean}), from \eqref{eq:fxQlf} we have 
\beq
\cQ_l(f) \leq \int_{\bbR} f(x) \rho(x)dx = \frac{1}{\sqrt{2\pi}} \int_{\bbR} e^{x^2/4} e^{-x^2/2} dx = \sqrt{2}\;.
\eeq
Hence, we obtain 
\beq
|\cQ_l(H_n)| \leq \cQ_l(|H_n|) \leq c \cQ_l(f) \leq \sqrt{2} c \approx 1.536451 < 2\;,
\eeq
where the first inequality is due to the positivity of the quadrature weights \eqref{eq:HermiteWeight}, and the second one is due to the bound \eqref{eq:Hnfn}.
\end{proof}

As for the GK quadrature and the tGKP quadrature, no theoretical result is known to us for assumption A.2. 
Numerically, we compute $\cQ_l(H_n)$ by all the three types of quadrature rules with all possible levels $l$ and degrees of Hermite polynomial $n$ upto machine precision. The results show that $A.2$ holds in all cases with a sharper bound $|\cQ_l(H_n)| \leq 1$. The left of Fig. \ref{fig:QuadAccu} displays the numerical value $|\cQ_l(H_n)|$ for the three quadrature rules with $l = 3$ and $n = 0, \dots, 150$ (the polynomial degree $n$ can not be larger due to machine precision); the right of Fig. \ref{fig:QuadAccu} shows $|\cQ_l(H_n)| \leq 1$ by the GH2 (GH with $m_l = 2^{l+1}-1$) quadrature at $l = 0, 1, 2, 3, 4, 5$ and $n = 0, \dots, 150$. Moreover, from the left figure we can also see that GH2 (with $m_3 = 15$ points) is exact (with machine precision) for $I(H_n)$ for $n = 0, \dots, 29$, and GK (with $m_3= 19$ points) is exact for $n = 0, \dots, 29$, which satisfy assumption A.1.

\vspace*{0.2cm}
\begin{figure}[htb!]
\begin{center}
\includegraphics[scale=0.33]{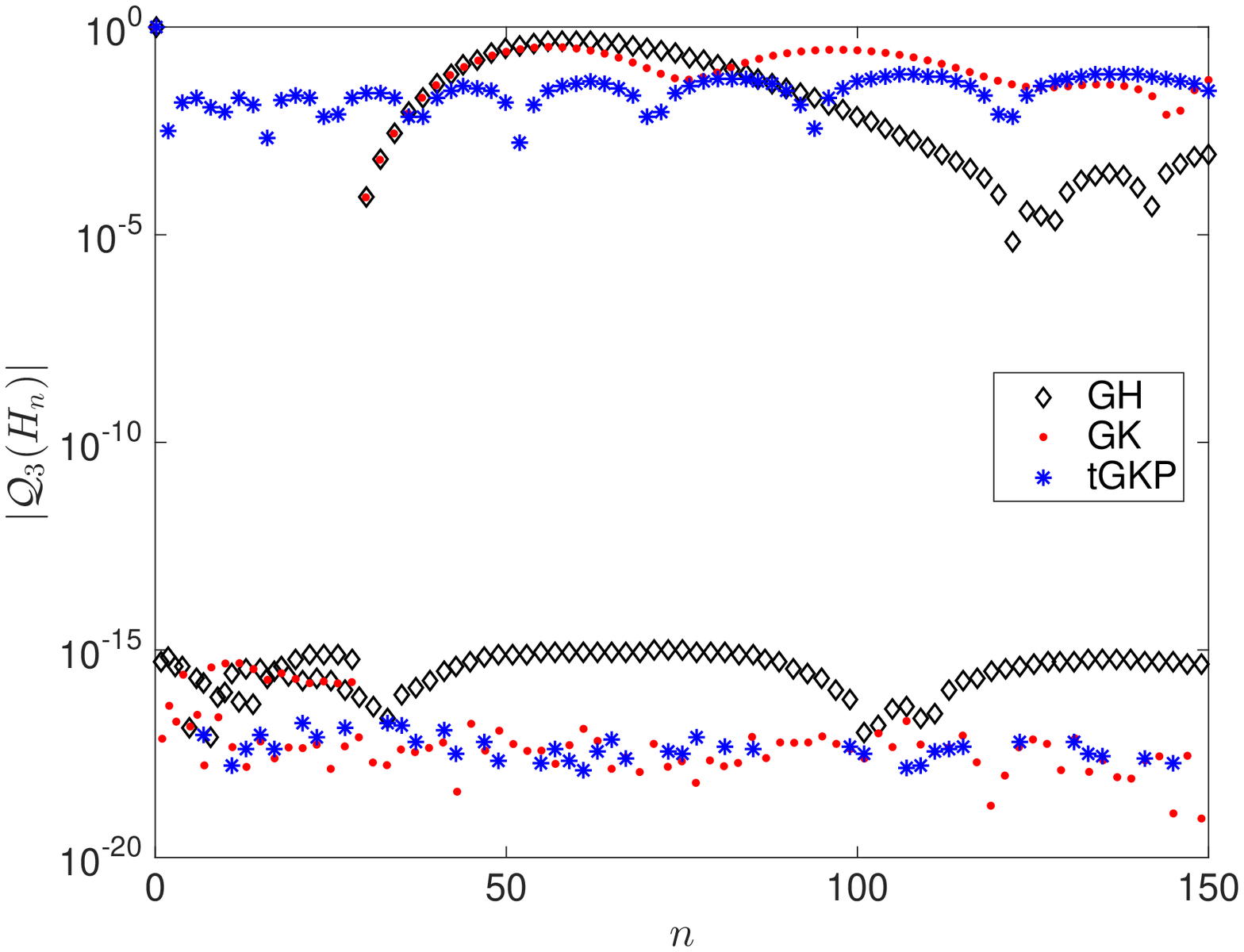}
\includegraphics[scale=0.33]{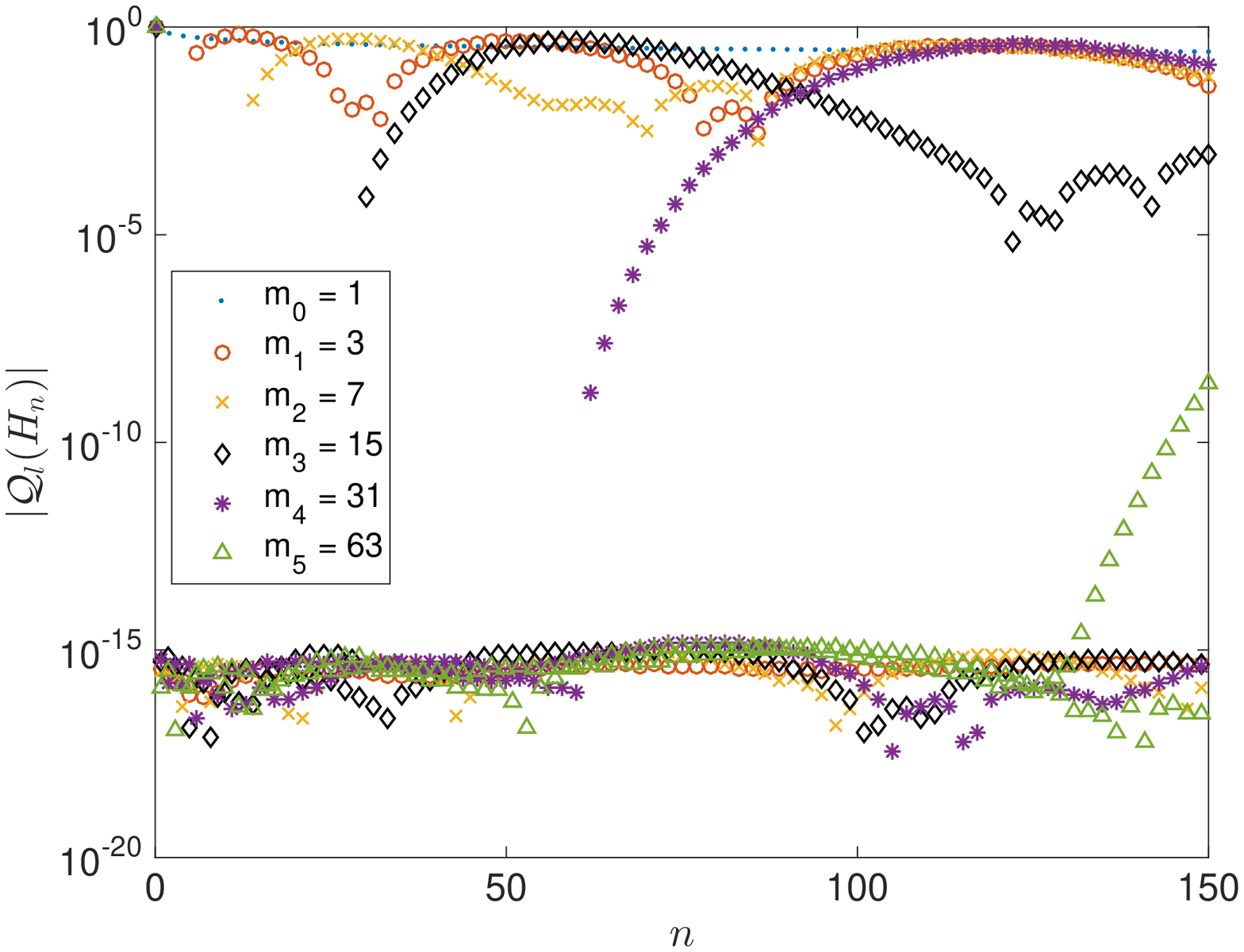}
\end{center}
\caption{Left: the numerical values $|\cQ_3(H_n)|$ by GH2, GK, and tGKP; right: the numerical values $|\cQ_l(H_n)|$ by GH2 with $m_l = 2^{l+1}-1$ points for $l = 0, 1, 2, 3, 4, 5$.}
\label{fig:QuadAccu}
\end{figure}

Assumption \ref{ass:Quadrature} implies the exactness and the boundedness of the sparse quadrature $\cQ_{\Lambda}$ in multiple dimensions as presented in the following lemma. Similar results have been obtained on the exactness of the sparse quadrature for integration with respect to uniform measure, see, e.g., \cite{back2011stochastic, Schillings2013}.
\begin{lemma}
Under Assumption \ref{ass:Quadrature}, for any admissible index set $\Lambda \subset \cF$, we have 
\beq\label{eq:ExactLambda}
I(f) = \cQ_\Lambda(f), \quad \forall f \in \bbP_\Lambda \otimes \cS\;,
\eeq 
where $\bbP_{\Lambda} = \text{span}\{\prod_{j\geq 1}y_j^{\nu_j}, \bsnu \in \Lambda\}$. In particular, as $H_\bsnu \in \bbP_\Lambda$, we have
\beq\label{eq:QuadLambda}
I(H_\bsnu) = \cQ_\Lambda(H_\bsnu) , \quad \forall \bsnu \in \Lambda\;.
\eeq
Moreover, for any $\bsnu \in \cF \setminus \bsnull$, we have 
\beq\label{eq:QuadBound}
|\cQ_{\Lambda \cap \cR_\bsnu} (H_{\bsnu})| \leq \prod_{j \in \bbJ_\bsnu} (1+\nu_j)^3\;.
\eeq 
where the index set $\cR_\bsnu := \{\bsmu \in \cF: \bsmu \preceq \bsnu\}$, and $\bbJ_\bsnu$ is the support set of $\bsnu$.
\end{lemma}
\begin{proof}
The result \eqref{eq:ExactLambda} can be obtained by induction based on the assumption A.1, e.g., as in \cite[Theorem 4.2]{Schillings2013} for the uniform measure.  
Here, we provide a different proof for the Gaussian measure. First, for $\Lambda = \{\bszero\}$, i.e., $f(\bsy) = u_\bszero$ with some function $u_\bszero \in \cS$ for all $\bsy \in Y$, we have $I(f) = u_\bszero$ and $\cQ_\bszero(f) = f(\bszero) = u_\bszero$, which verifies \eqref{eq:ExactLambda}. 
Suppose \eqref{eq:ExactLambda} holds for an admissible set $\Lambda$, then we only need to verify that \eqref{eq:ExactLambda} also holds for the admissible set $\Lambda_{+k} = \Lambda \cup \{\bsnu_{+k}\}$ for all possible $k \in \bbN$, where $\bsnu_{+k} = \bsnu^*+\bse_k$ for some $\bsnu^* \in \Lambda$ such that $(\bsnu^*)_k \geq \nu_k$ for all $\bsnu \in \Lambda$. Here $\bse_k\in \cF$ whose $k$-th elements is one and all other elements are zero. In fact, the function $f \in \bbP_{\Lambda_{+k}} \otimes \cS$ can be decomposed as 
\beq
f(\bsy) = \sum_{\bsnu \in \Lambda_{+k} } \bsy^{\bsnu} u_\bsnu = \sum_{\bsnu \in \Lambda } \bsy^{\bsnu} u_\bsnu +\bsy^{\bsnu_{+k}} u_{\bsnu_{+k}} \equiv f_{\Lambda}(\bsy) + f_{+k}(\bsy), 
\eeq
where we have denoted $f_{\Lambda}(\bsy) = \sum_{\bsnu \in \Lambda } \bsy^{\bsnu} u_\bsnu $ and $f_{+k}(\bsy) = \bsy^{\bsnu_{+k}} u_{\bsnu_{+k}} $.
Then by the definition \eqref{eq:QuadLambdaIntegral} of the sparse quadrature operator, we have 
\beq
\cQ_{\Lambda_{+k}}(f_{\Lambda}) = \cQ_{\Lambda}(f_{\Lambda}) + \triangle_{\bsnu_{+k}} (f_{\Lambda})\;,
\eeq
where the first term $\cQ_{\Lambda}(f_{\Lambda}) = I(f_{\Lambda})$ by the induction's assumption, and the second term, by the definition \eqref{eq:TensorDiff}, can be explicitly written as 
\beq
\triangle_{\bsnu_{+k}} (f_{\Lambda}) = \sum_{\bsnu \in \Lambda} u_{\bsnu} \bigotimes_{j \in \bbJ_{\bsnu_{+k}}} (\cQ_{(\bsnu_{+k})_j} - \cQ_{(\bsnu_{+k})_j - 1}) (y_j^{\nu_j})\;.
\eeq
By A.1 and the fact $\nu_k \leq (\bsnu^*)_k$ for all $\bsnu \in \Lambda$ and $\bsnu_{+k} = \bsnu^* + \bse_k$, we have
\beq\label{eq:cQnuk}
(\cQ_{(\bsnu_{+k})_k} - \cQ_{(\bsnu_{+k})_k - 1}) (y_k^{\nu_k}) = (\cQ_{(\bsnu^*)_k+1} - \cQ_{(\bsnu^*)_k }) (y_k^{\nu_k})= I(y_k^{\nu_k}) - I(y_k^{\nu_k}) = 0\;,
\eeq 
which implies that $\triangle_{\bsnu_{+k}}(f_{\Lambda}) = 0$, thus $\cQ_{\Lambda_{+k}}(f_{\Lambda}) = I(f_{\Lambda})$. As for $f_{+k}$, we have 
\beq
\cQ_{\Lambda_{+k}}(f_{+k}) = \sum_{\bsnu \in \Lambda_{+k}} \triangle_\bsnu (f_{+k}) = \sum_{\bsnu \in \cR_{\bsnu_{+k}}} \triangle_\bsnu (f_{+k}) + \sum_{\bsnu \in \Lambda_{+k} \setminus \cR_{\bsnu_{+k}}} \triangle_\bsnu (f_{+k})\;, 
\eeq
where we recall that $\cR_{\bsnu_{+k}} = \{\bsmu \in \cF: \bsmu \preceq \bsnu_{+k} \}$. Then by A.1 the first term yields 
\beq
\sum_{\bsnu \in \cR_{\bsnu_{+k}}} \triangle_\bsnu (f_{+k}) =  \cQ_{\cR_{\bsnu_{+k}}} (f_{+k}) = u_{\bsnu_{+k}} \bigotimes_{j\geq 1} \cQ_{(\bsnu_{+k})_j}\left(y_j^{(\bsnu_{+k})_j}\right) = I(f_{+k})\;,
\eeq
and $\triangle_\bsnu(f_{+k})$ vanishes for each $\bsnu \in \Lambda_{+k} \setminus \cR_{\bsnu_{+k}}$ by the same reasoning as in \eqref{eq:cQnuk}, i.e., there exists $j \in \bbJ_{\bsnu}$ such that $\nu_j > (\bsnu_{+k})_j$, so that $(\cQ_{\nu_j} - \cQ_{\nu_j - 1})\left(y_j^{(\bsnu_{+k})_j}\right)= 0$. Therefore, we also have $\cQ_{\Lambda_{+k}}(f_{+k}) = I(f_{+k})$, so that $\cQ_{\Lambda_{+k}}(f) = I(f)$ for any $f \in \bbP_{\Lambda_{+k}} \otimes \cS$. This completes the induction and concludes the equality \eqref{eq:ExactLambda}.

To check \eqref{eq:QuadBound}, by the definition of the sparse quadrature in \eqref{eq:QuadLambdaIntegral} we have
\beq
|\cQ_{\Lambda \cap \cR_\bsnu} (H_{\bsnu})| = \Big | \sum_{\bsmu \in \Lambda \cap \cR_\bsnu} \triangle_\bsmu(H_\bsnu) \Big | \leq \sum_{\bsmu \in \Lambda \cap \cR_\bsnu} |\triangle_\bsmu(H_\bsnu)| \leq \sum_{\bsmu \in \cR_\bsnu} |\triangle_\bsmu(H_\bsnu)|\;.
\eeq
By the definition of $\triangle_\bsmu$ in \eqref{eq:TensorDiff}, we have 
\beq
|\triangle_\bsmu(H_\bsnu)| \leq \prod_{j \in \bbJ_{\bsmu}} |\cQ_{\mu_j}(H_{\nu_j}) - \cQ_{\mu_j - 1}(H_{\nu_j})| \leq \prod_{j \in \bbJ_{\bsmu}}4 = 4^{|\bbJ_\bsmu|}\;,
\eeq
where the second bound is due to the assumption A.2. Therefore, we have 
\beq
 \sum_{\bsmu \in \cR_\bsnu} |\triangle_\bsmu(H_\bsnu)| \leq  \sum_{\bsmu \in \cR_\bsnu} 4^{|\bbJ_\bsmu|} \leq \sum_{\bsmu \in \cR_\bsnu} 4^{|\bbJ_\bsnu|} = \prod_{j \in \bbJ_\bsnu} 4(1+\nu_j) \leq \prod_{j \in \bbJ_\bsnu} (1+\nu_j)^3\;,
\eeq
where for the equality we have used $\sum_{\bsmu \in \cR_\bsnu} 1 = \prod_{j \in \bbJ_\bsnu} (1+\nu_j)$ and for
last inequality we have used $4(1+n) \leq (1+n)^3$ for $n \geq 1$, which completes the proof.
\end{proof}

The following lemma bounds the quadrature error $||I(f) - \cQ_{\Lambda_N}(f)||_\cS$ in terms of the weighted $\ell^1$-norm of the Hermite coefficient $\{||f_{\bsnu}||_\cS\}_{\bsnu \in \cF\setminus \Lambda_N}$. Similar results using Legendre polynomial expansion and triangular inequality can be found in \cite[Lemma 4.2]{chkifa2014high} and \cite[Lemma 4.5]{Schillings2013} for interpolation and integration with uniform measure. Instead of relying on the Lebesgue constant in these papers, we use the orthogonality of the Hermite polynomials and the bound in assumption A.2. 
\begin{lemma}\label{lemma:WeightedSum}
Under Assumption \ref{ass:Quadrature}, for any $f\in L^2_{\bsgamma}(Y,\cS)$, we have that for any $N \in \bbN$, there exists an admissible index set $\Lambda_N \subset \cF$ with $|\Lambda_N| = N$, such that
\beq\label{eq:weightedSum}
||I(f) - \cQ_{\Lambda_N}(f)||_\cS \leq \sum_{\bsnu \in \cF\setminus \Lambda_N}c_\bsnu ||f_\bsnu||_\cS \;,
\eeq
where $c_\bsnu := \prod_{j\in \bbJ_\bsnu}(1+\nu_j)^3$, the upper bound obtained in \eqref{eq:QuadBound}.
\end{lemma}
\begin{proof} 
As $f\in L^2_{\bsgamma}(Y,\cS)$, we have the polynomial expansion of $f$ on the Hermite series as in \eqref{eq:HermiteExpansion}, so that 
\beq
\cQ_{\Lambda_N} (f) = \cQ_{\Lambda_N} \left( \sum_{\bsnu \in \cF} f_{\bsnu} H_{\bsnu} \right) =  \sum_{\bsnu \in \Lambda_N} f_{\bsnu} \cQ_{\Lambda_N} (H_{\bsnu}) + \sum_{\bsnu \in \cF \setminus \Lambda_N} f_{\bsnu} \cQ_{\Lambda_N} (H_{\bsnu})\;.
\eeq
Therefore, by the identity \eqref{eq:QuadLambda} we obtain
\beq\label{eq:IQError}
||I(f) - \cQ_{\Lambda_N} (f)||_\cS \leq \sum_{\bsnu \in \cF \setminus \Lambda_N} ||f_{\bsnu}||_\cS |(I-\cQ_{\Lambda_N}) (H_{\bsnu})|
\eeq
For any $\bsnu \in \cF\setminus \bsnull$, there exists $j \in \bbN$ such that $\nu_j \neq 0$, for which we have $I_j(H_{\nu_j}) = 0$ due to the orthogonality of $H_{\nu_j}$, hence
\beq
I(H_\bsnu) = \prod_{j\geq 1} I_j(H_{\nu_j}) = 0\;.
\eeq
Moreover, for any $\bsnu \in \cF$, we have 
\beq
\begin{split}
\cQ_{\Lambda_N}(H_\bsnu)& = \sum_{\mu \in \Lambda_N} \triangle_\bsmu(H_\bsnu) \\
&= \sum_{\mu \in \Lambda_N} \prod_{j \geq 1} (\cQ_{\mu_j}(H_{\nu_j}) - \cQ_{\mu_j-1}(H_{\nu_j}))\\
& = \sum_{\mu \in \Lambda_N \cap \cR_\bsnu} \prod_{j \geq 1} (\cQ_{\mu_j}(H_{\nu_j}) - \cQ_{\mu_j-1}(H_{\nu_j}))\\
&= \sum_{\mu \in \Lambda_N \cap \cR_\bsnu} \triangle_\bsmu(H_\bsnu) = \cQ_{\Lambda_N \cap \cR_\bsnu}(H_\bsnu)
\;,
\end{split}
\eeq
where the third equality is due to the assumption A.1. As a result, \eqref{eq:IQError} becomes 
\beq
||I(f) - \cQ_{\Lambda_N} (f)||_\cS \leq \sum_{\bsnu \in \cF \setminus \Lambda_N} ||f_{\bsnu}||_\cS |\cQ_{\Lambda_N \cap \cR_\bsnu}(H_\bsnu)| \leq \sum_{\bsnu \in \cF \setminus \Lambda_N} c_\bsnu  ||f_{\bsnu}||_\cS\;,
\eeq
which completes the proof by using the bound \eqref{eq:QuadBound}.
\end{proof}

In order to control the quadrature error, which is bounded by a weighted sum of the Hermite coefficients as above, we make the following assumptions from \cite[Theorem 3.3]{bachmayr2017sparse} on the derivatives of the function $f$ with respect to the parameter $\bsy$.

\begin{assumption}\label{ass:DeriBound}
\begin{enumerate}
\item[B.1]
Let 
$0 < q < 2$
, and $(\tau_j)_{j\geq 1}$ be a positive sequence such that 
\beq\label{eq:lqtauj}
(\tau_j^{-1})_{j\geq 1} \in \ell^q(\bbN)\;.
\eeq 
\item[B.2]
Let $r$ be the smallest integer such that $r > 14/q$, we assume $\partial^\bsmu_\bsy f \in L^2_{\bsgamma}(Y,\cS)$ and there holds 
\beq\label{eq:boundedinteg}
\sum_{|\bsmu|_\infty \leq r} \frac{\bstau^{2\bsmu}}{\bsmu!} \int_{Y} ||\partial^\bsmu_\bsy f(\bsy)||^2_\cS d\bsgamma(\bsy) < \infty\;,
\eeq 
where $\bstau^{2\bsmu} = \prod_{j\geq 1} \tau_j^{2\mu_j}$, $\bsmu ! = \prod_{j \geq 1} \mu_j!$, and $\partial_\bsy^\bsmu f(\bsy) = \left(\prod_{j \geq 1} \partial_{y_j}^{\mu_j} \right) f(\bsy)$.
\end{enumerate}
\end{assumption}

\begin{remark}
Assumption \ref{ass:DeriBound} characterizes the relation between the regularity of the function $f$ with respect to the parameter $\bsy$ and sparsity of the parametrization, i.e., the anisotropic property of the function with respect to different dimensions. The smaller $q$ is, the faster $\tau_j$ grows, so the faster $\partial_\bsy^{\bsmu}f(\bsy)$ decays with respect to $j$, and as $r > 14/q$ becomes larger, the higher orders of derivative are needed. We will present two examples in the next section to verify Assumption \ref{ass:DeriBound} and illustrate this discussion.
\end{remark} 

The following result establishes the equivalence between the weighted summability of the integral of the mixed derivatives and the weighted summability of the Hermite coefficients, which is the key to bring the sparsity of the parametrization to the dimension-independent convergence rate. 

\begin{proposition}\label{prop:EquivFourierDeri}
\cite[Theorem 3.3, Lemma 5.1]{bachmayr2017sparse}
Under Assumption \ref{ass:DeriBound}, we have 
\beq
\sum_{|\bsmu|_\infty \leq r} \frac{\bstau^{2\bsmu}}{\bsmu!} \int_{Y} ||\partial^\bsmu_\bsy f(\bsy)||^2_\cS d\bsgamma(\bsy) = \sum_{\bsnu \in \cF} b_{\bsnu} ||f_\bsnu||_\cS^2\;,
\eeq
where the weights $b_\bsnu$ given by 
\beq\label{eq:bnu}
b_\bsnu = \sum_{|\bsmu|_\infty\leq r} 
\left(
\begin{array}{cc}
\bsnu \\
\bsmu
\end{array}
\right) \bstau^{2\bsmu}, \quad \text{ with } 
\left(
\begin{array}{cc}
\bsnu \\
\bsmu
\end{array}
\right) = 
\prod_{j \geq 1}
\left(
\begin{array}{cc}
\nu_j \\
\mu_j
\end{array}
\right)\;,
\eeq
satisfies the summability condition 
\beq\label{eq:bsnuq2} 
\sum_{\bsnu \in \cF} b_\bsnu^{-q/2} < \infty 
\eeq
for any integer $r$ such that $r > 2/q$.
\end{proposition}


Based on the summability \eqref{eq:bsnuq2} and its proof, we obtain the following result. 
\begin{lemma}\label{lemm:cnubnu}
Under Assumption \ref{ass:DeriBound}, 
for any $\eta \geq q/4$, we have 
\beq\label{eq:bnucnu}
\sum_{\bsnu \in \cF} \left(\frac{b_\bsnu}{c_\bsnu^{1/\eta}}\right)^{-2\eta} < \infty\;.
\eeq
\end{lemma}

\begin{proof}
By the definition of $b_\bsnu $ in \eqref{eq:bnu}, we can rewrite it as
\beq
b_\bsnu = \prod_{j\geq 1} \left(
\sum_{l = 0}^r
\left(
\begin{array}{cc}
\nu_j \\
l
\end{array}
\right)
\tau_j^{2l}
\right)\;.
\eeq
Then the left hand side of \eqref{eq:bnucnu} can be written in the factorized form as 
\beq
\begin{split}
\sum_{\bsnu \in \cF} \left(\frac{b_\bsnu}{c_\bsnu^{1/\eta}}\right)^{-2\eta}& = \sum_{\bsnu \in \cF} \prod_{j\geq 1}
\left(
\sum_{l=0}^r
\left(
\begin{array}{cc}
\nu_j \\
l
\end{array}
\right)
\tau_j^{2l}
\frac{1}{(1+\nu_j)^{3/\eta}}
\right)^{-2\eta}\\
& = \prod_{j\geq 1} 
\sum_{n\geq 0} 
\left(
\sum_{l=0}^r
\left(
\begin{array}{cc}
n \\
l
\end{array}
\right)
\tau_j^{2l}
\frac{1}{(1+n)^{3/\eta}}
\right)^{-2\eta}\;,
\end{split}  
\eeq 
as long as we can show that the product on the right hand side is finite. Now we have 
\beq\label{eq:subWeightSumD}
\begin{split}
&\sum_{n\geq 0} 
\left(
\sum_{l=0}^r
\left(
\begin{array}{cc}
n \\
l
\end{array}
\right)
\tau_j^{2l}
\frac{1}{(1+n)^{3/\eta}}
\right)^{-2\eta}\\
&\leq \sum_{n \geq 0} 
\left(
\left(
\begin{array}{cc}
n \\
n \wedge r
\end{array}
\right)
\tau_j^{2(n\wedge r)}
\frac{1}{(1+n)^{3/\eta}}
\right)^{-2\eta}\\
& \leq 1 + 2^{6} \tau_j^{-4\eta} + \cdots +r^{6} \tau_j^{-4\eta (r-1)} + C_{r,\eta} \tau_j^{- 4\eta r} =: d_j(r,\eta,\tau_j)\;,
\end{split}
\eeq
where in the first inequality we have only kept the term $l = n \wedge r =\min\{n, r\}$, and the constant $C_{r,\eta}$ is defined as
\beq
C_{r,\eta} = \sum_{n\geq r} 
\left(
\left(
\begin{array}{cc}
n\\
r
\end{array}
\right)
\frac{1}{(1+n)^{3/\eta}}
\right)^{-2\eta}= 
(r!)^{2\eta} \sum_{n\geq 0}
\left(
\frac{(n+1)\cdots(n+r)}{(1+n+r)^{3/\eta}}
\right)^{-2\eta}\;.
\eeq
As the term in the big parentheses grows as $n^{r-3/\eta}$ when $n \to \infty$, 
and $2\eta (r-3/\eta) > 1$ for any $\eta \geq q/4$ when $r > 14/q$, so that $C_{r,\eta} < \infty$. Since $(\tau_j^{-1})_{j \geq 1} \in \ell^q(\bbN)$ by Assumption \ref{ass:DeriBound}, we have $\tau_j \to \infty$ as $j \to \infty$, so that there exists $J_\tau < \infty$ such that $\tau_j > 1$ for all $j > J_\tau$. For $j > J_\tau$, we can bound the right hand side of \eqref{eq:subWeightSumD} by 
\beq
d_j(r,\eta,\tau_j) \leq 1+ (2^{6}  + \cdots + r^{6} +  C_{r,\eta}) \tau_j^{-4\eta}\;. 
\eeq
Consequently,  by setting $D_{r,\eta} = 2^{6}  + \cdots + r^{6} +  C_{r,\eta}$, we have
\beq\label{eq:bncneta2}
\sum_{\bsnu \in \cF} \left(\frac{b_\bsnu}{c_\bsnu^{1/\eta}}\right)^{-2\eta} \leq \prod_{j\geq 1} d_j(r,\eta,\tau_j) \leq \prod_{1\leq j \leq J_\tau} d_j(r,\eta,\tau_j) \prod_{j > J_\tau} (1+ D_{r,\eta} \tau_j^{-4\eta})\;,
\eeq
where the first term is bounded as $J_\tau < \infty$. The second term can be written as
\beq
\prod_{j > J_\tau} (1+ D_{r,\eta} \tau_j^{-4\eta}) = \exp\left(\sum_{j > J_\tau}\log\left(1+D_{r,\eta}\tau_j^{-4\eta}\right) \right) \;,
\eeq
which, by using  $\log(1+x) \leq x$ for all $x> -1$, can be bounded by 
\beq\label{eq:expDrq}
\exp\left(\sum_{j > J_\tau}\log\left(1+D_{r,\eta}\tau_j^{-4\eta}\right) \right) \leq \exp\left(D_{r,\eta} \sum_{j > J_\tau}  \tau_j^{-4\eta}\right),
\eeq
which is finite when $\eta \geq q/4$ since $(\tau_j^{-1})_{j\geq 1} \in \ell^q(\bbN)$ in Assumption \ref{ass:DeriBound}. Hence, \eqref{eq:bnucnu} is concluded by \eqref{eq:bncneta2} and \eqref{eq:expDrq}.
\end{proof}

We are at the point to state and prove the main theorem. The idea behind the proof is from the short discussion in \cite[Remark 5.1]{bachmayr2017sparse} and the result \cite[Lemma 2.9]{ZS17_723}.
\begin{theorem}\label{thm:N-termConv}
Under Assumption \ref{ass:Quadrature} and \ref{ass:DeriBound}, there exists an admissible index set $\Lambda_N \subset \cF$, a set of indices corresponding to the $N$ smallest value of $b_\bsnu$ defined in \eqref{eq:bnu}, such that the sparse quadrature error is bounded by
\beq\label{eq:quaderror}
||I(f)- \cQ_{\Lambda_N}(f)||_{\cS} \leq C (N+1)^{-s}, \quad s = \frac{1}{q} - \frac{1}{2}\;,
\eeq 
where the constant $C$ is independent of $N$.
\end{theorem}
\begin{proof}
We consider the right hand side of \eqref{eq:weightedSum} in Lemma \ref{lemma:WeightedSum}, which we can bound by multiplying and dividing $b_\bsnu^{-1/2+\eta}$ with $\eta \geq q/4$ as 
\beq\label{eq:cnufnubound}
\sum_{\bsnu \in \cF\setminus \Lambda_N}c_\bsnu ||f_\bsnu||_\cS \leq \sup_{\bsnu \in \cF\setminus \Lambda_N} b_\bsnu^{-1/2+\eta} \sum_{\bsnu \in \cF\setminus \Lambda_N} \frac{c_\bsnu}{b_\bsnu^{\eta}} b_\bsnu^{1/2} ||f_\bsnu||_\cS\;,
\eeq
where the second term can be bounded by using Cauchy--Schwarz inequality as 
\beq
\begin{split}
\sum_{\bsnu \in \cF\setminus \Lambda_N} \frac{c_\bsnu}{b_\bsnu^{\eta}} b_\bsnu^{1/2} ||f_\bsnu||_\cS 
&\leq 
\left(
\sum_{\bsnu \in \cF\setminus \Lambda_N} \left(\frac{c_\bsnu}{b_\bsnu^{\eta}}\right)^2
\right)^{1/2}
\left(
\sum_{\bsnu \in \cF\setminus \Lambda_N} b_\bsnu  ||f_\bsnu||_\cS^2
\right)^{1/2}\\
& \leq 
\left(
\sum_{\bsnu \in \cF} \left(\frac{b_\bsnu}{c_\bsnu^{1/\eta}}\right)^{-2\eta}
\right)^{1/2}
\left(
\sum_{\bsnu \in \cF} b_\bsnu  ||f_\bsnu||_\cS^2
\right)^{1/2}\;,
\end{split}
\eeq
which is finite as a result of Lemma \ref{lemm:cnubnu} for the first term and Assumption \ref{ass:DeriBound} and Proposition \ref{prop:EquivFourierDeri} for the second. By an increasing rearrangement of the sequence $(b_\bsnu)_{\bsnu \in \cF}$, which is equivalent to a decreasing rearrangement of $(b_\bsnu^{-1/2+\eta})_{\bsnu \in \cF}$ for $\eta < 1/2$, which we denote as $(d_n)_{n \geq 1}$, the first term on the right hand side of \eqref{eq:cnufnubound} becomes
\beq
\sup_{\bsnu \in \cF\setminus \Lambda_N} b_\bsnu^{-(1-2\eta)/2} = d_{N+1}^{}\;.
\eeq
Since $(b_\bsnu^{-1/2})_{\bsnu \in \cF} \in \ell^{q}(\cF)$ as given in Proposition \ref{prop:EquivFourierDeri}, so that $(d_n)_{n \geq 1} \in \ell^{q/(1-2\eta)}(\bbN)$. As a result, by taking $\eta = q/4$, the smallest value for $\eta$, we have $(d_n)_{n \geq 1} \in \ell^{\tilde{q}}(\bbN)$ where $\tilde{q} = 2q/(2-q) \in (0, \infty)$ for $q \in (0, 2)$. 
As $(d_n)_{n \geq 1}$ is monotonically decreasing, when $\tilde{q} \in (0, 1)$, by H\"older's inequality for $s = \frac{1}{\tilde{q}}$ and its conjugate $t = \frac{1}{1-\tilde{q}}$ we obtain
\beq
d_{n}^{\tilde{q}^2} \leq \frac{1}{n} \sum_{i = 1}^n d_i^{\tilde{q}^2} \leq \frac{1}{n} \left(\left(d_i^{\tilde{q}^2}\right)^{s}\right)^{\frac{1}{s}} \left(\sum_{i = 1}^n 1^{t}\right)^{\frac{1}{t}} = \left(\sum_{i = 1}^n d_i^{\tilde{q}} \right)^{\tilde{q}} n^{-\tilde{q}}, 
\eeq
so that 
\beq
d_{n} \leq  \left(\sum_{i = 1}^n d_i^{\tilde{q}}\right)^{\frac{1}{\tilde{q}}} n^{-s} \leq \left(\sum_{i = 1}^\infty d_i^{\tilde{q}}\right)^{\frac{1}{\tilde{q}}}  n^{-s},  
\eeq
For $\tilde{q} \geq 1$, again by H\"older's inequality for $\tilde{q}$ and its conjugate $t = \frac{1}{1-s}$ where $s = \frac{1}{\tilde{q}}$ we have 
\beq
d_n  \leq n^{-1} \sum_{i = 1}^n d_i \leq n^{-1} \left(\sum_{i=1}^n d_i^{\tilde{q}} \right)^{\frac{1}{\tilde{q}}} \left(\sum_{i=1}^n 1^{t} \right)^{1- s} \leq  \left(\sum_{i=1}^\infty d_i^{\tilde{q}} \right)^{\frac{1}{\tilde{q}}} n^{-s}. 
\eeq
Consequently, the main result \eqref{eq:quaderror} holds with the constant 
\beq
C = \left(\sum_{i=1}^\infty d_i^{\tilde{q}} \right)^{\frac{1}{\tilde{q}}} \left(
\sum_{\bsnu \in \cF} \left(\frac{b_\bsnu}{c_\bsnu^{1/\eta}}\right)^{-2\eta}
\right)^{1/2}
\left(
\sum_{\bsnu \in \cF} b_\bsnu  ||f_\bsnu||_\cS^2
\right)^{1/2} < \infty,
\eeq 
which is independent of $N$. 
To conclude the proof, we need to show that the index set $\Lambda_N$ can be taken such that it is admissible, for which we only need to verify that for any $k \in \bbN$ and $\bsnu \in \cF$, we have 
\beq\label{eq:bnuekgeqbnu}
b_{\bsnu+\bse_k} \geq b_{\bsnu}\;.
\eeq
This is true by the definition of $b_{\bsnu}$ in \eqref{eq:bnu}, i.e., for Kronecker delta $\delta_{jk}$,
\beq
b_{\bsnu+\bse_k} 
= \sum_{|\bsmu|_\infty\leq r}  \prod_{j \geq 1}
\left(
\begin{array}{cc}
\nu_j +\delta_{jk} \\
\mu_j
\end{array}
\right) \tau_j^{2\mu_j}
\geq
\sum_{|\bsmu|_\infty\leq r}  \prod_{j \geq 1}
\left(
\begin{array}{cc}
\nu_j \\
\mu_j
\end{array}
\right) \tau_j^{2\mu_j} = b_\bsnu\;.
\eeq
\end{proof}

\begin{remark}
The convergence of the quadrature error with respect to the number of indices does not depend on the number of the parameter dimensions, thus breaking the curse of dimensionality. It only depends on the summability parameter $q$, which measures the sparsity of the parametric function with respect to the parameters: the smaller $q$ is, the sparser $f$ is, the faster the convergence becomes. 
\end{remark}

\begin{remark}\label{rmk:ConvRate}
For any parametric function satisfying the Assumption \ref{ass:DeriBound}, our theorem implies that we can construct the admissible index set completely based on the definition of $b_\bsnu$ in \eqref{eq:bnu} in order to achieve the convergence rate $N^{-s}$ with $s = 1/q - 1/2$. This convergence rate is obtained as an upper bound, which is not necessarily optimal. 
In fact, our numerical tests indicate that it could be improved.
\end{remark}

The convergence rate is obtained with respect to the number of indices $N$ in the index set $\Lambda_N$, which is not necessarily the same as the number of quadrature points. The following corollary provides a convergence rate with respect to the number of quadrature points in the case of Gauss--Hermite quadrature with $m_l = l + 1$.

\begin{corollary}\label{cor:index2point}
As a result of Theorem \ref{thm:N-termConv}, for the case of Gauss--Hermite quadrature with $m_l = l + 1$, the sparse quadrature error is bounded by 
\beq
||I(f)- \cQ_{\Lambda_N}(f)||_{\cS} \leq C N_p^{-s/2}, \quad s = \frac{1}{q} - \frac{1}{2}\;,
\eeq
where $C$ is independent of the number of quadrature points $N_p$ corresponding to $\Lambda_N$.
\end{corollary}
\begin{proof}
The bound is a result of \cite[Proposition 18]{ernst2016convergence}, which states that there exists a constant $C$ such that $N_p \leq  C N^2$. 
\end{proof}

\begin{remark}
Similar convergence rates are observed in practice for both GH1, GH2, and GK with respect to the number of quadrature points as that of indices, as shown in our numerical tests. The reason might be that $\cQ_l$, which uses $m_l$ quadrature points, is exact at least for $P_{n}$ with $n \geq m_l-1$ (in fact it is exact for $P_{2m_l-1}$ by GH quadrature), which is much richer than $P_{l}$.
\end{remark}


\subsection{Examples}
\label{subsec:Examples}
The dimension-independent convergence rate relies on the assumption on the derivatives of the function $f(\bsy)$ with respect to the parameter $\bsy$ as stated in Assumption \ref{ass:DeriBound}. Here we provide two examples which satisfy such assumption. For both examples, we assume a common structure that the function $f$ depends on $\bsy$ through $\kappa(\bsy)$ as $f(\kappa(\bsy))$, where $\kappa$ is given by
\beq\label{eq:kappa}
\kappa(\bsy) = \sum_{j \geq 1} y_j \psi_j\;,
\eeq
where we assume $\max_{j\geq 1}||\psi_j|| < \infty$, e.g., $||\psi_j|| = |\psi_j|$ if $\psi \in \bbR$ and $||\psi_j||  = ||\psi_j||_{L^\infty(D)}$ if $\psi_j$ is a function in a physical domain $D$. 
\subsubsection{Example 1 -- A nonlinear parametric function}
\label{sec:ex1}
We first consider a function that does not depend on the physical coordinate $x$, where we set $\psi_j = j^{-\alpha}$ in $\kappa$,  in particular, 
\beq\label{eq:fbsy}
f(\bsy) = f(\kappa(\bsy)) = \exp\left(\sum_{j \geq 1} y_j j^{-\alpha} \right), \quad \alpha > 1\;.
\eeq 
To satisfy Assumption \ref{ass:DeriBound}, we compute 
\beq
\sum_{|\bsmu|_\infty \leq r} \frac{\bstau^{2\bsmu}}{\bsmu!} \int_{Y} |\partial^\bsmu_\bsy f(\bsy)|^2 d\bsgamma(\bsy) = \int_Y f^2(\bsy) d\bsgamma(\bsy) \sum_{|\bsmu|_\infty \leq r} \frac{\bstau^{2\bsmu}}{\bsmu!} \prod_{j \geq 1}j^{-2\alpha \mu_j} \;,
\eeq
where, for $\alpha > 1$, we have 
\beq
\int_Y f^2(\bsy) d\bsgamma(\bsy) = \prod_{j\geq 1} \int_{\bbR} e^{2j^{-\alpha}y_j} \rho(y_j)dy_j = \exp\left(2\sum_{j\geq 1} j^{-2\alpha}\right) < \infty\;.
\eeq
Moreover, we have the bound (by using $1+x+\cdots + x^r/r!  < e^x$ for any $x > 0$)
\beq
\sum_{|\bsmu|_\infty \leq r} \frac{\bstau^{2\bsmu}}{\bsmu!} \prod_{j \geq 1}j^{-2\alpha \mu_j}  = \prod_{j\geq 1} \left(\sum_{l=0}^r \frac{(\tau_j^{2} j^{-2\alpha})^l}{l!}\right) \leq \exp\left(\sum_{j\geq 1} \tau_j^{2} j^{-2\alpha}\right)\;,
\eeq
which is finite if and only if $\tau_j \lesssim j^{\alpha - 1/2 -\varepsilon}$ for arbitrary $\varepsilon > 0$, so that $(\tau_j^{-1})_{j \geq 1} \in \ell^q(\bbN)$ for $q \geq 1/(\alpha - 1/2-\varepsilon)$. By Theorem \ref{thm:N-termConv}, we obtain the convergence rate $N^{-s}$ for $s = 1/q - 1/2 \leq \alpha - 1 -\varepsilon$.
Note that the case $\alpha \leq 1$ is not covered by the theorem.

\subsubsection{Example 2 -- PDE solution as a nonlinear map}
\label{sec:ex3}
We consider the solution (nonlinear with respect to $\kappa$) of the diffusion equation: find $u(\bsy)\in H^1_0(D)$ such that 
\beq
- \text{div}(e^{\kappa(\bsy)} \nabla u(\bsy)) = g,  \text{ in } D\;,
\eeq
with homogeneous Dirichlet boundary condition, and $g \in H^{-1}(D)$. 
This example is studied in detail in \cite{bachmayr2017sparse}. Under the parametrization \eqref{eq:kappa}, for $(\tau_j)_{j\geq 1}$ such that
\beq\label{eq:taupsi}
\sup_{x\in D} \sum_{j \geq 1} \tau_j |\psi_j(x)| < \frac{\text{ln} 2}{\sqrt{r}}\;,
\eeq
and $(\tau_j^{-1})_{j\geq 1} \in \ell^q(\bbN)$ for any $0 < q < \infty$, they proved the bound \cite[Theorem 4.2]{bachmayr2017sparse}
\beq\label{eq:DiffusionBound}
\sum_{|\bsmu|_\infty \leq r} \frac{\bstau^{2\bsmu}}{\bsmu!} \int_{Y} ||\partial^\bsmu_\bsy u(\bsy)||^2_\cS d\bsgamma(\bsy) \leq C\int_{Y} \exp(4||\kappa(\bsy)||_{L^{\infty}(D)}) d\bsgamma(\bsy) < \infty\;,
\eeq
where $\cS =H^1_0(D)$, $C$ is a constant independent of $\bsy$.  The first inequality is ensured by \eqref{eq:taupsi} from a careful estimate of the partial derivatives of $u$ with respect to $\bsy$ and the sum of their integrals, while the second inequality is ensured by $(\tau_j^{-1})_{j\geq 1} \in \ell^q(\bbN)$. Then the convergence rate $N^{-s}$ with $s = 1/q - 1/2$ in Theorem \ref{thm:N-termConv} is established for $f(\bsy) = u(\bsy)$. Note that in \cite{bachmayr2017sparse} only $r > 2/q$ is needed for the convergence result of a Hermite polynomial approximation error, while we need $r > 14/q$ for the convergence of the sparse quadrature error due to the proof in Lemma \ref{lemm:cnubnu}.

Here, the solution $u(\bsy)$ can be replaced by a bounded linear functional of $f(\bsy) = f(u(\bsy))$, and the inequality \eqref{eq:DiffusionBound} can be verified for $f$ due to 
\beq\label{eq:linearf}
|\partial^\bsmu_\bsy f(\bsy)|^2 \leq ||f||^2_{\cS'} ||\partial^\bsmu_\bsy u(\bsy)||^2_\cS.
\eeq 

\section{Construction of the sparse quadrature}
\label{sec:construction}

We present two algorithms for the construction of the sparse quadrature -- one is a-priori construction that guarantees the dimension-independent convergence rate in Theorem \ref{thm:N-termConv}; the other is a goal-oriented a-posteriori construction based on a-posteriori error indicator -- the difference quadrature $\triangle_\bsnu(f)$ in \eqref{eq:TensorDiff} that depends on each specific function $f$, which however can not guarantee the dimension-independent convergence rate in theory but achieve so in our numerical experiments in Sec \ref{sec:numerics}. 
 
\subsection{A-priori construction}
\label{sec:aprioricons}
A-priori construction of sparse grids has been considered in the literature, e.g., in \cite{ma2009adaptive, beck2014quasi}. In our setting,
from Theorem \ref{thm:N-termConv} we observe that the dimension-independent convergence rate of the sparse quadrature can be achieved by choosing the admissible index set $\Lambda_N$ with indices $\bsnu \in \cF$ corresponding to the largest value of $b_{\bsnu}$. While we can compute $b_\bsnu$ for all the indices $\bsnu \in \cF_{r,J}$ where 
\beq
 \cF_{r,J}= \{\bsnu \in \cF: |\bsnu|_{\infty} \leq r, \text{ and } \nu_j = 0 \text{ for } j > J \}\;,
\eeq
it is expensive/unfeasible if $r$ and $J$ are very large or infinite. For a feasible construction, we first arrange $(\tau_j)_{j\geq 1}$ to be in increasing order. Then, thanks to the monotonic increasing property of $b_\bsnu$ in \eqref{eq:bnuekgeqbnu}, we can adaptively construct the admissible index set $\Lambda_N$ by Algorithm \ref{alg:SparseQuad} (with candidate indices from a forward neighbor index set, see \eqref{eq:ForwNeib} ahead). 
Note that even for indices that cannot be sorted
in lexicographic order, e.g., $\bsnu = (2, 1)$ and $\bsmu = (1, 2)$,  $b_\bsnu > b_\bsmu$ due to the reordering just introduced. This implies that the a-priori construction, that iteratively explores variables one after the other (see again \eqref{eq:ForwNeib}
ahead), will never miss the largest index still not included in the set, which guarantees that the convergence rate predicted by theory will be attained. We explain this algorithm in detail in the next section.

We remark that this a-priori construction depends only on the parameters $q$, $\bstau$ and $r$ in Assumption \ref{ass:DeriBound} for any function satisfying such assumption.  However, it is not always straightforward or possible to verify this assumption especially for nonlinear function with respect to the parameter as in Example 2. In this situation, and in the common parametrization as in \eqref{eq:kappa}, we use $\tau_j = j^{\alpha - 1}$ when $||\psi_j||$ decays as $j^{-\alpha}$ as demonstrated in Section 5.2 (see Fig. \ref{fig:CompareTau}), and choose $r = \text{floor}(14(\alpha-1))+1$, the closest integer larger than $14(\alpha-1)$ according to Assumption \ref{ass:DeriBound}. Alternatively, we turn to a goal-oriented a-posteriori construction that does not need $q$, $\bstau$ and $r$. 

\subsection{Goal-oriented a-posteriori construction}
\label{sec:aposteriori}
We present a goal-oriented a-posteriori construction of the sparse quadrature based on a dimension-adaptive tensor-product quadrature initially developed in \cite{gerstner2003dimension} 
which we 
call \emph{adaptive sparse quadrature}, whose associated grids $G_{\Lambda}$ is called \emph{adaptive sparse grids}. The basic idea is based on the following adaptive process: given an admissible index set $\Lambda$, we search an index $\bsnu \in \cF$ among the forward neighbors of $\Lambda$ ($\bsnu \in \cF $ is called a forward neighbor of $\Lambda$ if $\Lambda \cup \bsnu$ is still admissible), at which $||\triangle_\bsnu||_\cS$ is maximized, and add this index to the index set $\Lambda = \Lambda \cup \{\bsnu\}$. As the number of forward neighbors depends on the dimension $J$ (in fact, the forward neighbors of $\bsnull$ are $\bse_j$ for all $j$), it is not feasible to search over all the forward neighbors in high or infinite dimensions. In such cases, it is usually reasonable to assume that the dimensions with small indices are more important than those with large indcies, as determined, e.g., by the decaying eigenvalues in Karhunen--Lo\`eve representation of a random field. Therefore, we can explore the forward neighbors dimension by dimension in the set (see, e.g., \cite{Schillings2013, chkifa2014high})
\beq\label{eq:ForwNeib}
\cN(\Lambda) := \{\bsnu \not \in \Lambda: \bsnu - \bse_j \in \Lambda, \forall j \in \bbJ_\bsnu \text{ and } \nu_j = 0\;, \forall j > j(\Lambda)+1\},
\eeq
where $\bbJ_\bsnu = \{j: \nu_j \neq 0\}$; $j(\Lambda)$ is the smallest $j$ such that $\nu_{j+1} = 0$ for all $\bsnu \in \Lambda$. More generally, $j(\Lambda) + K$ for a certain $K \geq 1$ can be used, see \cite{nobile2016adaptive}.

The adaptive sparse quadrature can be constructed following a basic greedy algorithm proposed in \cite{gerstner2003dimension}, which was improved on the data structure in \cite{klimke2006uncertainty} to cope with very high dimensions (e.g., upto $10^4$ dimensions in a personal laptop with $16$GB memory). We present the goal-oriented a-posteriori construction also in Algorithm \ref{alg:SparseQuad}. 

\begin{algorithm}
\caption{Adaptive sparse quadrature}
\label{alg:SparseQuad}
\begin{algorithmic}[1]
\STATE{\textbf{Input: } 
maximum number of indices $N_{\text{max}}$, function $f$.}
\STATE{\textbf{Output: } the admissible index set $\Lambda_N$, quadrature $\cQ_{\Lambda_N}(f)$.}
\STATE{Set $N= 1$, $\Lambda_N = \{\bsnull\}$, evaluate $f(\bsnull)$ and set $
\cQ_{\Lambda_N}(f)=f(\bsnull)$.}
\WHILE{$N < N_{\text{max}}$}
\STATE{Construct the forward neighbor set $\cN(\Lambda_N)$ by \eqref{eq:ForwNeib}.}
\IF {a-priori construction}
\STATE{Compute $b_{\bsnu}$ for all $\bsnu \in \cN(\Lambda_N)$ by \eqref{eq:bnu}.}
\STATE{Take $\bsnu = \argmin_{\bsmu \in \cN(\Lambda_N)} b_\bsnu$.}
\ELSE 
\STATE{Compute $\triangle_\bsnu(f)$ for all $\bsnu \in \cN(\Lambda_N)$ by \eqref{eq:TensorQuad}.}
\STATE{Take $\bsnu = \argmax_{\bsmu \in \cN(\Lambda_N)} ||\triangle_{\bsmu}(f)||_\cS$.}
\ENDIF
\STATE{Enrich the index set $\Lambda_{N+1} = \Lambda_N \cup \{\bsnu\}$}.
\STATE{Set $\cQ_{\Lambda_{N+1}}(f) = \cQ_{\Lambda_N}(f) + \triangle_{\bsnu}(f)$.}
\STATE{Set $ N \leftarrow N + 1$.}
\ENDWHILE
\end{algorithmic}
\end{algorithm}

\begin{remark}
Instead of using the maximum number of indices as the stopping criterion, 
we can use some others, such as the maximum number of points, or an heuristic error indicator $||\sum_{\bsmu \in \cN(\Lambda_N)}\triangle_{\bsmu}(f)||_\cS$, or $b_{\bsnu}^{-(2-q)/4q}$ for the a-priori construction. Moreover, for the a-posteriori construction, it is also a common practice to chose $\bsnu$ as $\bsnu = \argmax_{\bsmu \in \cN(\Lambda_N)} ||\triangle_{\bsmu}(f)||_\cS /|G_\bsmu|$ to balance the error and the work, e.g., \cite{gerstner2003dimension, nobile2016adaptive}. We caution that these heuristic error indicators are not rigorous and may lead to early stop of the algorithm in the case that $||\triangle_{\bsmu}(f)||_\cS$ is critically small for all $\bsmu$ in $\cN(\Lambda_N)$, which can be possibly addressed by a verification process \cite{chen2015new}. 
\end{remark}

\begin{remark}
Note that to construct $\Lambda_N$, we need to evaluate the function $f$ at all quadrature points corresponding to $\cN(\Lambda_N)$ by the a-posteriori construction, so that the total number of function evaluations is larger than that in $\Lambda_N$ as presented in Corollary \ref{cor:index2point}. We will also investigate the convergence rate with respect to the total number of quadrature points in the numerical experiments.
\end{remark}

\section{Numerical experiments}
\label{sec:numerics}
In this section, we present two numerical experiments for a parametric function and a parametric PDE to demonstrate the convergence property of the sparse quadrature using different univariate quadrature rules and different construction schemes in comparison with the Monte Carlo quadrature.

\label{sec:Numerics}
\subsection{A parametric function}
\label{subsec:function}
We first consider the nonlinear parametric function presented in Example 1, Sec \ref{sec:ex1}. The expectation of the function is given analytically, which is 
\beq
I(f) = \exp\left( \frac{1}{2} \zeta(2\alpha) \right)\;,
\eeq 
where $ \zeta(2\alpha) = \sum_{j\geq 1} j^{-2\alpha}$ is the Riemann zeta function. We compute it by truncation of $j$ at $10^4$ dimensions and use it as the reference value.
We run Algorithm \ref{alg:SparseQuad} for the construction of the sparse quadrature with both the a-priori construction in Sec. \ref{sec:aprioricons}, and the goal-oriented a-posteriori construction in Sec. \ref{sec:aposteriori}. For the former, we use $\tau_j = j^{\alpha - 1/2}$, as obtained in Example 1, for the computation of $b_\bsnu$ in \eqref{eq:bnu}. 
We set 
the maximum number of sparse grid points at $10^5$. The forward neighbor index set \eqref{eq:ForwNeib} is used since $\tau_j$ is monotonically increasing. We test the four quadrature rules: 1) Gauss--Hermite rule with $m_l = l+1$ (GH1 for short); 2) Gauss--Hermite rule with $m_l = 2^{l+1}-1$ (GH2); 3) transformed Gauss--Kronrod--Patterson rule (tGKP) with maximum level $l = 6$; 4) Genz--Keister rule (GK) with maximum level $l = 4$.

\begin{figure}[!htb]
\begin{center}
\includegraphics[scale=0.34]{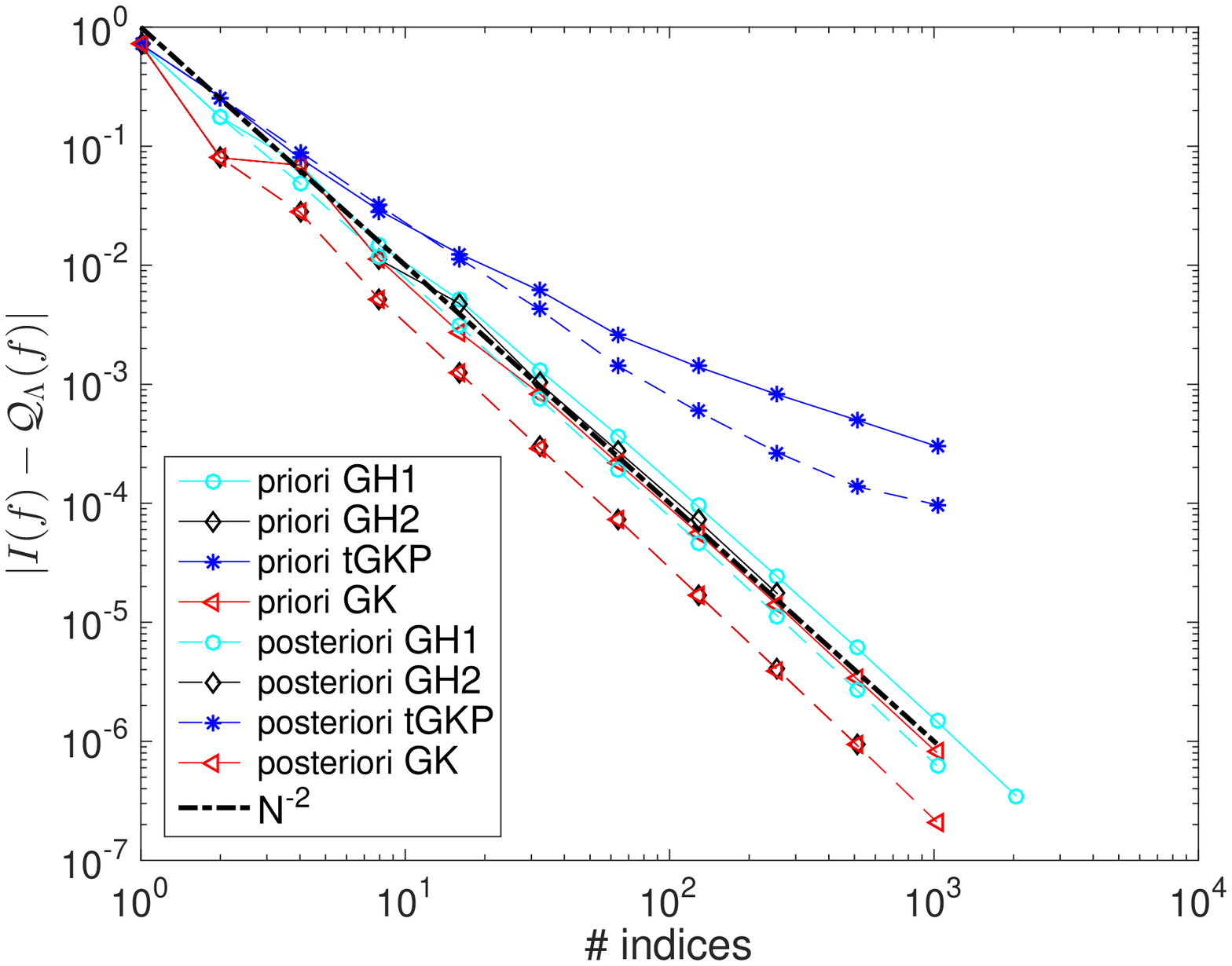}
\hspace*{0.2cm}
\includegraphics[scale=0.34]{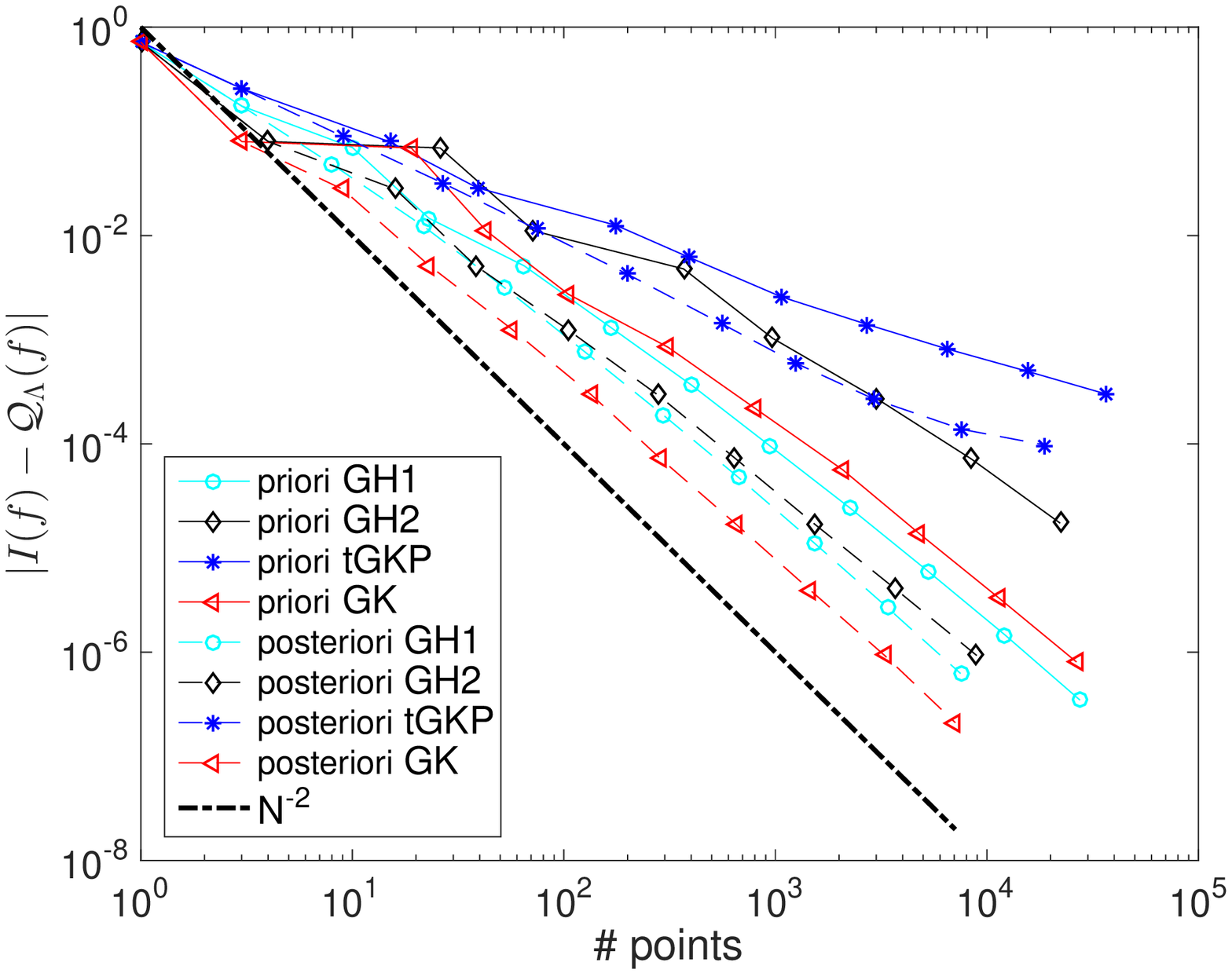}

%
\end{center}
\caption{Decay of quadrature errors $|I(f) - \cQ_{\Lambda}(f)|$ with respect to the number of indices (left) and the number of points (right) in $\Lambda$. Reported are for the different quadrature rules constructed by both the a-priori and the a-posteriori schemes with Algorithm \ref{alg:SparseQuad}. Here, $\alpha = 2$.}\label{fig:OldPrioriPosteriori}
\end{figure}

Figure \ref{fig:OldPrioriPosteriori} displays the decay of the quadrature errors with respect to the number of indices $|\Lambda|$ and the number of sparse grid points (function evaluations) $|G_\Lambda|$ in $\Lambda$. We can observe a dimension-independent convergence rate of the quadrature error, not only with respect to the number of indices as predicted by Theorem \ref{thm:N-termConv}, but also with respect to the number of points. 
Note that the convergence rate obtained is indeed dimension-independent, since only part of the dimensions at
disposal have been activated as observed in Fig. \ref{fig:level}: in other words, had we considered even more than the
current $10^4$ random variables, possibly countably many, we would have observed the same convergence
curve.
It is evident from the comparison that both the a-priori and the a-posteriori construction schemes lead to very close convergence rates for the quadrature rules GH1, GH2 and GK, while the a-posteriori construction gives smaller quadrature errors at the same number of indices/points for all four quadrature rules. 

The numerical convergence rate with respect to the number of indices is about $N^{-s}$ for GH1, GH2, and GK, with $s = 2$ for $\alpha = 2$, which is faster than that predicted by Theorem \ref{thm:N-termConv} at $s = \alpha -1$. This indicates that the convergence rate obtained in Theorem \ref{thm:N-termConv} is possibly not optimal. 
Note that the convergence is sightly slower than $N^{-2}$ with respect to the number of points, which is due to the larger number of points than the number of indices. The performance of GH1, GH2, and GK are very close: the errors of GH2 and GK overlap with respect to the number of indices while the latter is smaller than the former with respect to the number of points, because GK points are nested while GH2 (also GH1) points are not. On the other hand, it is shown that tGKP does not converge as fast as the other three rules and gets stagnated for a large number of indices and points. This is due to the fact that the degree of exactness of tGKP is much smaller than the others; in particular, it does not satisfy A.1 of Assumption \ref{ass:Quadrature} as shown in Fig \ref{fig:QuadAccu}.

\begin{figure}[!htb]
\begin{center}
\includegraphics[scale=0.34]{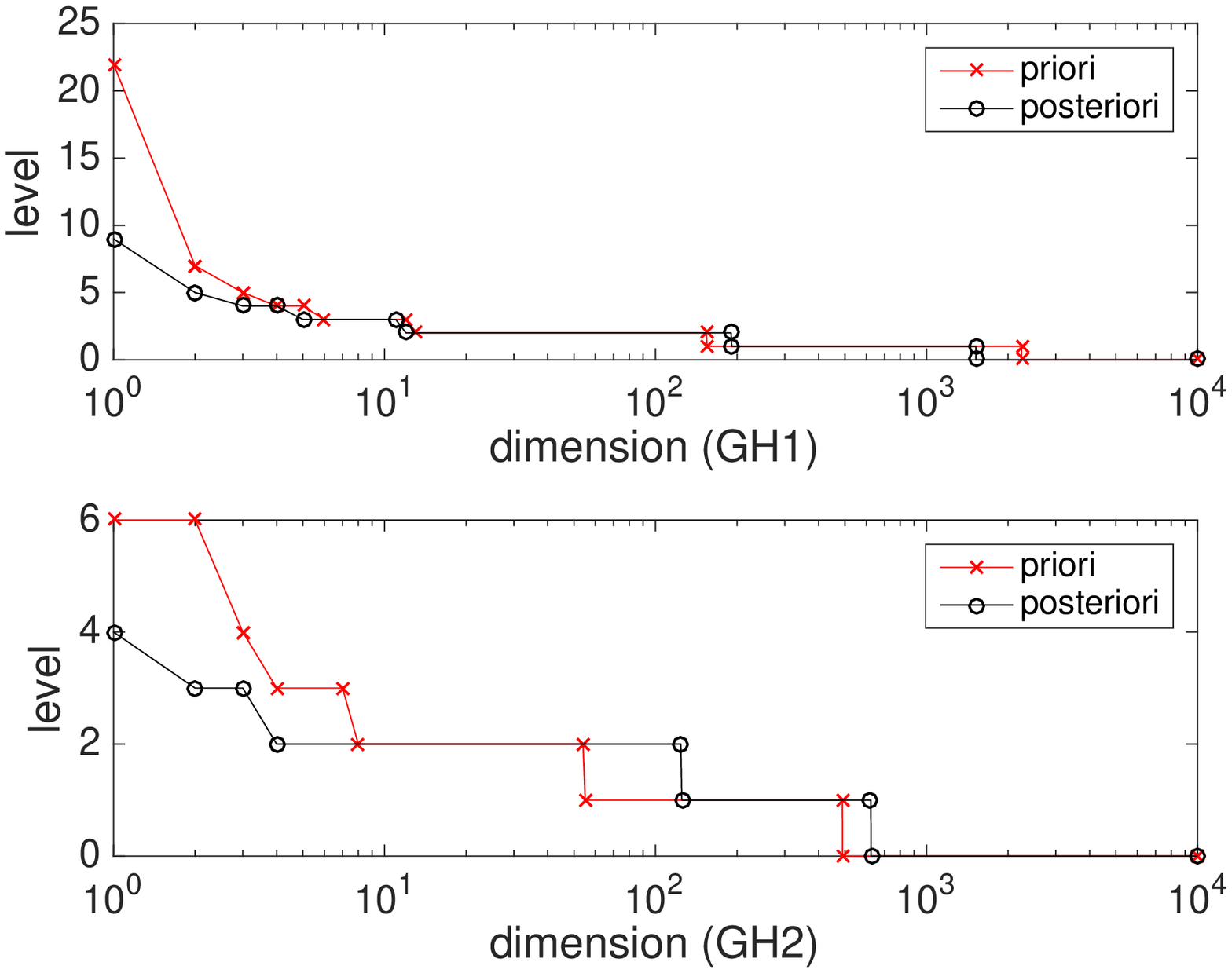}
\hspace*{0.5cm}
\includegraphics[scale=0.34]{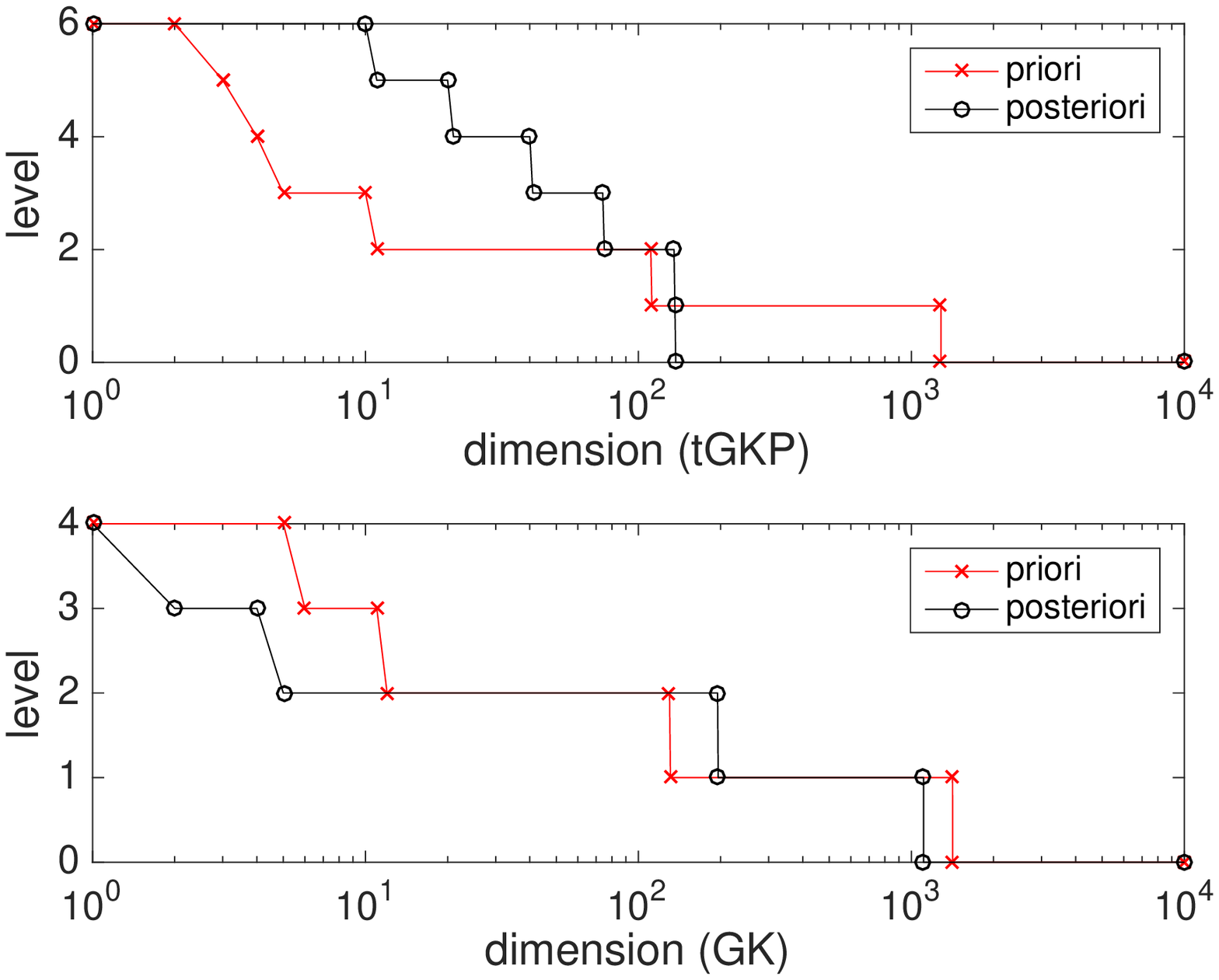}

%
\end{center}
\caption{Maximum level ($\max_{\bsnu \in \Lambda\cup \cN(\Lambda)} \nu_j$, $j = 1, \dots, 10^4$) in each of the $10^4$ dimensions constructed by the a-priori and the a-posteriori schemes for the four quadrature rules. $\alpha = 2$.}\label{fig:level}
\end{figure}

The sparse grid level $l$ for the two construction schemes with the four quadrature rules is displayed in Fig. \ref{fig:level}. Note that we have set the maximum level for GH2 and tGKP as $6$, and for GK as $4$ due to the availability of the quadrature points (for tGKP and GK). The a-priori construction tends to use higher levels for the first few dimensions than the a-posteriori construction for GH1, GH2, and GK, and gives rise to the larger number of points that become useless because of the high exactness of the GH and GK quadrature rules (see the early divergence of the errors in the right part of Fig. \ref{fig:OldPrioriPosteriori}). This high exactness is explored and benefited by the a-posteriori construction. On the other hand, the low exactness of the tGKP is not seen by the a-priori construction but by the a-posteriori, see the different levels for tGKP in Fig. \ref{fig:level}. Moreover, the a-priori construction leads to less accurate quadrature results compared to the a-posteriori construction, especially for GH2, GK, and tGKP as the number of these quadrature points double from one level to the next. As for GH1, the a-priori construction is very close to the a-posteriori construction in terms of accuracy. This is because only one quadrature point is added from one level to the next, so that the number of indices and the number of quadrature points are closer than those for the other three quadrature rules. Note that the a-priori construction is performed completely based on the quantity $b_\bsnu$ in \eqref{eq:bnu}, which only depends on the index for fixed $(\tau_j)_{j\geq 1}$, regardless of how many quadrature points are used in the same index set.

The convergence rates have been investigated with respect to the number of indices and points in $\Lambda$ to demonstrate the results in Theorem \ref{thm:N-termConv}. However, in order to construct $\Lambda$, the indices in its forward neighbor set $\cN(\Lambda)$ (see the definition \eqref{eq:ForwNeib}) have to be searched over. Hence, we need to evaluate the function at each quadrature point in $\cN(\Lambda)$ by the a-posteriori construction, or evaluate $b_{\bsnu}$ (defined in \eqref{eq:bnu}) by the a-priori construction. Here we emphasize that the computational cost for the evaluation of $b_{\bsnu}$ could be negligible compared to that of the function evaluation which requires, e.g., PDE solve, so that the a-priori construction is potentially more efficient than the a-posteriori. 
For instance, here 30601 function evaluations are performed out of 100500 points (the remaining points are in the forward neighbor set $\cN(\Lambda)$) by GH1 quadrature rule.

To investigate the convergence rate with respect to the total number of indices and points in $\bar{\Lambda} = \Lambda \cup \cN(\Lambda)$, which represents the total computational cost, we compute the quadrature error $|I(f) - \cQ_{\bar{\Lambda}}(f)|$ for the GK rule with $\alpha = 1, 2, 3$. We also compute the Monte Carlo quadrature error by an average of $100$ trials for all $\alpha$ in $10^3$ dimensions. The quadrature errors are reported in Fig. \ref{fig:AllPrioriPosteriori}. 

\begin{figure}[!htb]
\begin{center}
\includegraphics[scale=0.34]{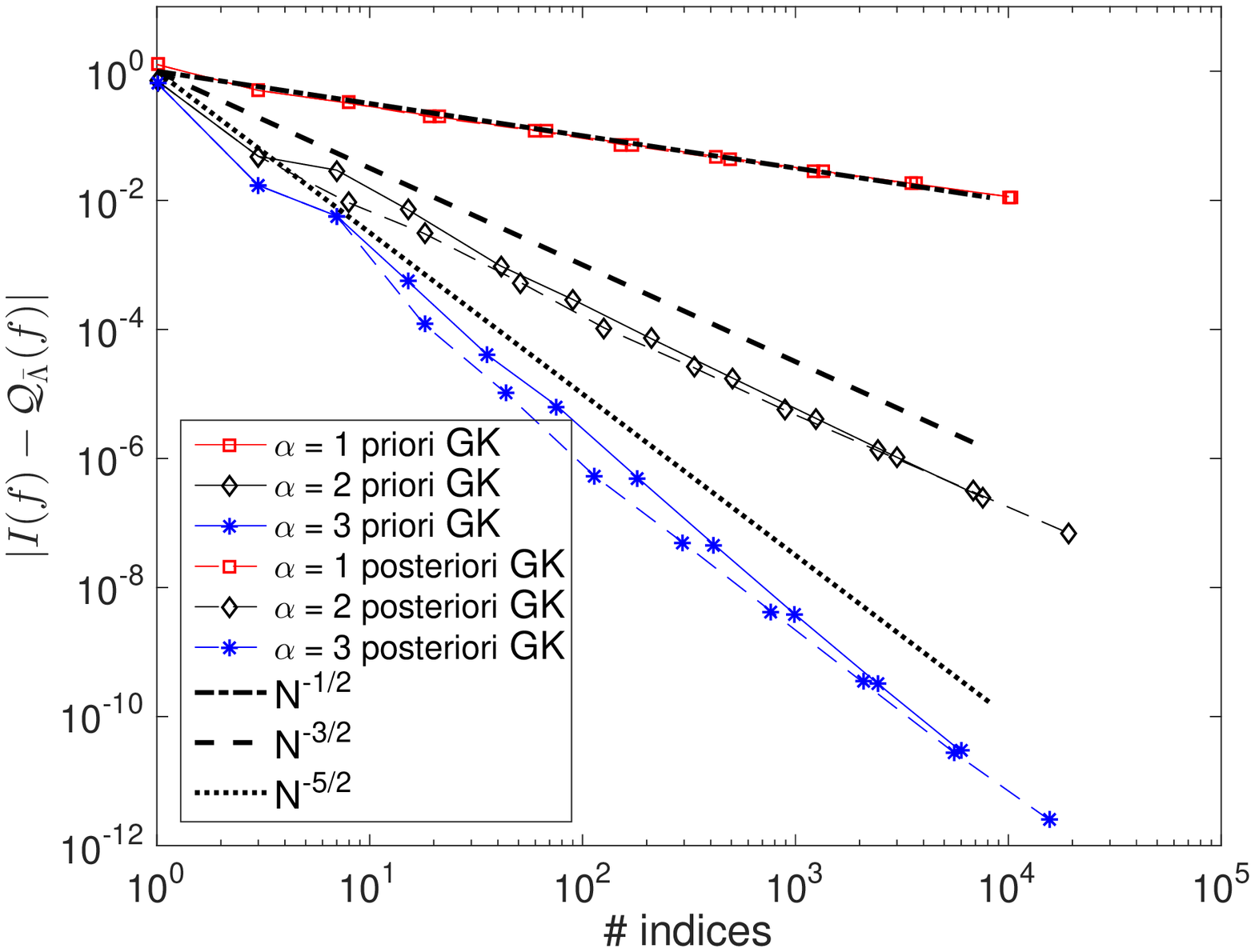}
\hspace*{0.2cm}
\includegraphics[scale=0.34]{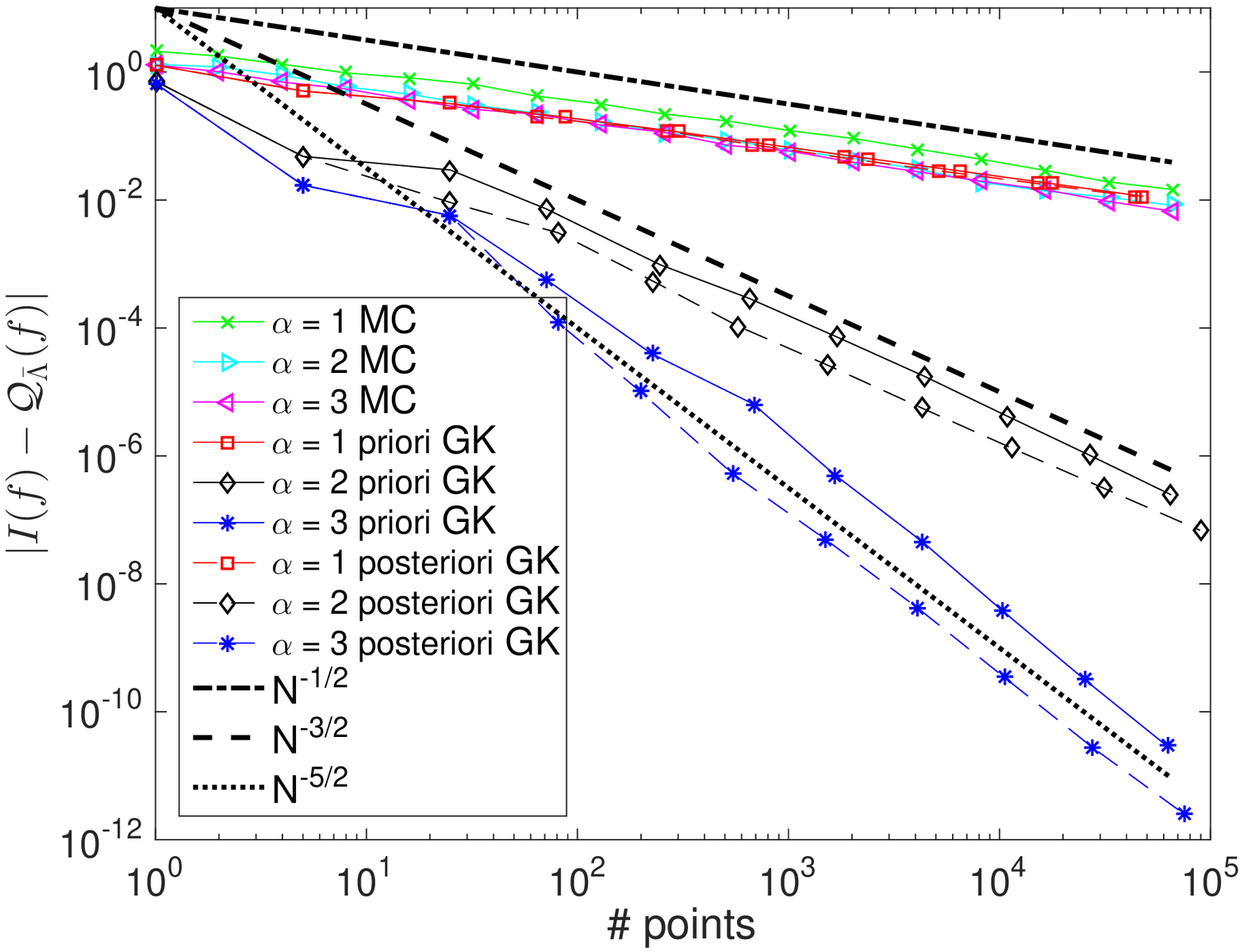}
\end{center}
\caption{Decay of quadrature errors $|I(f) - \cQ_{\bar{\Lambda}}(f)|$ with respect to the number of indices (left) and the number of points (right) in $\bar{\Lambda} = \Lambda \cup \cN(\Lambda)$. Reported are for the Monte Carlo (MC) and the GK quadrature rules constructed by both the a-priori and the a-posteriori construction schemes.}\label{fig:AllPrioriPosteriori}
\end{figure}

We can observe that the convergence rates of the quadrature errors with respect to both the total number of indices and the total number of points corresponding to the union set $\bar{\Lambda}$ are about $N^{-s}$, where $s = \alpha - 1/2$ for all $\alpha = 1, 2, 3$, by both the a-priori and the a-posteriori construction schemes. Meanwhile, the average of Monte Carlo (MC) quadrature errors decays as $N^{-1/2}$ for all $\alpha$, which is much slower than that of the sparse quadrature errors for $\alpha = 2, 3$. In the case $\alpha = 1$, the sparse quadrature still achieves very close convergence rate as $N^{-1/2}$ for MC and with smaller errors in this test example, see in the right part of Fig. \ref{fig:AllPrioriPosteriori}. Note that the MC quadrature error is measured in average/expectation, which could be much less accurate depending on the trial, while the sparse quadrature error is deterministically bounded.

\subsection{A parametric PDE}
\label{subsec:PDE}
In this section, we consider the parametric PDE of Example 2 in Sec. \ref{sec:ex3}, where the coefficient $\kappa$ is a Gaussian random field allowing the Karhunen--Lo\`eve expansion
\beq\label{eq:KLExpansion}
\kappa = \kappa_0 + \sum_{j \geq 1} \sqrt{\lambda_j} \phi_j y_j\;,
\eeq
where $(\lambda_j, \phi_j)_{j\geq 1}$ are the eigenpairs of $(-\delta \triangle)^{-\alpha}$, $\delta, \alpha > 0$, with homogeneous Dirichelet boundary condition on the boundary $\partial D$ of the domain $D \in \bbR^d$, and $(y_j)_{j \geq 1}$ are i.i.d. standard Gaussian random variables. For the simple case $D = (0,1)$, we have for $\delta = 1/\pi^2$,
\beq 
\lambda_j = j^{-2\alpha}, \text{ and } \phi_j = \sin(\pi j x)\;.
\eeq
This monodimensional PDE problem under the above parametrization is well-posed under the condition $\alpha > 1/2$, see \cite[Assumption 3.1]{charrier2012strong}.
In the numerical test, we set $\kappa_0 = 0$, the forcing term $g = 1$, and prescribe zero Dirichlet boundary condition at $x=0, 1$. A uniform mesh with mesh size $h = 1/2^{10}$ is used for the discretization of the domain $D$, therefore we truncate $j$ with $J = 1023$ dimensions in the parametrization \eqref{eq:KLExpansion}. We use a finite element method with piecewise linear element to solve the elliptic PDE.
Under the parametrization \eqref{eq:KLExpansion}, our quantity of interest is the average value of $u$ in $D$ and we compute its first two moments, i.e., we compute $\bbE[f_1]$ and $\bbE[f_2]$, where 
\beq
f_1(\bsy) = Q(u(\bsy)) \text{ and } f_2(\bsy) = Q^2(u(\bsy)), \text{ where } Q(u(\bsy)) = \int_D u(\bsy) dx\;.
\eeq 

We construct the sparse quadrature by both the a-priori and the a-posteriori construction schemes presented in Algorithm \ref{alg:SparseQuad}. For the a-priori construction, to satisfy the condition \eqref{eq:taupsi} with $\psi_j = \sqrt{\lambda_j} \phi_j = j^{-\alpha} \sin(\pi j x)$, a choice of $\tau_j \propto j^{\alpha - 1 - \varepsilon}$ for arbitrary small $\varepsilon > 0$ is sufficient since 
\beq
\sup_{x \in D} \sum_{j \geq 1} \tau_j |\psi_j(x)| \leq \sum_{j\geq 1} \tau_j ||\psi_j||_{L^\infty(D)} =  \sum_{j\geq 1} \tau_j j^{-\alpha}\;.
\eeq 
Here, we set $\tau_j = j^{\alpha - 1}$ with $\alpha = 2$. To run Algorithm \ref{alg:SparseQuad}, we set the maximum number of sparse grid points set as $10^5$.

\begin{figure}[!htb]
\begin{center}
\includegraphics[scale=0.34]{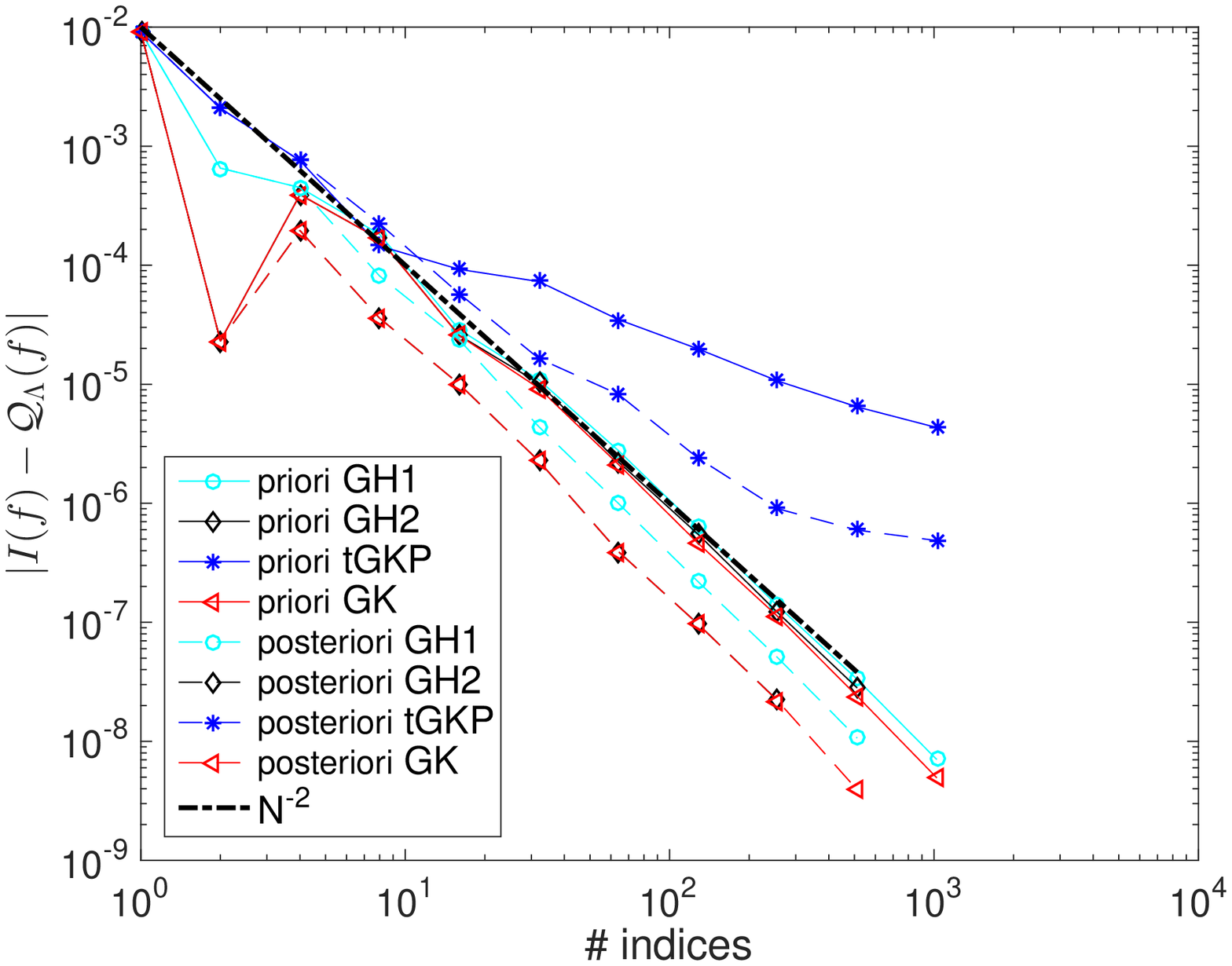}
\hspace*{0.2cm}
\includegraphics[scale=0.34]{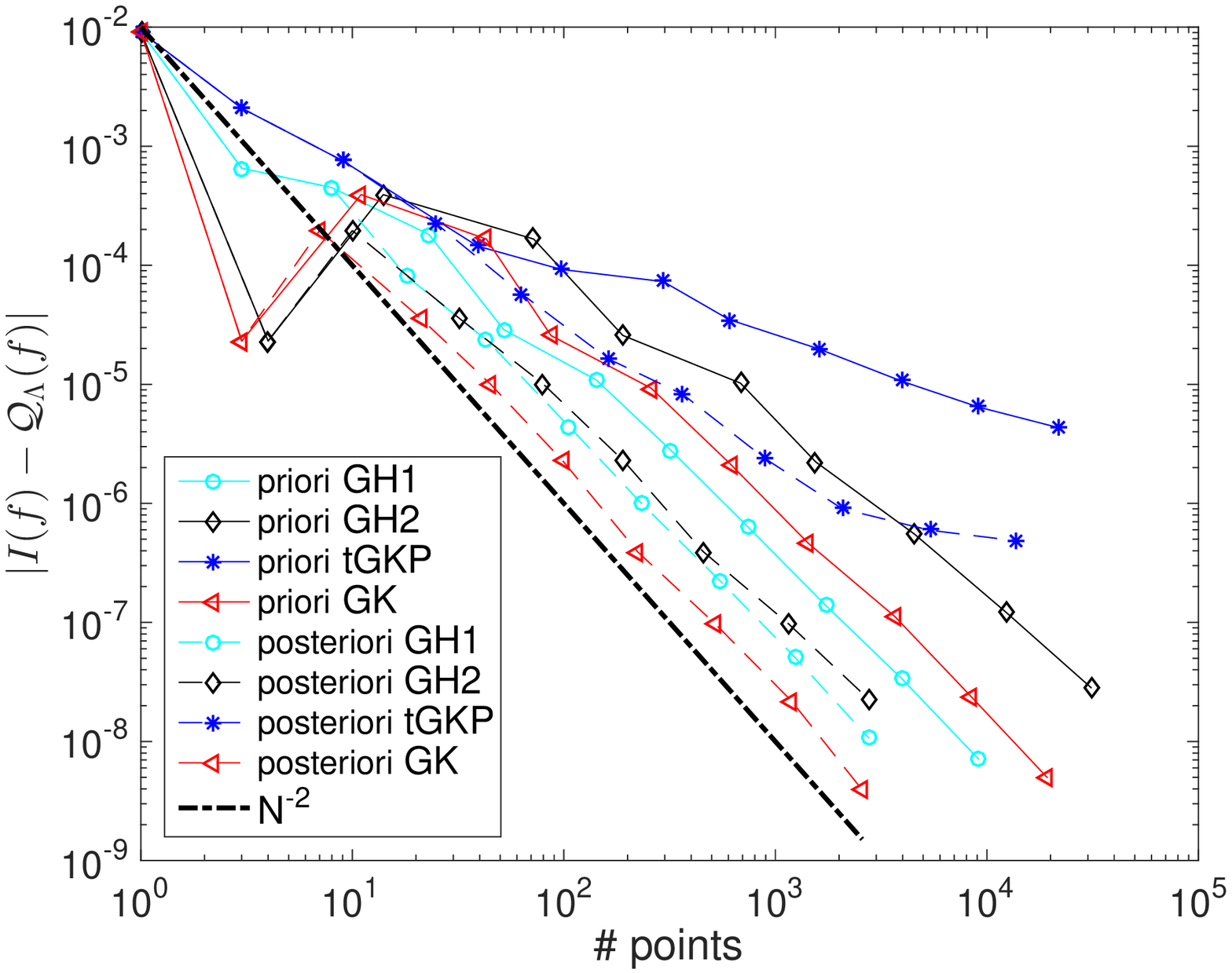}

\vspace*{0.2cm}

\includegraphics[scale=0.34]{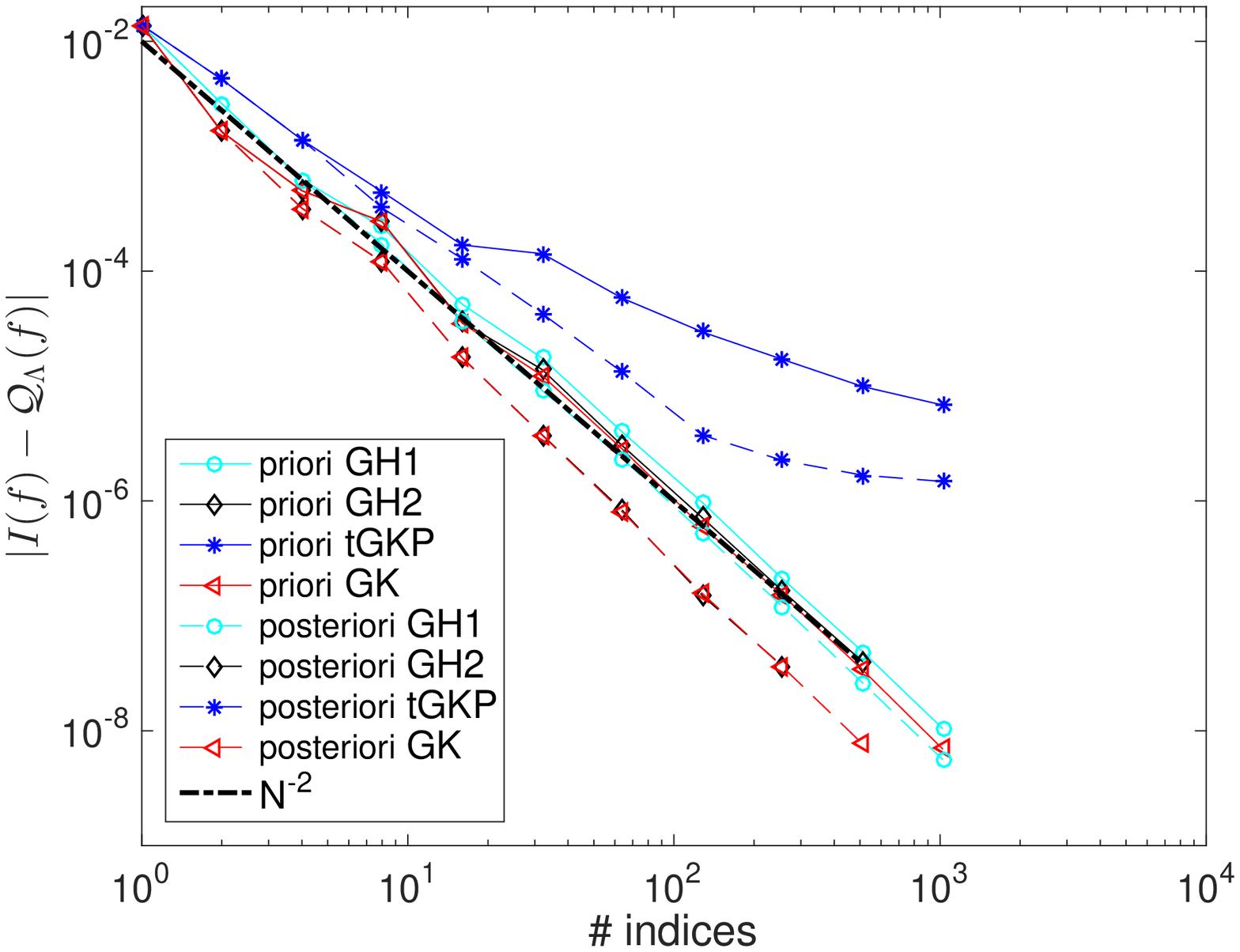}
\hspace*{0.2cm}
\includegraphics[scale=0.34]{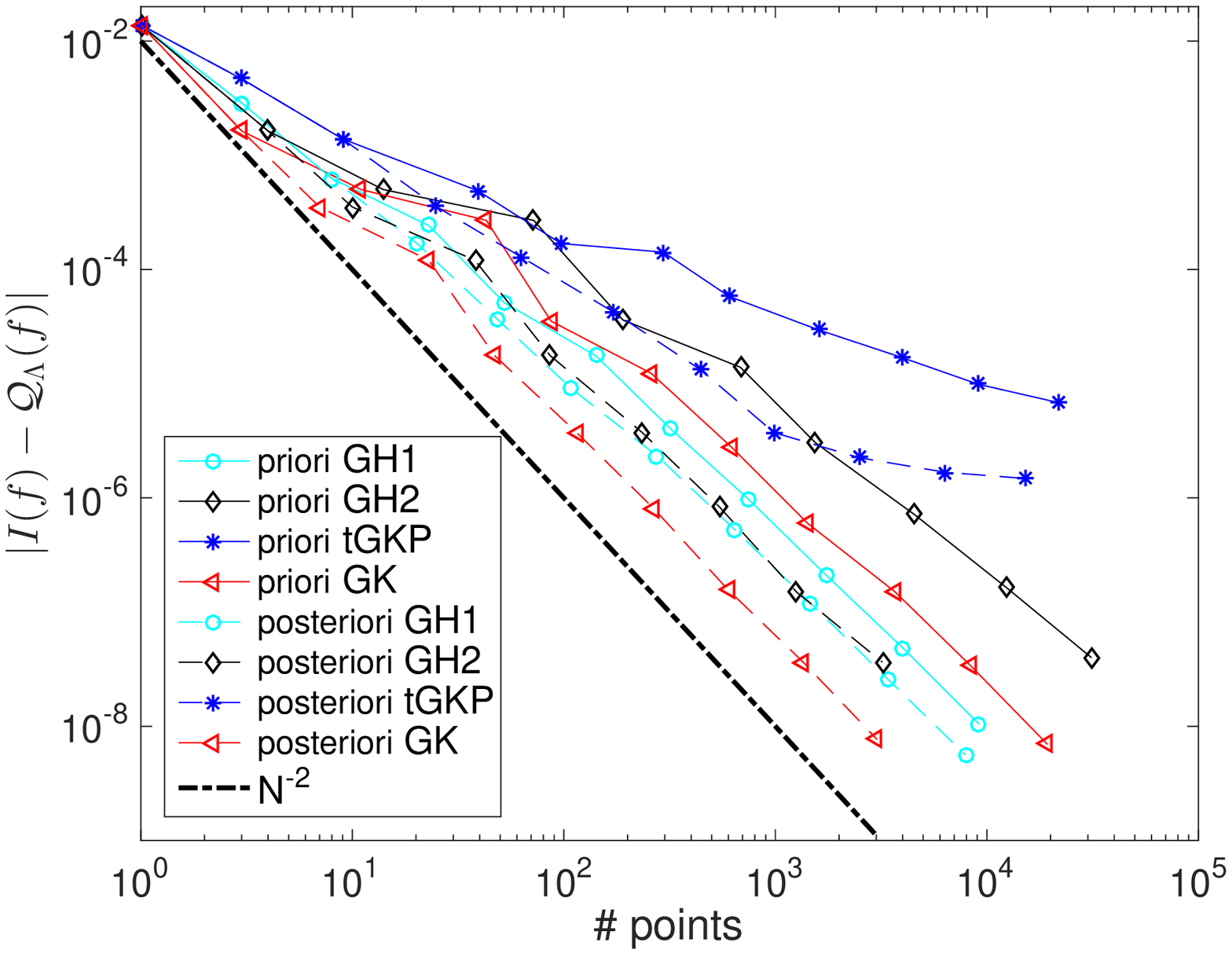}
\end{center}
\caption{Decay of quadrature errors $|I(f) - \cQ_{\Lambda}(f)|$ with respect to the number of indices (left) and the number of points (right) in $\Lambda$. Reported are for the different quadrature rules constructed by both the a-priori and the a-posteriori construction Algorithm \ref{alg:SparseQuad}. $\alpha = 2$. Top: $f_1$; bottom: $f_2$.}\label{fig:OldPrioriPosterioriPDE}
\end{figure}

Fig. \ref{fig:OldPrioriPosterioriPDE} displays the convergence of the quadrature errors of the two moments $\bbE[f_1]$ and $\bbE[f_2]$ with respect to the number of indices and points in the index set $\Lambda$, where we compute the error by 
\beq 
|I(f) - \cQ_{\Lambda}(f)| \approx |\cQ_{\bar{\Lambda}_{\text{max}}}^{\text{GK}}(f) - \cQ_{\Lambda}(f)|\;.
\eeq
Here $\cQ_{\bar{\Lambda}_{\text{max}}}^{\text{GK}}(f)$ is the approximation of $I(f)$ by the a-posteriori GK quadrature at the largest index set $\bar{\Lambda}_{\text{max}} = \Lambda_{\text{max}} \cup \cN(\Lambda_{\text{max}})$ with about $10^5$ quadrature points. GK quadrature is used since it is more accurate for this test example as shown in Fig. \ref{fig:OldPrioriPosterioriPDE}. 
Moreover, the number of activated dimensions in $\Lambda$, for which the maximum grid level is larger than $1$ in $\Lambda \cup \cN(\Lambda)$, is smaller than the number of the full dimensions for all quadrature rules, in particular smaller than the number of dimensions activated by the a-posteriori GK in $\bar{\Lambda}_{\text{max}}$,
see Fig. \ref{fig:levelPDE}, which indicates that the quadrature errors computed for the indices and the points in $\Lambda$ are unbiased and the convergence rate is dimension-independent. From the decaying of the quadrature errors, we can observe the dimension-independent convergence rate about $N^{-s}$ with $s = 2$ with respect to the number of both indices and points in $\Lambda$, for both quantities of interest $f_1$ and $f_2$. 
Again, GK quadrature turns out to be the most accurate and tGKP is the least with the same number of quadrature points.
The a-priori construction gives less accurate quadrature results compared to the a-posteriori construction, in particular for GH2, tGKP, and GK as explained in the last section. We remark that the same index set has been constructed for both $f_1$ and $f_2$ by the a-priori construction, while by the a-posteriori construction, the index sets for the two quantities are different. 
This can be illustrated by Fig. \ref{fig:levelPDE}, where the maximum level in each dimension is the same for $f_1$ and $f_2$ by the a-priori construction and different by the posteriori construction, see the comparison of GH1 and GK for the two quantities. 
Therefore, the same index set can be used for different quantities of interest (with the same $(\tau_j)_{j\geq 1}$) once constructed by the a-priori scheme. On the other hand, the posteriori scheme requires a complete reconstruction of the index set for each new quantity of interest.

\begin{figure}[!htb]
\begin{center}
\includegraphics[scale=0.34]{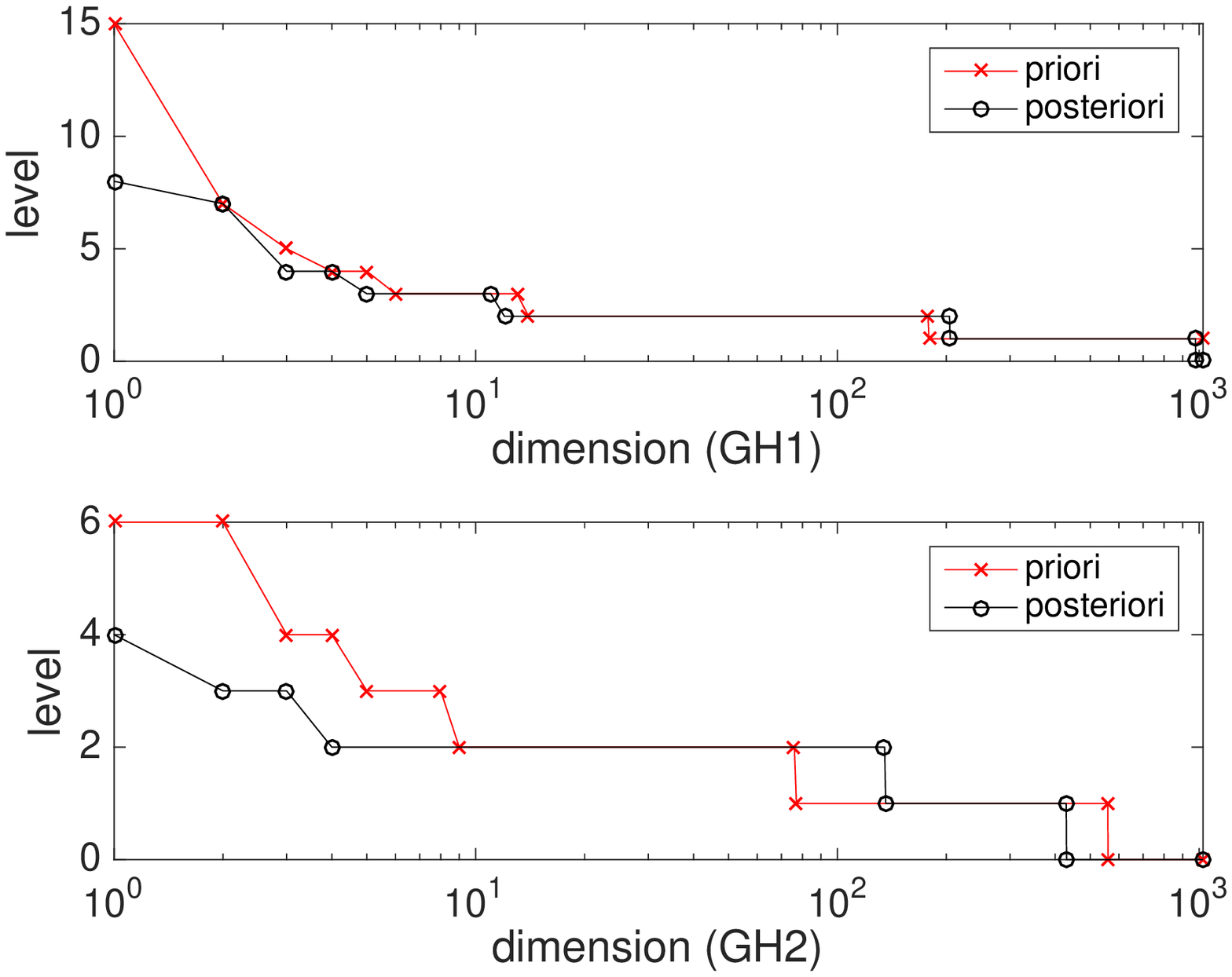}
\hspace*{0.5cm}
\includegraphics[scale=0.34]{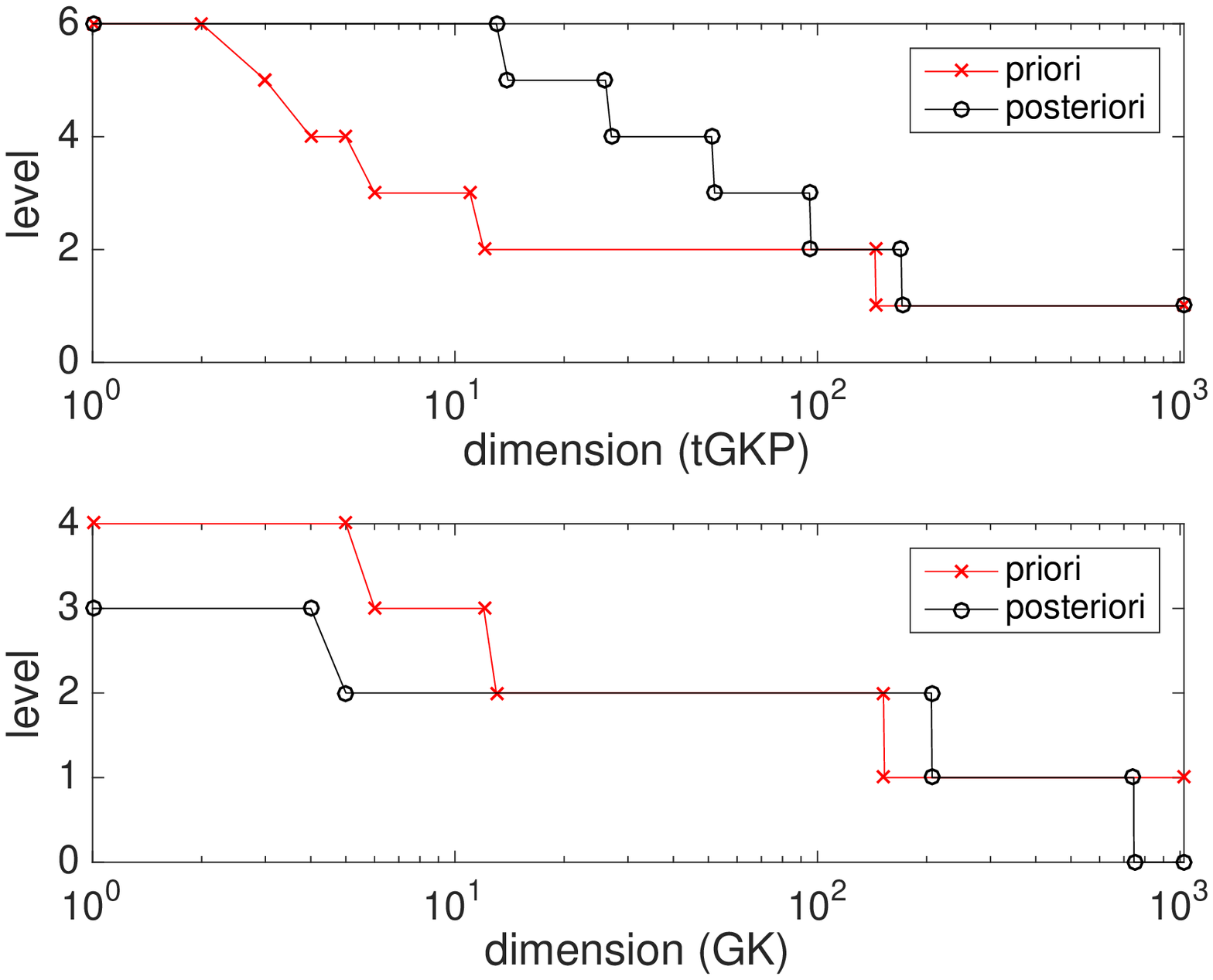}

\vspace*{0.2cm}

\includegraphics[scale=0.34]{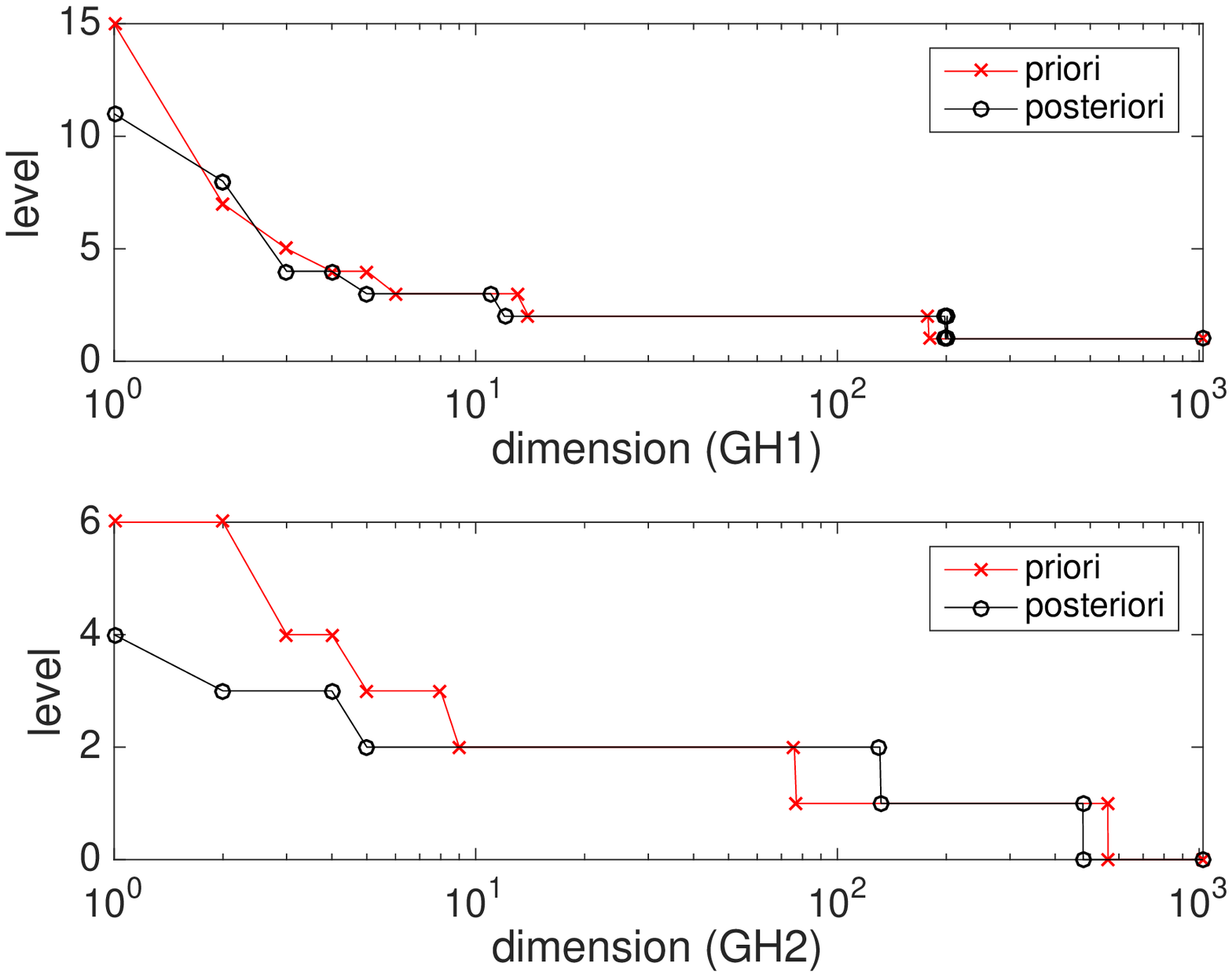}
\hspace*{0.5cm}
\includegraphics[scale=0.34]{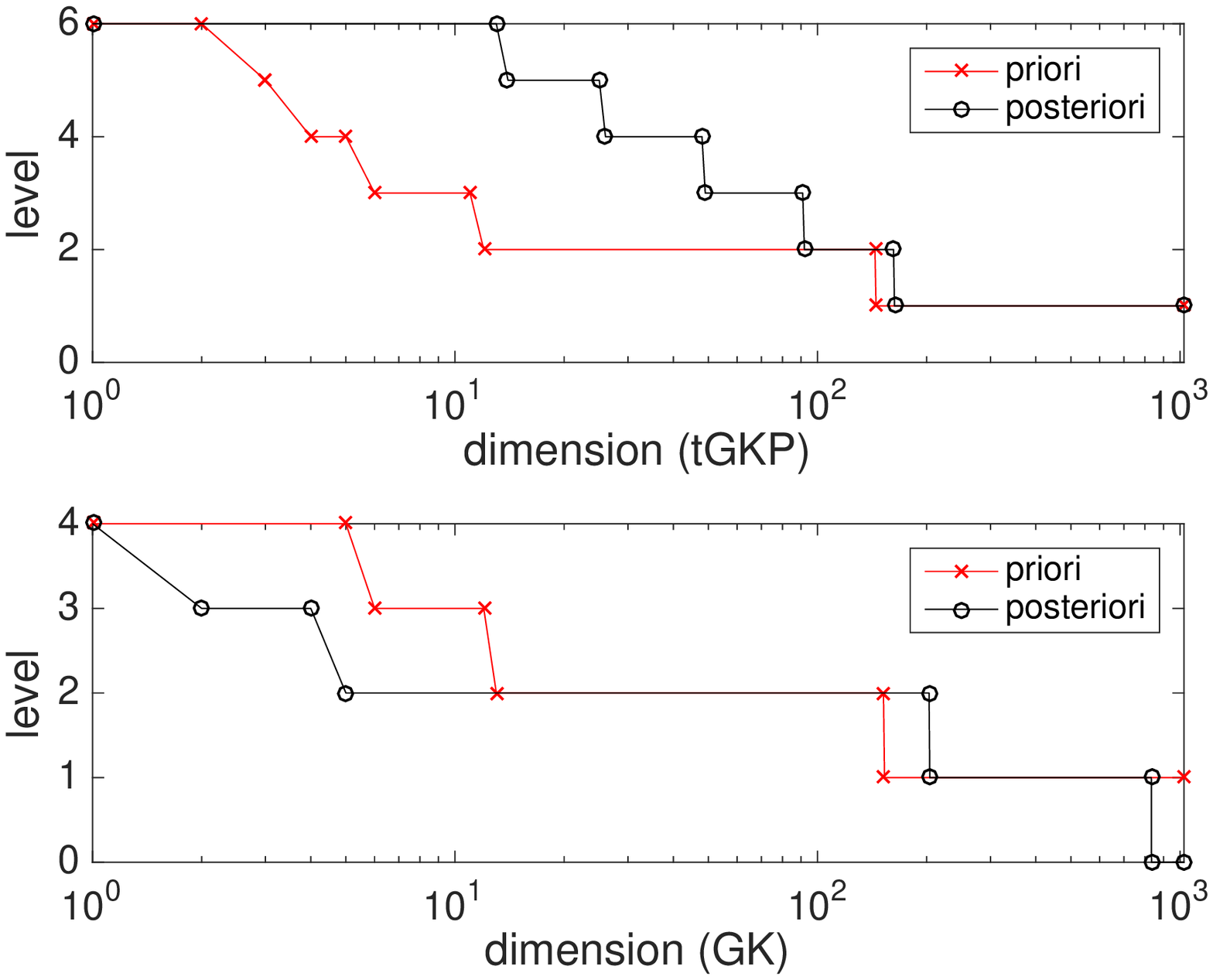}
\end{center}
\caption{Maximum level ($\max_{\bsnu \in \Lambda\cup \cN(\Lambda)} \nu_j$, $j = 1, \dots, 1023$) in each dimension constructed by the a-priori and the a-posteriori schemes for the four quadrature rules. $\alpha = 2$. Top: $f_1$; bottom: $f_2$.}\label{fig:levelPDE}
\end{figure}

\begin{figure}[!htb]
\begin{center}
\includegraphics[scale=0.34]{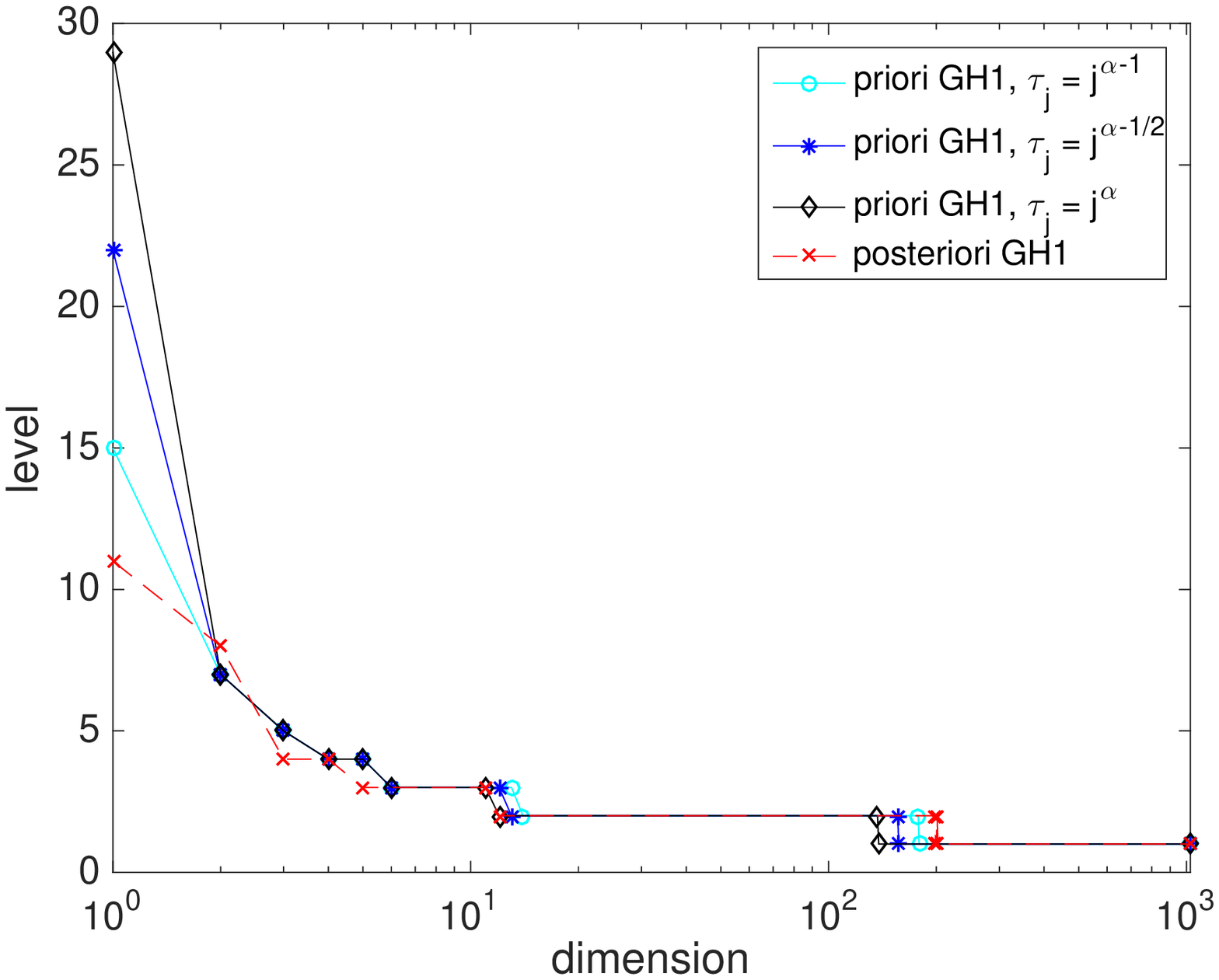}
\hspace*{0.2cm}
\includegraphics[scale=0.34]{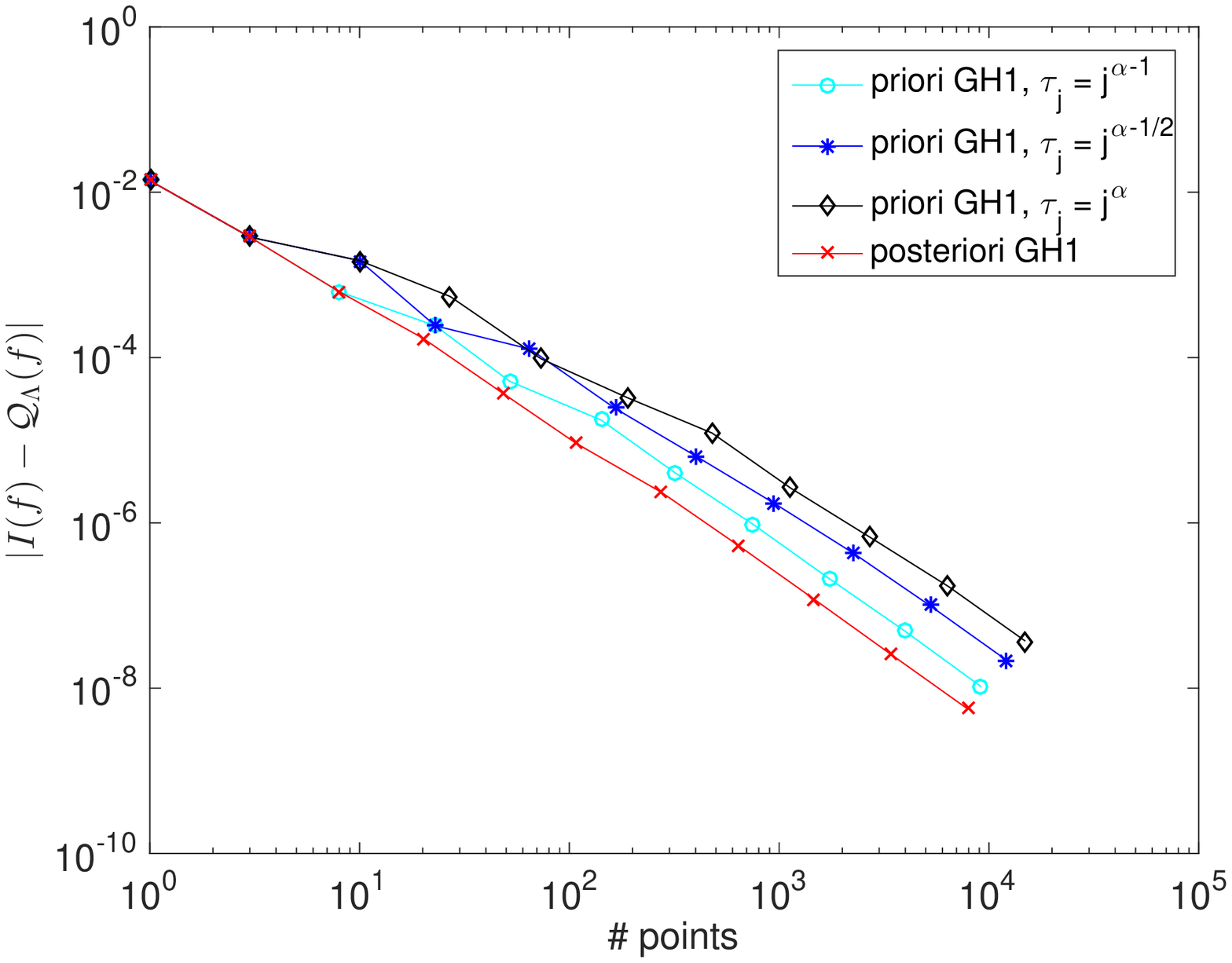}
\end{center}
\caption{Left: maximum level ($\max_{\bsnu \in \Lambda\cup \cN(\Lambda)} \nu_j$, $j = 1, \dots, 1023$) in each dimension constructed by the a-priori scheme with different $(\tau_j)_{j\geq 1}$ and the a-posteriori scheme, all using GH1. Right: the corresponding sparse quadrature errors. $\alpha = 2$.}\label{fig:CompareTau}
\end{figure}

Note that with $\tau_j = j^{\alpha - 1}$, i.e., $(\tau^{-1}_j)_{j\geq 1} \in \ell^q(\bbN)$ for $q > 1/(\alpha - 1)$, the numerical convergence about $N^{-s}$ with $s = 2$ is faster than the convergence of $N^{-s}$ with $s = 1/q - 1/2 < \alpha - 3/2 = 1/2$ according to Theorem \ref{thm:N-termConv}. However, as the choice $\tau_j = j^{\alpha - 1}$ might be only a sufficient condition for the Assumption \ref{ass:DeriBound}, so we may numerically relax it. Here we also test $\tau_j = j^{\alpha - 1/2}$ and $\tau_j = j^{\alpha}$. The maximum level in each dimension and the convergence of the quadrature errors are shown in Fig. \ref{fig:CompareTau} for the a-priori construction with GH1. We can see that the three choices of $\tau_j$ produce a very close convergence rate $N^{-s}$ with $s = 2$, though $\tau_j = j^{\alpha - 1}$ leads to more accurate quadrature than $\tau_j = j^{\alpha - 1/2}$ and $\tau_j = j^{\alpha}$. The maximum levels from the three choices are also the same except in a small number of dimensions. 

\begin{figure}[!htb]
\begin{center}
\includegraphics[scale=0.34]{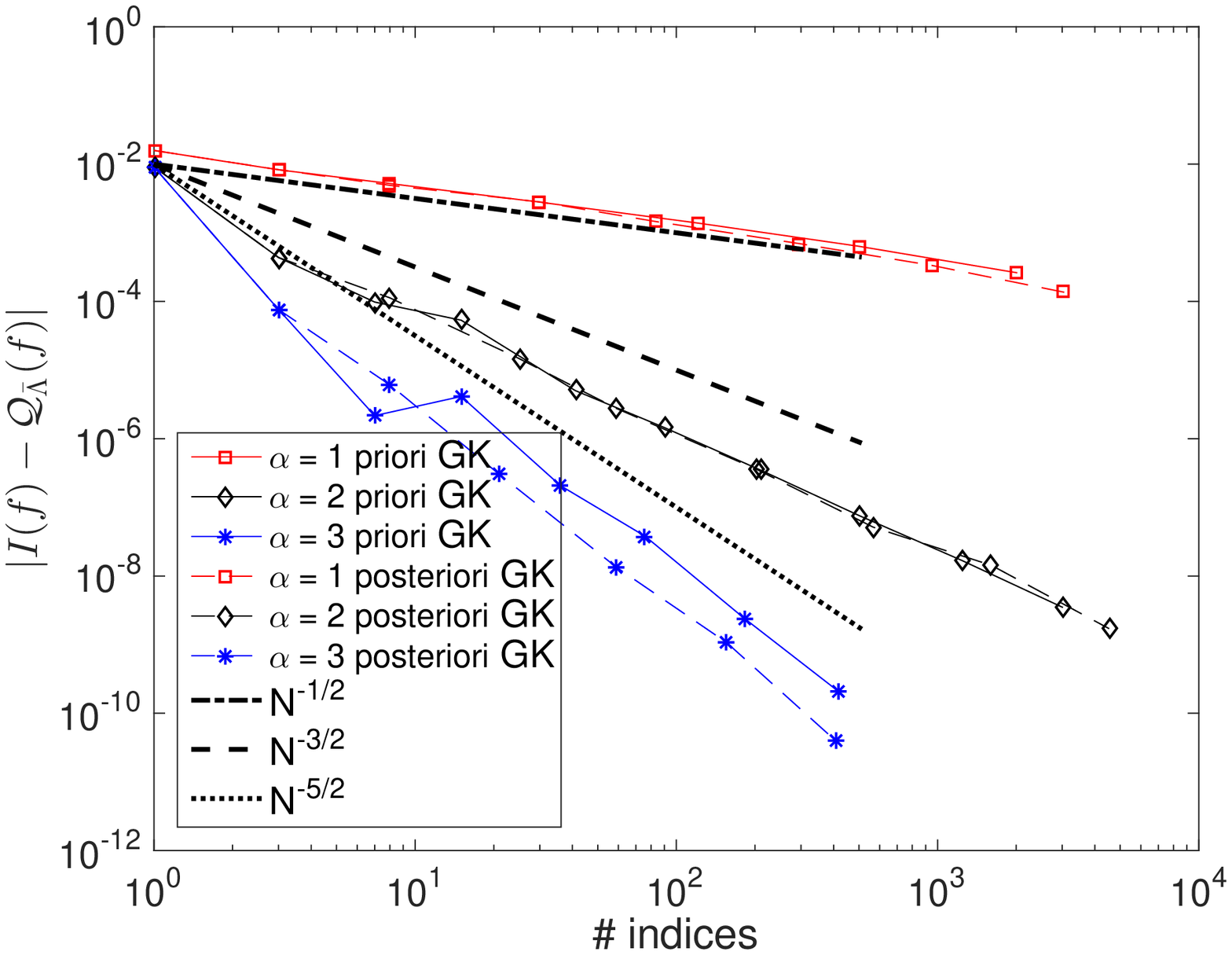}
\hspace*{0.2cm}
\includegraphics[scale=0.34]{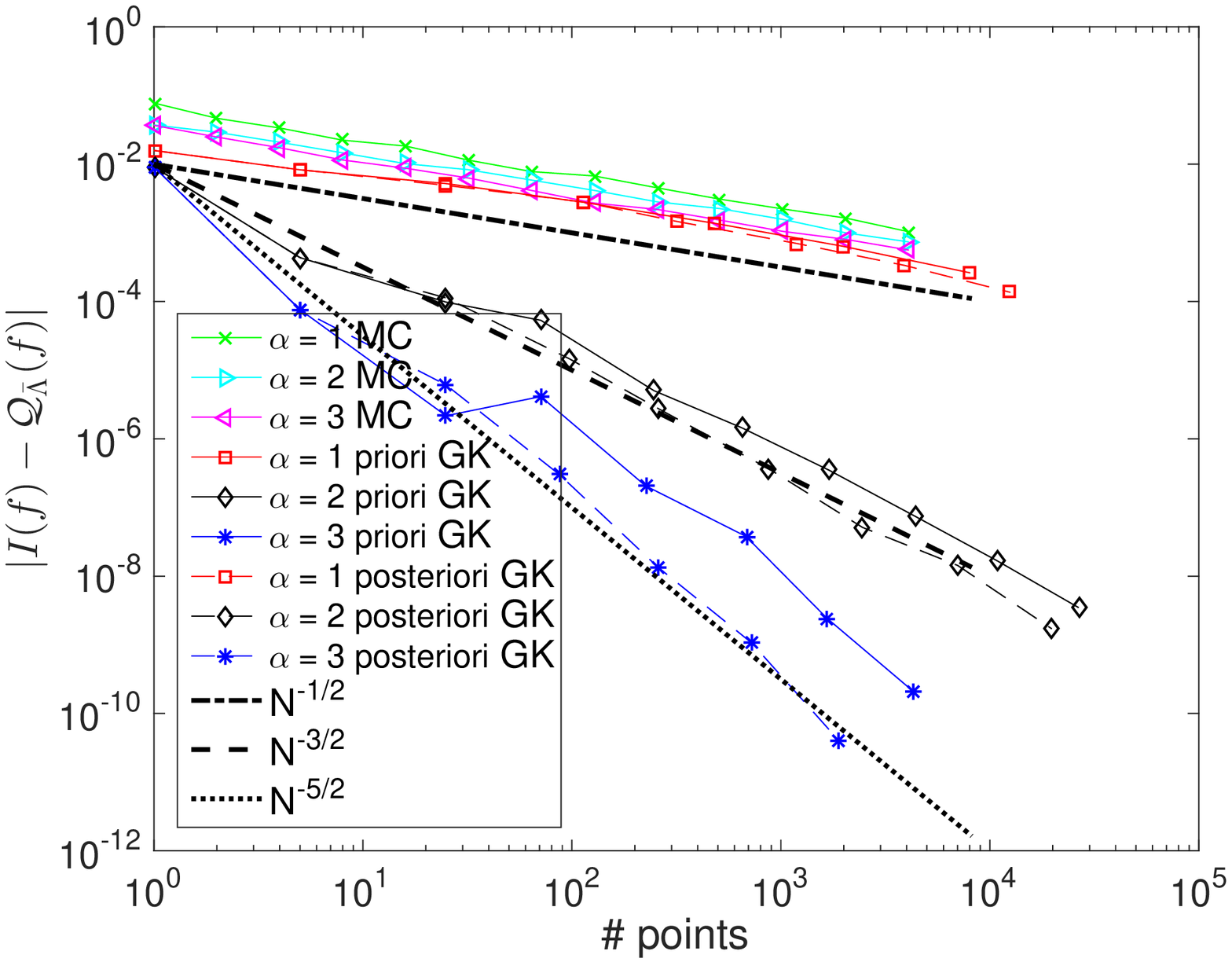}

\vspace*{0.2cm}

\includegraphics[scale=0.34]{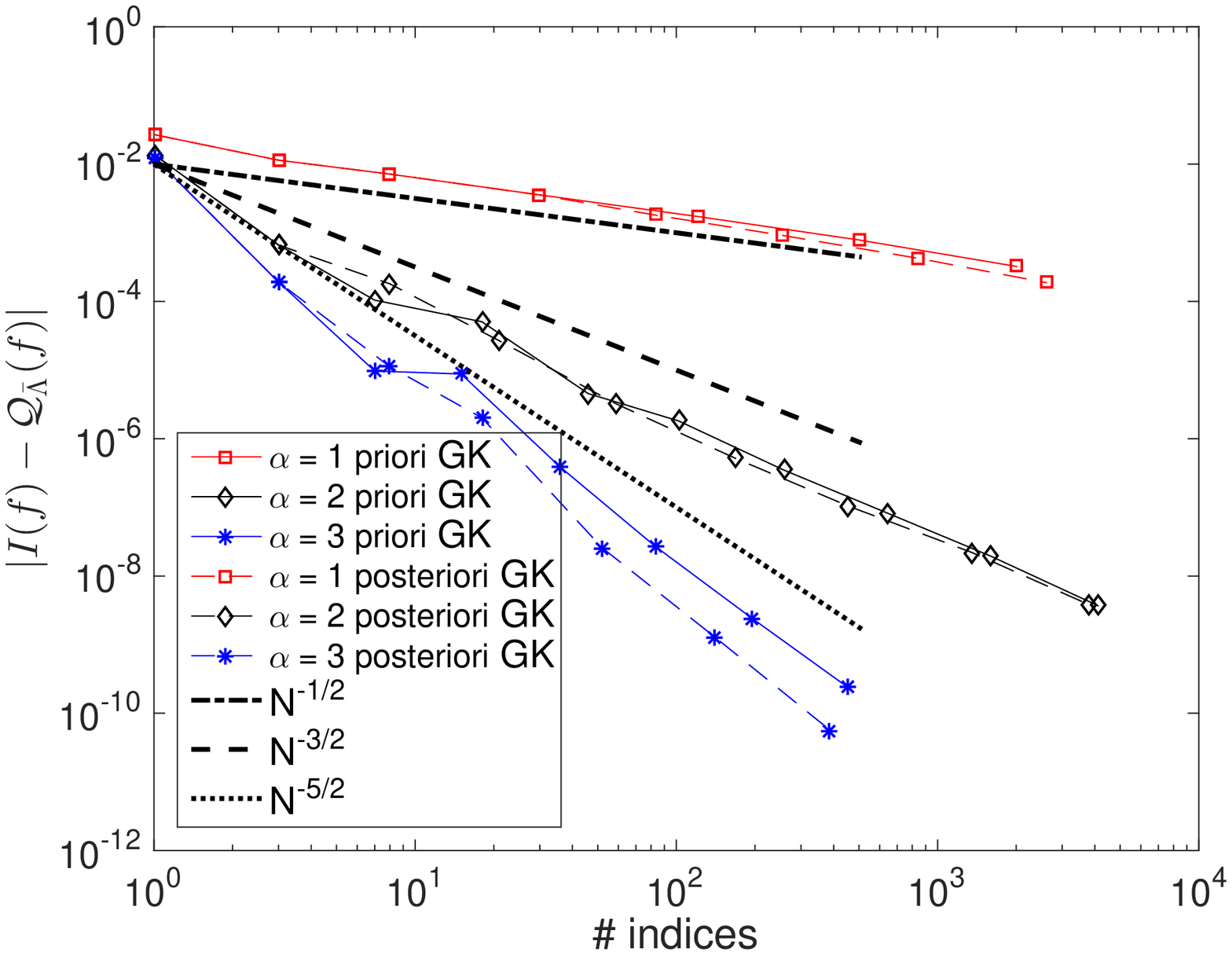}
\hspace*{0.2cm}
\includegraphics[scale=0.34]{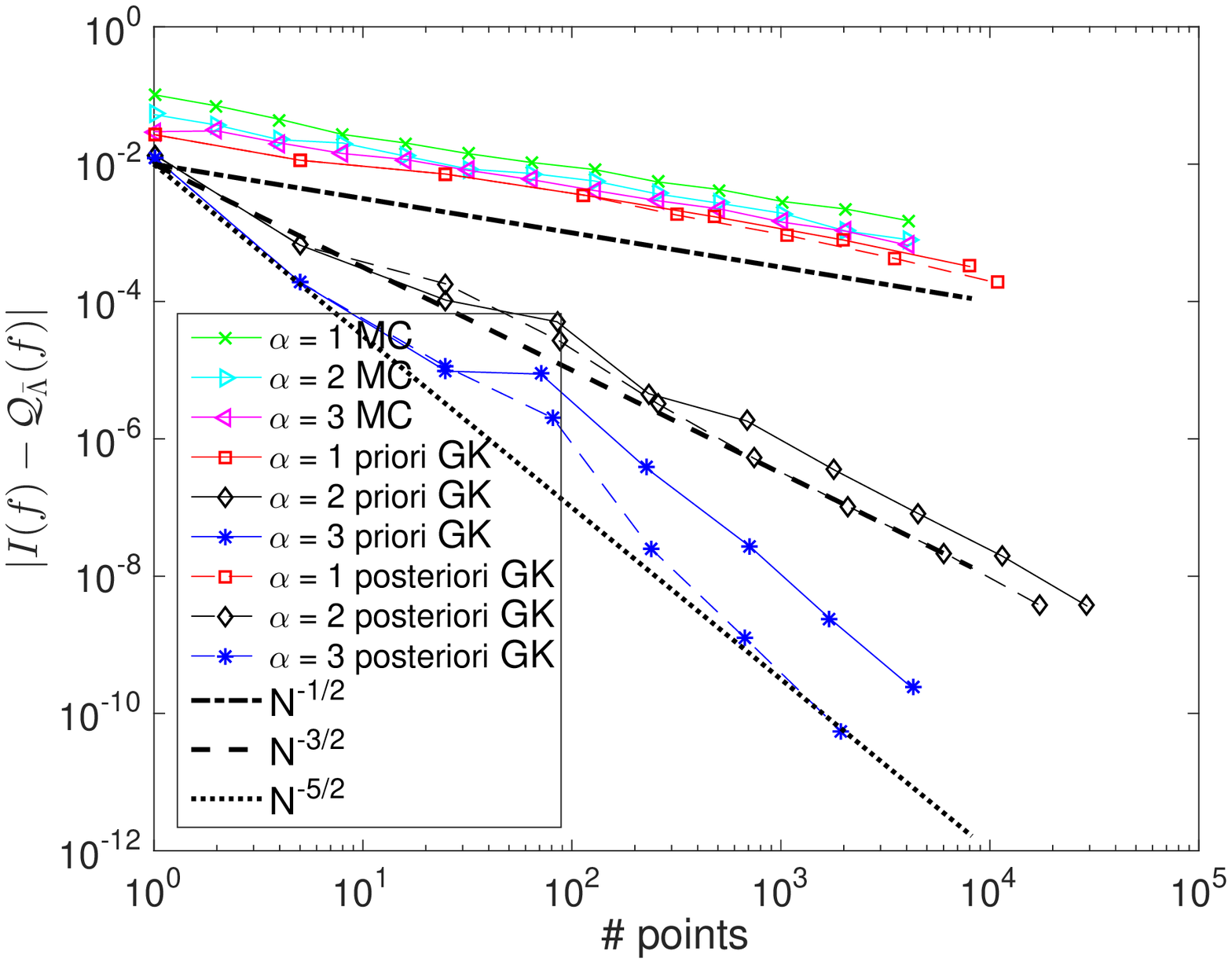}
\end{center}
\caption{Decay of quadrature errors $|I(f) - \cQ_{\bar{\Lambda}}(f)|$ with respect to the number of indices (left) and points (right) in $\bar{\Lambda} = \Lambda \cup \cN(\Lambda)$. Results are shown for both the a-priori and the a-posteriori construction schemes with the MC and the GK quadrature rules. Top: $f_1$; bottom: $f_2$.}\label{fig:AllPrioriPosterioriPDE}
\end{figure}

Finally, in Fig. \ref{fig:AllPrioriPosterioriPDE} we report the decaying of the sparse quadrature errors for both $\bbE[f_1]$ and $\bbE[f_2]$ with respect to both the number of indices and the number of points in the union set $\bar{\Lambda} = \Lambda \cup \cN(\Lambda)$, which correspond to the total computational cost. We use the most accurate GK quadrature rule and test $\alpha = 1, 2, 3$. The convergence rate about $N^{-s}$ with $s = \alpha - 1/2$ can be observed for all $\alpha$ and for both the a-priori construction and the a-posteriori construction, which indicates that the convergence rate only depends on the sparsity parameter $\alpha$, and is much higher than the Monte Carlo convergence rate $N^{-1/2}$ for $\alpha = 2, 3$. In the case $\alpha = 1$, the sparse quadrature errors converge with rate about $N^{-1/2}$ and is smaller than that of Monte Carlo quadrature errors, which are computed as the average of 100 trials.

\section{Conclusion}
\label{sec:Conclusion}
In this work, we analyzed the dimension-independent convergence property of an abstract sparse quadrature for high-dimensional integration with Gaussian measure under certain assumptions on the univariate quadrature rules and the regularity of the parametric function with respect to the parameters, which established the foundation of efficient algorithms to break the curse of dimensionality commonly faced by a class of high and infinite-dimensional integration problems. We presented both a-priori and a-posteriori construction schemes for numerical integration. Moreover, we investigated the a-priori and the a-posteriori construction schemes with four kinds of different univariate quadrature rules and studied their convergence properties through numerical experiments on a nonlinear parametric function and a nonlinear parametric PDE. The numerical results demonstrate that the convergence rates of the quadrature errors do not depend on the number of dimensions but only on some parameter related to the regularity of the parametric function. This conclusion holds not only for the convergence of the quadrature errors with respect to the number of the indices in the admissible index set as stated in the main theorem, but also for that with respect to the total number of quadrature points corresponding to the union of the admissible index set and its forward neighbor set, i.e., with respect to the total number of function evaluations or PDE solutions. The convergence of the sparse quadrature errors (with rate $N^{-s}$) is faster than the Monte Carlo quadrature errors (i.e., $s > 1/2$) in all the numerical examples with sufficiently large $\alpha$ (or small $q$) which indicates the regularity of the parametric function. 
The numerical convergence rates in the examples are larger than those of the theoretical prediction in the main theorem, which indicate that the latter may not be optimal. How to improve the theoretical convergence rate is worthy to investigate. Further work on the development and the application of the sparse quadrature in solving high-dimensional integration problems in different areas, such as Bayesian inverse problems \cite{chen2016hessian} and optimization under uncertainty, are interesting and promising. Moreover, comparison of the sparse quadrature with a type of quasi-Monte Carlo quadrature \cite{graham2015quasi, kuo2015multilevel} is interesting for high-dimensional integration with Gaussian measure.

\bibliographystyle{plain}
\bibliography{bibliography.bib}

\begin{thebibliography}{10}

\bibitem{abramowitz1966handbook}
M.~Abramowitz and I.A. Stegun.
\newblock Handbook of mathematical functions.
\newblock {\em Applied mathematics series}, 55:62, 1966.

\bibitem{babuvska2010stochastic}
I.~Babu{\v{s}}ka, F.~Nobile, and R.~Tempone.
\newblock A stochastic collocation method for elliptic partial differential
  equations with random input data.
\newblock {\em SIAM Review}, 52(3):317, 2010.

\bibitem{bachmayr2017sparse}
M.~Bachmayr, A.~Cohen, R.~DeVore, and G.~Migliorati.
\newblock Sparse polynomial approximation of parametric elliptic {PDEs}. part
  ii: lognormal coefficients.
\newblock {\em ESAIM: Mathematical Modelling and Numerical Analysis},
  51(1):341--363, 2017.

\bibitem{back2011stochastic}
J.~Beck, F.~Nobile, L.~Tamellini, and R.~Tempone.
\newblock {Stochastic spectral Galerkin and collocation methods for PDEs with
  random coefficients: A numerical comparison}.
\newblock In J.S. Hesthaven and E.M. R{\o}nquist, editors, {\em Spectral and
  High Order Methods for Partial Differential Equations}, pages 43--62.
  Springer-Verlag, Berlin, 2011.

\bibitem{beck2014quasi}
J.~Beck, F.~Nobile, L.~Tamellini, and R.~Tempone.
\newblock A quasi-optimal sparse grids procedure for groundwater flows.
\newblock In {\em Spectral and High Order Methods for Partial Differential
  Equations-ICOSAHOM 2012}, pages 1--16. Springer, 2014.

\bibitem{bungartz2004sparse}
H.J. Bungartz and M.~Griebel.
\newblock Sparse grids.
\newblock {\em Acta Numerica}, 13(1):147--269, 2004.

\bibitem{caflisch1998monte}
R.E. Caflisch.
\newblock Monte {C}arlo and quasi-{M}onte {C}arlo methods.
\newblock {\em Acta Numerica}, 1998:1--49, 1998.

\bibitem{charrier2012strong}
J.~Charrier.
\newblock Strong and weak error estimates for elliptic partial differential
  equations with random coefficients.
\newblock {\em SIAM Journal on numerical analysis}, 50(1):216--246, 2012.

\bibitem{chen2013weightedrbm}
P.~Chen and A.~Quarteroni.
\newblock Weighted reduced basis method for stochastic optimal control problems
  with elliptic {PDE} constraints.
\newblock {\em SIAM/ASA J. Uncertainty Quantification}, 2(1):364--396, 2014.

\bibitem{chen2015new}
P.~Chen and A.~Quarteroni.
\newblock A new algorithm for high-dimensional uncertainty quantification based
  on dimension-adaptive sparse grid approximation and reduced basis methods.
\newblock {\em Journal of Computational Physics}, 298:176--193, 2015.

\bibitem{chen2015sparse}
P.~Chen and Ch. Schwab.
\newblock Sparse-grid, reduced-basis {B}ayesian inversion.
\newblock {\em Computer Methods in Applied Mechanics and Engineering}, 297:84
  -- 115, 2015.

\bibitem{chen2016sparse}
P.~Chen and Ch. Schwab.
\newblock Sparse-grid, reduced-basis {B}ayesian inversion: Nonaffine-parametric
  nonlinear equations.
\newblock {\em Journal of Computational Physics}, 316:470--503, 2016.

\bibitem{chen2016hessian}
P.~Chen, U.~Villa, and O.~Ghattas.
\newblock Hessian-based sparse quadrature for high-dimensional {B}ayesian
  inverse problems.
\newblock {\em submitted}, 2017.

\bibitem{chkifa2014high}
A.~Chkifa, A.~Cohen, and Ch. Schwab.
\newblock High-dimensional adaptive sparse polynomial interpolation and
  applications to parametric pdes.
\newblock {\em Foundations of Computational Mathematics}, 14(4):601--633, 2014.

\bibitem{chkifa2015breaking}
A.~Chkifa, A.~Cohen, and Ch. Schwab.
\newblock Breaking the curse of dimensionality in sparse polynomial
  approximation of parametric {PDEs}.
\newblock {\em Journal de Math{\'e}matiques Pures et Appliqu{\'e}es},
  103(2):400--428, 2015.

\bibitem{cohen2010convergence}
A.~Cohen, R.~DeVore, and C.~Schwab.
\newblock Convergence rates of best {N-term Galerkin approximations for a class
  of elliptic sPDEs}.
\newblock {\em Foundations of Computational Mathematics}, 10(6):615--646, 2010.

\bibitem{cohen2011analytic}
A.~Cohen, R.~Devore, and C.~Schwab.
\newblock Analytic regularity and polynomial approximation of parametric and
  stochastic elliptic {PDEs}.
\newblock {\em Analysis and Applications}, 9(01):11--47, 2011.

\bibitem{ernst2014stochastic}
O.G. Ernst and B.~Sprungk.
\newblock Stochastic collocation for elliptic {PDEs} with random data: the
  lognormal case.
\newblock In {\em Sparse Grids and Applications-Munich 2012}, pages 29--53.
  Springer, 2014.

\bibitem{ernst2016convergence}
O.G. Ernst, B~Sprungk, and L.~Tamellini.
\newblock Convergence of sparse collocation for functions of countably many
  {G}aussian random variables - with application to lognormal elliptic
  diffusion problems.
\newblock {\em arXiv:1611.07239}, 2016.

\bibitem{genz1996fully}
A.~Genz and B.D. Keister.
\newblock Fully symmetric interpolatory rules for multiple integrals over
  infinite regions with {G}aussian weight.
\newblock {\em Journal of Computational and Applied Mathematics},
  71(2):299--309, 1996.

\bibitem{gerstner1998numerical}
T.~Gerstner and M.~Griebel.
\newblock Numerical integration using sparse grids.
\newblock {\em Numerical algorithms}, 18(3-4):209--232, 1998.

\bibitem{gerstner2003dimension}
T.~Gerstner and M.~Griebel.
\newblock Dimension--adaptive tensor--product quadrature.
\newblock {\em Computing}, 71(1):65--87, 2003.

\bibitem{ghanem2003stochastic}
R.G. Ghanem and P.D. Spanos.
\newblock {\em {Stochastic {F}inite {E}lements: a {S}pectral {A}pproach}}.
\newblock Dover Civil and Mechanical Engineering, Courier Dover Publications,
  Springer-Verlag, New York, 1991.

\bibitem{gil2007numerical}
A.~Gil, J.~Segura, and N.M. Temme.
\newblock {\em Numerical methods for special functions}.
\newblock SIAM, 2007.

\bibitem{gittelson2010stochastic}
C.J. Gittelson.
\newblock Stochastic galerkin discretization of the log-normal isotropic
  diffusion problem.
\newblock {\em Mathematical Models and Methods in Applied Sciences},
  20(02):237--263, 2010.

\bibitem{graham2015quasi}
I.G. Graham, F.Y. Kuo, J.A. Nichols, R.~Scheichl, Ch. Schwab, and I.H. Sloan.
\newblock {Quasi-Monte Carlo finite element methods for elliptic PDEs with
  lognormal random coefficients}.
\newblock {\em Numerische Mathematik}, pages 1--40, 2015.

\bibitem{griebel2010dimension}
M.~Griebel and M.~Holtz.
\newblock Dimension-wise integration of high-dimensional functions with
  applications to finance.
\newblock {\em Journal of Complexity}, 26(5):455--489, 2010.

\bibitem{hoang2014n}
V.H. Hoang and C.~Schwab.
\newblock {N-term Wiener chaos approximation rates for elliptic PDEs with
  lognormal gaussian random inputs}.
\newblock {\em Mathematical Models and Methods in Applied Sciences},
  24(04):797--826, 2014.

\bibitem{klimke2006uncertainty}
A.~Klimke.
\newblock {\em Uncertainty modeling using fuzzy arithmetic and sparse grids.
  Universit{\"a}t Stuttgart}.
\newblock PhD thesis, Universit{\"a}t Stuttgart, Germany, 2006.

\bibitem{kronrod1965nodes}
A.S. Kronrod.
\newblock {\em Nodes and weights of quadrature formulas: sixteen-place tables}.
\newblock Consultants Bureau, New York, 1965.

\bibitem{kuo2015multilevel}
F.Y. Kuo, R.~Scheichl, Ch. Schwab, I.H. Sloan, and E.~Ullmann.
\newblock Multilevel quasi-{Monte Carlo} methods for lognormal diffusion
  problems.
\newblock {\em arXiv preprint arXiv:1507.01090}, 2015.

\bibitem{le2010introduction}
O.P. Le~Ma{\^\i}tre and O.M. Knio.
\newblock {\em Introduction: Uncertainty Quantification and Propagation}.
\newblock Springer, 2010.

\bibitem{li2007probabilistic}
H.~Li and D.~Zhang.
\newblock Probabilistic collocation method for flow in porous media:
  Comparisons with other stochastic methods.
\newblock {\em Water Resources Research}, 43(9), 2007.

\bibitem{lin2009efficient}
G.~Lin and A.M. Tartakovsky.
\newblock An efficient, high-order probabilistic collocation method on sparse
  grids for three-dimensional flow and solute transport in randomly
  heterogeneous porous media.
\newblock {\em Advances in Water Resources}, 32(5):712--722, 2009.

\bibitem{ma2009adaptive}
X.~Ma and N.~Zabaras.
\newblock {An adaptive hierarchical sparse grid collocation algorithm for the
  solution of stochastic differential equations}.
\newblock {\em Journal of Computational Physics}, 228(8):3084--3113, 2009.

\bibitem{nevai1980mean}
P.G. Nevai.
\newblock Mean convergence of {L}agrange interpolation, {II}.
\newblock {\em Journal of Approximation Theory}, 30(4):263--276, 1980.

\bibitem{nobile2014convergence}
F.~Nobile, L.~Tamellini, and R.~Tempone.
\newblock Convergence of quasi-optimal sparse grid approximation of
  {H}ilbert-valued functions: application to random elliptic {PDEs}.
\newblock {\em Numerische Mathematik}, 2015.

\bibitem{nobile2016adaptive}
F.~Nobile, L.~Tamellini, F.~Tesei, and R.~Tempone.
\newblock An adaptive sparse grid algorithm for elliptic {PDEs} with lognormal
  diffusion coefficient.
\newblock In {\em Sparse Grids and Applications-Stuttgart 2014}, pages
  191--220. Springer, 2016.

\bibitem{nobile2008anisotropic}
F.~Nobile, R.~Tempone, and C.G. Webster.
\newblock An anisotropic sparse grid stochastic collocation method for partial
  differential equations with random input data.
\newblock {\em SIAM Journal on Numerical Analysis}, 46(5):2411--2442, 2008.

\bibitem{nobile2008sparse}
F.~Nobile, R.~Tempone, and C.G. Webster.
\newblock A sparse grid stochastic collocation method for partial differential
  equations with random input data.
\newblock {\em SIAM Journal on Numerical Analysis}, 46(5):2309--2345, 2008.

\bibitem{patterson1968optimum}
T.N.L. Patterson.
\newblock The optimum addition of points to quadrature formulae.
\newblock {\em Mathematics of Computation}, 22(104):847--856, 1968.

\bibitem{schillings2011efficient}
C.~Schillings, S.~Schmidt, and V.~Schulz.
\newblock Efficient shape optimization for certain and uncertain aerodynamic
  design.
\newblock {\em Computers \& Fluids}, 46(1):78--87, 2011.

\bibitem{Schillings2013}
C.~Schillings and Ch. Schwab.
\newblock Sparse, adaptive {S}molyak quadratures for {B}ayesian inverse
  problems.
\newblock {\em Inverse Problems}, 29(6), 2013.

\bibitem{Schillings2014}
C.~Schillings and Ch. Schwab.
\newblock Sparsity in {B}ayesian inversion of parametric operator equations.
\newblock {\em Inverse Problems}, 30(6), 2014.

\bibitem{Schwab2006Karhunen}
Ch. Schwab and R.~A. Todor.
\newblock Karhunen--{L}o\`eve approximation of random fields by generalized
  fast multipole methods.
\newblock {\em Journal of Computational Physics}, 217(1):100--122, 2006.

\bibitem{smith2013uncertainty}
R.C. Smith.
\newblock {\em Uncertainty quantification: theory, implementation, and
  applications}, volume~12.
\newblock SIAM, 2013.

\bibitem{smolyak1963quadrature}
S.A. Smolyak.
\newblock Quadrature and interpolation formulas for tensor products of certain
  classes of functions.
\newblock In {\em Doklady Akademii Nauk SSSR}, volume~4, pages 240--243, 1963.

\bibitem{szeg1939orthogonal}
Gabor Szeg\"o.
\newblock {\em Orthogonal polynomials}, volume~23.
\newblock American Mathematical Soc., 1939.

\bibitem{xiu2005high}
D.~Xiu and J.S. Hesthaven.
\newblock {High-order collocation methods for differential equations with
  random inputs}.
\newblock {\em SIAM Journal on Scientific Computing}, 27(3):1118--1139, 2005.

\bibitem{xiu2010numerical}
Dongbin Xiu.
\newblock {\em Numerical methods for stochastic computations: a spectral method
  approach}.
\newblock Princeton University Press, 2010.

\bibitem{ZS17_723}
J.~Zech and Ch. Schwab.
\newblock Convergence rates of high dimensional {S}molyak quadrature.
\newblock Technical Report 2017-27, Seminar for Applied Mathematics, ETH
  Z{\"u}rich, Switzerland, 2017.

\end{thebibliography}

\end{document}